\documentclass[11pt]{amsart}
\usepackage{pdfpages,graphicx}
\usepackage{enumitem}
\usepackage{latexsym,amsmath,amsfonts,amsthm,amssymb,color}
\usepackage{fullpage}
\usepackage{tikz} 
\usepackage{bigints} 
\usepackage{hyperref}
\usepackage{bm}

\colorlet{symbols}{black} 
\tikzset{
	dot/.style={circle,fill=symbols,draw=symbols,inner sep=0pt,minimum size=0.4pt},
	basic/.style={draw=symbols},
	>=stealth,
	}
\definecolor{darkgreen}{rgb}{0.0, 0.2, 0.13}
\newcommand{\red}[1]{\begingroup\color{red} #1\endgroup}

\newcommand{\Null}[1]{}

\newtheorem{theorem}{Theorem}
\newtheorem{lemma}[theorem]{Lemma}
\newtheorem{proposition}[theorem]{Proposition}
\newtheorem{cor}[theorem]{Corollary}
\newtheorem{definition}[theorem]{Definition}
\newtheorem{remark}[theorem]{Remark}
\newtheorem{conjecture}[theorem]{Conjecture}

\renewcommand{\L}{{\mathcal L}}
\newcommand{\R}{{\mathbb R }}

\newcommand{\C}{{\mathbb C }}
\newcommand{\Z}{{\mathbb Z }}

\renewcommand{\H}{{\mathbb H }}
\renewcommand{\Z}{{\mathbb Z}}

\newcommand{\Phase}{\mathbf S} 
\newcommand{\Dress}{\mathbf B_+^N}
\newcommand{\Dresso}{\mathbf B^N}
\newcommand{\Undress}{\mathbf B_-^N}

\newcommand{\M}{\mathbf M}
\newcommand{\V}{\mathbf V}
\newcommand{\Backlund}{B\"acklund }

\newcommand{\balpha}{\bm{\alpha}}
\newcommand{\bbeta}{\bm{\beta}}
\newcommand{\bpsi}{\bm{\psi}}
\newcommand{\bzeta}{\bm{\zeta}}

\newcommand{\bfs}{{\mathbf s}}
\newcommand{\bfz}{{\mathbf z}}
\newcommand{\bfu}{{\mathbf u}}

\newcommand{\bfk}{{\bm{\kappa}}}

\newcommand{\tpsi}{{\tilde \psi}}
\newcommand{\bfH}{{\mathbf H}}

\DeclareMathOperator{\real}{\mathrm{Re}}
\DeclareMathOperator{\im}{\mathrm{Im}}
\DeclareMathOperator{\dist}{\mathrm{dist}} 
 
\DeclareMathOperator{\sech}{\mathrm{sech}}

\DeclareMathOperator{\tr}{\mathrm{Tr}} 
\DeclareMathOperator{\diag}{\mathrm{diag}} 
\DeclareMathOperator{\Ran}{\mathrm{Range}} 
\DeclareMathOperator{\sgn}{\mathrm{sgn}}

\numberwithin{equation}{section}
\numberwithin{theorem}{section}

\begin{document}
\title{Multisolitons for the cubic NLS in 1-d and their stability}

\author{Herbert Koch}
\address{Mathematisches Institut   \\ Universit\"at Bonn }
\email{koch@math.uni-bonn.de}

\author{ Daniel Tataru}
\address {Department of Mathematics \\
  University of California, Berkeley} 
\email{tataru@math.berkeley.edu}

\begin{abstract}
  For both the cubic Nonlinear Schr\"odinger Equation (NLS) as well as the
  modified Korteweg-de Vries (mKdV) equation in one space dimension we consider  the set $\M_N$ of pure $N$-soliton states, and their associated multisoliton solutions.
We prove that (i)  the set $\M_N$ is a uniformly smooth manifold, and (ii) the $\M_N$ states are uniformly  stable in  $H^s$,  for  each $s>-\frac12$.

One main tool in our analysis is an iterated \Backlund transform, which allows us to nonlinearly add a multisoliton to an existing soliton free state (the soliton addition map)  or alternatively to remove a multisoliton from a multisoliton state (the soliton removal map). The properties and the regularity of these maps are extensively studied. 
\end{abstract}

\maketitle

\setcounter{tocdepth}{1}
\tableofcontents

\section{Introduction}
In this article we consider the focusing cubic Nonlinear Schr\"odinger  equation (NLS) 
\begin{equation}
i u_t + u_{xx} + 2 u |u|^2 = 0, \qquad u(0) = u_0,
\label{nls}\end{equation}
and the complex focusing modified Korteweg-de Vries equation (mKdV) 
\begin{equation}
u_t + u_{xxx} + 6|u|^2u_x  = 0, \qquad u(0) = u_0,
\label{mkdv}\end{equation}
on the real line, with real or complex solutions in one space dimension. 

We are considering these equations together since they are commuting 
Hamiltonian flows. They are also completely integrable with a common Lax
operator, and they are the first two nontrivial flows of the NLS-hierarchy of a countable number of commuting flows. 
Both admit soliton solutions, and the pure soliton states are common for the  two equations. These can be obtained from the state
\begin{equation}\label{soliton}
Q_0 = 2 \sech 2x
\end{equation}
via scaling and Galilean symmetry (which we will call \emph{spectral parameters}) and translations
and phase shifts  (which we will call the \emph{scattering parameters}).
The set of all single soliton states can be thought of as a four dimensional real manifold, or alternately as a two dimensional complex
manifold, which we denote by $\M_1$. The pure soliton states are known to be orbitally stable in various topologies such as $L^2$ and $H^1$.

Our aim in this paper is to study multisoliton states and solutions,
which can be thought of as the nonlinear superposition of several
single soliton solutions.  
  Precisely, for $N \geq 1$ we consider the
family of pure $N$-soliton states, which we denote by $\M_N$, and we
investigate its geometry as well as its stability with respect to both
the mKdV flow, the NLS flow as well as all higher order flows. 
Some of this analysis has been carried out before by several authors, under the assumption that the spectral
parameters of the $N$ component solitons
(or equivalently, the eigenvalues of the associated Lax operator) are separated. Instead, our
emphasis will be on what happens when the spectral parameters are
close, including the higher multiplicity case.  Our two main results can be summarized as follows:

\begin{description}
\item[I. Regularity] The family $\M_N$ of $N$-solitons is a uniformly smooth, symplectic $4N$ dimensional submanifold
of $H^s$ for all $s > -\frac12$.

\item[II. Stability] The family $\M_N$ is uniformly 
  stable with respect to both the NLS and the mKdV flows in $H^s$ for
  all $s > -\frac12$, in the sense that for any initial data $u_0$ with distance $\varepsilon$ to $\M_N$ there is a pure $N$ soliton solution $v$ so that $\Vert u(t) - v(t) \Vert_{H^s} = O(\varepsilon) $. 
\end{description}

The main tool in our analysis is the \Backlund transform,
which has been studied before in a context where the spectral parameters are separated.  Instead, here we study
the \Backlund transform when the spectral parameters are close are equal, and the results above
can be viewed to a certain extent as a consequence of this analysis. Our study  of the 
\Backlund transform is contained in Sections~\ref{s:backlund}, \ref{s:dress}, \ref{s:extended}, followed by the main results 
 in Section~\ref{s:dress+}.

The aim of the rest of the introduction is to provide sufficient
background in order to allow us to state more precise formulations for
both results above.

\subsection{Symmetries  and conservation laws}
Both the NLS equation \eqref{nls} and the mKdV equation \eqref{mkdv} are invariant with respect to translations in space and time and  with respect to phase shifts, i.e. multiplication by a complex number of modulus $1$.

Moreover the NLS equation is invariant 
with respect to  scaling
\[
u(x,t) \to \lambda u(\lambda x,\lambda^2 t),
\]
whereas  the mKdV  equation is invariant with respect to 
\[
u(x,t) \to \lambda u(\lambda x,\lambda^3 t). 
\]
The initial data for the two problems scales in the same way,
\begin{equation} \label{scaling}
u_0(x) \to \lambda u_0(\lambda x),
\end{equation}
 and so does the Sobolev space $\dot H^{-\frac12}$, which one may view as the
critical Sobolev space. 
 
The NLS and the mKdV flows  are in effect part of an hierarchy of  an infinite number of  commuting 
Hamiltonian flows with respect to the symplectic form
\[
\omega(u,v) = 2\im \int u \bar v \ dx,
\]
hence the Hamiltonian equations of the Hamiltonian $H$ are
\[ \frac{d}{ds} H(u+sv)\Big|_{s=0} = i  \int \dot u v - \overline{\dot u} v dx. \]
We consider Hamiltonians which are integrals over densities which we write as
sums over products of $u$ and $\bar u$ and their derivatives, which are invariant under exchanging $u$ and $\bar u$. The Hamiltonian equations are
\begin{equation}   \dot u = \frac1i \frac{\delta H(u)}{\delta \bar u }. \end{equation}
Each of these Hamiltonians of the hierarchy yields joint conservation laws for all of
these flows. The first several energies are as follows:
\begin{equation} \label{eq:energies} 
\begin{split} 
H_0=&    \int |u|^2 \, dx, \\
 H_1=& \  \frac1i \int u  \partial_x
              \bar u \, dx ,
\\ H_2=&  \int
              |u_x|^2 - |u|^4 dx, 
\\H_3 = & \  \frac1i \int \ u_x \partial_x \overline{ u}_x     
- 3  |u|^2 u \partial_x  \bar u\, dx ,
\\ H_4 =&  \int  |u_{xx}|^2 - ||u|^2_x|^2 - \frac32 |(u^2)_x|^2 + 2 |u|^6\, dx.
\end{split}
\end{equation}
The even ones are even with respect to complex conjugation and have a
positive definite principal part, and we will refer to them as
energies. The odd ones are odd under the replacement of $u$ by its
complex conjugate, and we will refer to them as momenta.
With respect
to the symplectic form above, With respect
to the symplectic form above, these commuting Hamiltonians generate
flows as follows: $H_0$ generates the phase shifts $e^{-it} u_0$, $H_1$ generates
the group of translations $u_0(x+t)$, 
$H_2$ the NLS flow, $H_3$ the mKdV flow,
etc. We denote the respective flows by $\Phi_n$, as operators acting 
on suitable Sobolev spaces. Because of the commuting property, one can also easily consider combinations of these flows.  To denote these
combined flows   in a compact fashion, for an arbitrary real polynomial 
\[
P(z) = \sum_{j=0}^n \beta_j z^j,
\]
we define
\begin{equation}\label{flows}
\Phi(2^{-1}P) = \Phi_0(\beta_0) \circ \Phi_1(2^{-1} \beta_1) \cdots \circ \Phi_n(2^{-n}\beta_n).
\end{equation}
Here the factors of $2$ arise as a consequence of a mismatch between the standard notations for the Fourier transform and for the scattering transform.

The integrable structure manifests itself in the existence of Lax pairs with the Lax operator 
\[ 
\mathcal{L}(u) = i \left( \begin{matrix} \partial_x & -u \\ -\bar u & -\partial_x  \end{matrix}\right) 
\] 
which we will discuss below.  The Lax operator for NLS was first introduced by Zakharov-Shabat~\cite{MR0406174}, following the circle of ideas initiated by Kruskal-Gardner-Green-Miura~\cite{kruskal1967method} and Lax~\cite{MR235310}.
Together with the associated scattering transform, it was studied later by many authors, including \cite{MR0450815}, \cite{MR928046}. The above flows define an evolution of the Lax operator and associated similarity transforms between those Lax operators.  

Here we note that if $u=Q_0$, the pure soliton in \eqref{soliton}, then 
\[ 
\mathcal{L}(Q_0)  \left(\begin{matrix}  \sech(2x) e^{-x} \\ - \sech(2x) e^x    \end{matrix} \right) 
= i \left( \begin{matrix} \sech(2x) e^{-x} \\ -\sech(2x) e^x      \end{matrix} \right), \qquad    \mathcal{L}(Q_0)  \left(\begin{matrix}  \sech(2x) e^{x} \\  \sech(2x) e^{-x}    \end{matrix} \right) 
= - i \left( \begin{matrix} \sech(2x) e^{x} \\ \sech(2x) e^{-x}      \end{matrix} \right),
\] 
and hence $z=\pm i$ are eigenvalues of $\mathcal{L}(Q_0) $.
The full spectrum of $L(Q_0)$ consists of these two eigenvalues together 
with the real axis, which represents its continuous spectrum.

All equations of the hierarchy are well-posed for initial data in the Schwartz space, see Zhou \cite{MR1617249}:  A modified inverse scattering transform maps
  \[ H^{l,k} = \{ u \in H^l, x^k u \in L^2 \} \]
  bijectively to the scattering data (more precisely some Riemann-Hilbert data) in a space denoted by
  $H^k \cap L^2(1+|z|^{2l}|dz|)$ for any $k \ge 1$ and $l\ge 0$ on a contour which is the union of the real line with a large circle. For more details,  see Theorem 1.8 in \cite{MR1617249} for the map, and Theorem 2.7 and Theorem 2.11 for the inverse map for the simpler cases $l=0,1$; The case of
  large $l$ require only an understanding of the situation of large $z$, which is the same as for small data and has been understood already in \cite{MR928046}. 
  The Hamiltonian evolution of the $j$-th Hamiltonian is effectively linear on the scattering data $r$: It maps
  \[  r   \to  e^{ i  \frac{t}2  (2z)^j} r   \]
  which defines a strongly continuous group on $H^k \cap L^2( 1+|z|^{2jk} ) $ 
 and  a unique evolution on Schwartz functions. This result is much older in the case of generic data, see Beals and Coifman \cite{MR928046}.    

The odd order equations preserve real data. This is mKdV hierarchy. In
this case stronger local existence results as well as illposedness results in the sense of failure of uniform continuity have been proven by
Gr\"unrock \cite{MR2653659}. These flows are never smooth in $H^s$
with respect to the initial data, no matter how large $s$ is chosen, if the order of the equation is $5$ or higher.  In view of recent work by Harrop-Griffith, Killip and Visan \cite{HKV}, who proved well-posedness for NLS and mKdV in $H^s$, $ s >-1/2$, one might
hope that the whole hierarchy is well-posed in $H^s$ for all $s>-1/2$
on the real line.

\subsection{ The Galilean invariance and frequency shifts}

Classically the Galilean invariance is a symmetry of the NLS equation,
\begin{equation}\label{galilei}
u(x,t) \to  e^{i(x \xi_0  -t\xi_0^2)}      u(x-2t\xi_0 ,t).  
\end{equation}
This can be reformulated using the notation of \eqref{flows} as 
\begin{equation}\label{phaseshift}
  \begin{split} 
\Phi(2tz^2) e^{ix\xi_{0}} u_0 & \,  =    \Phi_2(t) (e^{ix\xi_0 } u_0) = e^{i x\xi_0} \Phi_2(t)\Phi_1(-2t\xi_0) \Phi_0(\xi_0^2) u_0
\\ & =   e^{ix\xi_0} \Phi(2 t z^2 - 2t\xi_0 z +  t \xi_0^2 )u_0
= e^{ix\xi} \Phi(2t (z-\xi_0/2)^2) u_0,
\end{split}
\end{equation}  
which says that phase shifts lead to linear combinations of all the lower flows. In this form it generalizes to all higher flows,
\begin{equation} 
\Phi(P) (e^{ix\xi_0 } u_0) = e^{ix\xi_0} \Phi(P(.-\xi_0/2))u_0.
\end{equation}

This is a straightforward algebraic computation for the corresponding linear flows, where one can use their Fourier representation. To pass to the  nonlinear flows, the easiest way is to use the scattering transform. This conjugates the nonlinear flows to the linear flows, while commuting with the phase shifts $G_\xi$. Such an argument 
applies rigorously in the Schwartz class, and extends by density to any Sobolev spaces $H^s$ where these flows are well defined.

\subsection{Solitons and soliton parameters}

Here we start from the soliton data $Q_0$ given by \eqref{soliton}.  
Then $2e^{4it} \sech(2x)$ is a soliton solution to NLS with initial data $Q_0$, and  $2\sech(2x-8t)$ is a soliton of mKdV with the same data. 
We use the symmetries of the NLS equation to generate a full set of
single soliton data. We divide this process into two parts:

\bigskip

\emph{a) Spectral parameters.} These correspond to two symmetries, namely 
(i) modulation,
\[
u(x) \to e^{-i x (2\xi)} u(x),
\]
and (ii) the scale invariance,
\[
 u(x) \to \lambda u(\lambda x).  
 \]
 The reason for the factor $2$ is again due to the mismatch between standard notations for the Fourier transform and the scattering transform as in \eqref{flows}.
We compose the two  symmetries into 
\begin{equation}\label{zdef} 
z =  \xi + i \lambda \in \bfH
\end{equation} 
and 
\[
S_z u(x) = e^{-i x (2\xi)}  \lambda u(\lambda x).  
\]

\bigskip 

\emph{b) Scattering parameters.} 
These correspond to two symmetries, namely 
(i) the translation invariance, which at the level of the initial data yield 
\[
u(x) \to u(x-a),
\]
and (ii) the phase invariance,
\[
 u(x) \to  e^{-2i \theta}  u(x).
\]

We also assemble these together into a complex variable as 
\begin{equation}\label{kappadef}  
\kappa =  a + i \theta,
\end{equation} 
and define
\[
S_\kappa u (x) = e^{-2i\theta} u(x-a).
\]
Then the set $\M_1$ of all single soliton data is given by 
 \[
Q_{z,\kappa}(x) =   S_{z} S_{\kappa} Q_0(x) = e^{-2i(\theta +   x\xi)} \lambda Q_0(\lambda(x-x_0) )
 , \qquad (z,\kappa) \in \Phase_1:=\bfH \times 
(\C/\ \pi i \Z),
\]
where the soliton location $x_0$ is defined by 
\[
a = x_0 \lambda.
\]
We can also represent $Q_{z,\kappa}$ using the first two flows,
\[
Q_{z,\kappa} =S_z \Phi( \theta -a \, \cdot) Q_0=  \Phi( \theta-\xi x_0 -x_0 \, \cdot) S_z Q_0.
\]

Again we calculate  
  \begin{equation}  \mathcal{L}(Q_{z,\kappa}) \left( \begin{matrix}   \sech(2(\lambda (x-x_0)) e^{-(i \theta+i \xi x  +\lambda(x-x_0))}      \\ -\sech(2(\lambda (x-x_0)) e^{i\theta+i \xi x + \lambda(x-x_0) }   \end{matrix} \right)   = z
    \left( \begin{matrix} \sech(2\lambda (x- x_0)) e^{-(i \theta +i \xi x + \lambda(x-x_0))  }   \\ - \sech(2(\lambda x-x_0))
      e^{i \theta+i\xi x +  \lambda(x-x_0)} \end{matrix} \right), 
  \end{equation}
which shows that $z$ is an eigenvalue and, with $\psi$ denoting the above eigenfunction to the eigenvalue $z$, we obtain 
\begin{equation} 
- e^{2\kappa}= \frac{\lim\limits_{x\to \infty}  e^{-izx} \psi_2(x)}{\lim\limits_{x\to -\infty} e^{izx}  \psi_1(x) }, 
\end{equation}
which gives the interpretation of $\kappa=i  \theta + \lambda x_0 $ as a scattering parameter.

  It is  interesting to describe the NLS and mKdV evolution in the single soliton space parametrized by $(z,\kappa) \in \Phase_1= \bfH \times 
(\C/\ \pi i \Z)$ as above. Clearly
\[
e^{4it} Q_0(x-x_0)
\]
is an NLS solution, and by scaling so is 
\[
 e^{4i \lambda^2 t} \lambda Q_0(\lambda(x-x_0)).
\]
The Galilean transform and phase shift give the general one soliton NLS solution  
\begin{equation} 
Q_{\theta,a,\xi,\lambda} =   e^{-2i\theta_0 -2 i \xi x}  e^{-i(4\xi^2 t - 4 \lambda^2 t) }         \lambda   Q_0(\lambda (x+4\xi t-x_0)). 
\end{equation}  
Here the parameters $\lambda$ and $\xi $ stay fixed, while for 
  $a$ and $\theta$ we obtain
  \[ 
  \dot a  = - 4 \lambda \xi, 
  \qquad
  \dot \theta  = 2( \xi^2 - \lambda^2).
  \]
  
  Thus for NLS (recall \eqref{zdef} and \eqref{kappadef}) we have
\begin{equation}\label{nls-S1}
\dot z = 0, \qquad 2 \dot \kappa = 2 ( \dot a + i \dot \theta )=    i (2z)^2, 
\end{equation}
Similarly,  for mKdV 
\[ 
Q_0( x-4t-x_0) 
\] 
is a solution, scaling gives gives the rescaled solution 
\[ 
\lambda Q( \lambda (x - 4 \lambda^2 t-x_0)), 
\] 
and \eqref{phaseshift}
specializes to  the general complex soliton solution 
\[
\begin{split} 
e^{2i\xi x}    \Phi_3(t) \Phi_2(-6\xi_0 t) \Phi_1(12t\xi_0^2) \Phi_0(-8t \xi_0^3)
\lambda Q_0( \lambda( x- x_0) )
 &
\\ 
= e^{i(2\xi x - t (8\xi^3 - 24 \xi \lambda^2))} Q_0(\lambda(x-x_0- 4 \lambda^2 t + 12t \xi^2 )) 
\end{split} 
\] 
and
\[ 
\dot a =  4 \lambda^3 - 12 \lambda \xi^2,     \qquad \dot \theta = 4 \xi^3 - 12 \xi \lambda^2,    \]
hence 
\begin{equation}\label{kdv-S1}
\dot z = 0, \qquad 2 \dot \kappa =     i (2z)^3.
\end{equation}
One can interpret these as Hamiltonian flows on 
\[ 
\M_1 = \{ Q_{\theta, a, \xi,\lambda} : \theta \in \R / \pi \Z, a, \xi \in \R, \lambda >0 \},
\] 
where 
\begin{equation}  
Q_{\theta,a,\xi,\lambda}(x) = e^{-2i\theta - 2i \xi x}     
\lambda   Q_0(\lambda x -a). 
\end{equation} 
 The trace formula (see Proposition~\ref{prop:trace}) shows that the restriction of the Hamiltonians to the manifold $\M_1$ is 
\[
H_{NLS} = \frac23 \im (2z)^3, \qquad H_{mKdV} = \frac24 \im (2z)^4,
\]
which could also be seen by a direct calculation. The restriction of the symplectic form 
\[ 
\sigma(u,v) = 2\im \int u \bar v dx 
\] 
defines a symplectic form on $\M_1$ which can be expressed in terms of the coordinates $(\lambda,\theta,a,\xi)$. We obtain the symplectic form by a direct calculation,
\Null{We begin with 
\[ 
\Vert Q_0 \Vert_{L^2}^2 = 4 \int \sech^2(2x) dx = 4 
\] 
and 
\[
\im \int \partial_\theta Q_{\theta,a,\xi,\lambda}\overline{ \partial_{a} Q_{\theta,a,\xi,\lambda}} dx = -2\lambda^{-1} \int Q_{0,0,\xi,\lambda} \partial_x Q_{0,0,\xi,\lambda} dx = 0, 
\] 
\[
\im \int \partial_\theta Q_{\theta,a,\xi,\lambda}\overline{ \partial_{\xi} Q_{\theta,a,\xi,\lambda}} dx = 0,
\]
\[
\im \int \partial_\theta Q_{\theta,a,\xi,\lambda}\overline{ \partial_{\lambda } Q_{\theta,a,\xi,\lambda}} dx = -2\lambda^{-1} \int Q_{0,0,0,\lambda} (Q_{0,0,0,\lambda} + x\partial_x Q_{0,0,0,\lambda}) dx = - \lambda^{-1} \Vert Q_{0,0,0,\lambda} \Vert_{L^2}^2 = -4,
\]
\[ 
  \im \int \partial_{a}  Q_{\theta,a,\xi,\lambda}\overline{ \partial_{\xi} Q_{\theta,a,\xi,\lambda}} dx = - 2\lambda^{-1}  \real \int \partial_x Q_{0,0,\xi, \lambda} x \overline{Q_{0,0,\xi, \lambda}} dx
      = \lambda^{-1}  \Vert Q_{0,0,0,\lambda} \Vert_{L^2}^2 = 4, 
\]
\[
\begin{split} 
      \im \int \partial_{a}  Q_{\theta,a,\xi,\lambda}\overline{ \partial_{\lambda } Q_{\theta,a,\xi,\lambda}} dx
       = &  -\lambda^{-1} \im \int \partial_{x} Q_{0,0,0,\lambda} (Q_{0,0,0,\lambda} + x  Q'_{0,0,0,\lambda})  dx = 0, 
\end{split} 
\]    
\[ \im \int \partial_{\lambda}  Q_{\theta,a,\xi,\lambda}\overline{ \partial_{\xi} Q_{\theta,a,\xi,\lambda}} dx = 2\lambda^{-1} \int ( 1+ x\partial_x ) Q_{0,0,0, \lambda} x Q_{0,0,0, \lambda} dx = 0. 
\] }
\[ 
\omega =
      8 (d\lambda \wedge d\theta  + da \wedge d\xi) = 2 (d \kappa \wedge dz +
      d \bar \kappa \wedge d \bar z),
\]
     and the Hamiltonian equations for the NLS flow on $\M_1$ are
\[
\dot \lambda =  -\frac18   \frac{\partial H}{\partial \theta}, 
\quad
\dot \xi =  \frac18  \frac{\partial H}{\partial a}, 
\quad
\dot a = -\frac18\frac{\partial H}{\partial \partial \xi},
\quad
\dot \theta = \frac18 \frac{\partial H}{\partial \lambda}.
\] 
which coincides with the dynamics above for NLS and mKdV in \eqref{nls-S1}, \eqref{kdv-S1}. 
     
A similar reasoning, also based on trace formulas, shows that the $n$-th Hamiltonian restricted to $\M_1$ is
\[
H_n= \frac{2}{n+1}  \im (2z)^{n+1},
\]
and the $n$-th evolution in $\M_1$ is
\begin{equation}\label{nflow-S1}
   \dot z = 0, \qquad  2\dot \kappa = i (2z)^n .
\end{equation}

\subsection{ The Lax operator and the transmission coefficient}

The Lax operator associated to a state $u$ is given by 
\[
\mathcal L = i \left(\begin{matrix} \partial_x & - u \\ - \bar u  & - \partial_x
 \end{matrix} 
\right),
\]
and the associated spectral problem is 
\[
L \psi = z \psi.
\]

For $z$ in the upper half-space one can consider special solutions $\psi_l$ and $\psi_r$, 
called the  Jost  solutions, 
\[
\psi_l (\xi,x,t) = \left( \begin{array}{c}    e^{-iz x}  \cr 
0 \end{array}\right)+ o(1) e^{\im z x}  \ \ \ \text{ as $x \to -\infty$},
\] 
\[ 
 \psi_l (\xi,x,t) = \left( \begin{array}{c}   T^{-1}(z)  e^{-iz x}  \cr 
0 \end{array}\right) + o(1)e^{\im zx}  \ \ \ \text{ as $x \to \infty$}.
\] 
The function $T = T(z,u)$, called the \emph{transmission coefficient}, 
is a meromorphic function in the upper half-space, and 
satisfies $|T| \geq 1$.  As $u$ evolves along any of the commuting flows
of the NLS family, the transmission coefficient rests unchanged.

The $L^2$ size of a state $u$ can be described using the transmission coefficient as 
\[
\| u\|_{L^2}^2 = \lim_{z \to i\infty} 2z  \ln T(z,u).
\]
Moving this relation to the real axis using the residue theorem yields the trace formula
\begin{equation} \label{tracel2}
\| u\|_{L^2}^2 = \frac1{\pi} \int_{\R} \ln |T(\xi/2)|d\xi  + 2  \sum_j n_j \im (2z_j),
\end{equation}
where $(z_j,n_j)$ are the poles of $T$ with their multiplicity.

A function  $u= Q_{z_0,\kappa}$ is a single soliton with spectral parameter $z_0$ 
if and only if  its transmission coefficient 
has a pole exactly at $z_0$  and
\[
T(z) =  \frac{z - \bar{z}_0}{z-z_0}. 
\]
In particular $z_0$ is an eigenvalue of its Lax operator $\mathcal L $, and the corresponding eigenfunction $\phi_{z_0}$ is a multiple of both $\psi_l$ and $\psi_r$. Then the  scattering 
parameter $\kappa$ can be read as the proportionality factor
between the two Jost functions,
\begin{equation}\label{scattering}  
\psi_l = - e^{2\kappa} \psi_r. 
\end{equation}

\subsection{ A full family of conservation laws}
In a prior article \cite{MR3874652} the authors have  extended the countable family of conservation laws
for the $1$-d cubic NLS and mKdV, associated to integer Sobolev indices, 
to a continuous family, associated to all real Sobolev exponents $s > -\frac12$. Related conserved energies have been independently constructed by Killip-Visan-Zhang \cite{MR3820439}, for the range $ -\frac12 < s \le 1$.  

\begin{theorem}[\cite{MR3874652}]\label{energies} 
  For each $s > -\frac12$ there exist energy functionals $E_s$  which are globally defined
\[
E_s :  H^s    \to \R,
\]
with the following properties:
\begin{enumerate}
\item $E_s$ is conserved along the NLS and mKdV flow.

\item For all $u \in H^s$ the limit of\footnote{The choice of signs
    $\mp$ corresponds to the defocusing/focusing case} $\mp \log |T|$
  exists as a positive measure, and the trace formula \eqref{deff+}
  holds with absolute convergence in all sums and
  integrals.
\item  If $\Vert u \Vert_{l^2_1DU^2}\le 1 $ then   
\[
\left| E_s(u) - \|u\|_{H^s}^2 \right| \lesssim  \Vert u \Vert_{l^2_1DU^2}^2 
\Vert u \Vert_{H^s}^2. 
\]
\item The  map 
\[ 
H^{\sigma} \times  (-\frac12,\sigma] \ni (u,s) \to  E_s(u)  
\]
 is analytic provided $\frac{i}2$ is not an eigenvalue, and it is continuous   in $u \in H^\sigma$
in general. It is also continuous in $s$, and analytic in $s$ for $s < \sigma$.
\end{enumerate}
\end{theorem}

Here the threshold $i/2$ is not important, and can be changed by scaling. Of course this would also change the energies $E_s$, though not in an essential  way.

The Banach space $l^2 DU^2 = L^2 +DU^2$ is the inhomogeneous version
of the $DU^2$ space, and contains all $H^s$ spaces with $s >
-\frac12$.  It is described in full detail in Appendix A of
\cite{MR3874652}. It can be viewed as a replacement for the unusable scaling
critical space $H^{-\frac12}$ and satisfies
\begin{equation}  
\Vert u \Vert_{l^2 DU^2} \lesssim \Vert u \Vert_{B^{-\frac12}_{2,1}} 
\lesssim \Vert u \Vert_{H^{s}}, \qquad s >-\frac12.
\end{equation}

The energies $E_s$ in the theorem were defined in \cite{MR3874652} in terms 
of the transmission coefficient $T$. A full description is provided 
in the next result, also from \cite{MR3874652}, which also doubles as a trace 
formula.

\begin{proposition}[Trace formulas, \cite{MR3874652}]
\label{prop:trace}  
Let  $N > [s] $ and $ u \in \mathcal{S}$.  In the upper half-space we define the function
\[
 \Xi_s(z) = \im \int_0^z (1+\zeta^2)^{s} d\zeta, 
\] 
which does not depend on the path of integration.  Then
\begin{equation}\label{deff++} 
\begin{split}  
E_s(u) = & \ \int (1+\xi^2)^{s} \real \ln T(\xi/2) d\xi+ 2 \sum_k m_k \Xi(2z_k)  \\
    = & \ 4\sin(\pi s) \!\! \int_1^\infty \!\! (\tau^2 -1)^ s\Big[\!\! - \!\! \real  \ln T(i\tau/2)  +\frac1{2\pi} 
\sum_{j=0}^N (-1)^j  H_{2j} \tau^{-2j-1}\Big] d\tau 
+ \sum_{j=0}^N \binom{s}{j}  H_{2j}, 
\end{split} 
\end{equation} 
where the $k$ sum runs  over all the poles $z_k$ of $T$ with multiplicity $m_j$. 
\end{proposition} 

If there are infinitely
many poles for $T$ in the upper half-space then the second expression
above is always a convergent integral, whereas in the first expression
we have a non-negative integral, plus a sum where all but finitely many
terms are positive. This simultaneously allows us to interpret the
trace of $\ln |T|$ on the real line as a non-negative measure, and to guarantee
the convergence in the $k$ summation.

We also remark on the contribution of the poles which are on the imaginary axis.
Precisely the function $\Xi_s$  is real analytic away from $z=i$. Thus the 
only nonsmooth dependence on $u$ in $E_s$ via the poles comes from
the poles which are  at  $i$. 

Here the choice of the function $(1+z^2)^s$ was somewhat arbitrary,
all that matters is that it is holomorphic in the upper half-space
minus $i[1,\infty)$ and has the appropriate behavior at infinity. 
In particular, the conserved Hamiltonians $H_j$ correspond to the 
the functions $z^j$, and they can be expressed as  
\begin{equation}
    H_j(u) =  \frac1\pi \int \xi^{j} \real \ln T(\xi/2) d\xi+ 2 \sum_k m_k \dot\Xi_j(2z_k),  
\end{equation}
where
\[
\dot\Xi_j(z) = \im \int_0^{2z} \zeta^j d\zeta =  \frac{1}{j+1} \im (2z)^{j+1}.
\]
For pure solitons the contribution of the first term vanishes, and we are 
left with
\begin{equation}
    H_j(Q_{z,\kappa}) = \frac{2}{j+1} \im (2z)^{j+1},  
\end{equation}
as mentioned earlier in the paper.


To further clarify the assertions in the theorem, we note that the energy conservation result is established for regular initial data.
By the local well-posedness theory, this extends to all
 $H^s$ data above the (current) Sobolev local well-posedness threshold,
which is $s \geq 0$ for NLS, respectively $s \geq \frac14$ for
mKdV. If $s$ is below these thresholds, then the energy conservation
property holds for all data at the threshold, i.e. for $L^2$ data for
NLS, respectively $H^\frac14$ data for mKdV.  It is not known whether
the two problems are well-posed below these thresholds and above the
scaling; however, it is known that local uniformly continuous
dependence fails, see \cite{MR2376575}. Recently Harrop-Griffith, Killip and Visan \cite{HKV} proved that in the defocusing case  the flow map extends to a continuous map on $H^s$, for $ s >-\frac12$, for both NLS and mKdV. 

One key consequence of the above result is that, if the initial data is in
$H^s$, then the solutions remain bounded in $H^s$ globally in
time in a uniform fashion:

\begin{cor}[\cite{MR3874652}] \label{c:large-data}
Let $s > -\frac12$, $R > 0$ and $u_0$ be an initial data for either NLS or mKdV so that 
\[
\| u_0\|_{H^s} \leq R
\]
Then the corresponding solution $u$  satisfies the global bound
\[
\| u(t)\|_{H^s} \lesssim F(R,s) : = \left\{ \begin{array}{ll} R+ R^{1+2s} & s \geq 0 \cr
R+ R^{\frac{1+4s}{1+2s}} &  s < 0 \end{array} \right.
\]
\end{cor}

This follows directly from the above theorem  if $R \ll 1$. For larger 
$R$  is still follows from the theorem, but only after applying the scaling 
\eqref{scaling}. Here one needs to make the choice   $\lambda = cR^{-2}$, with $c \ll 1$ if 
$s \geq 0$, respectively $\lambda = cR^{-\frac{2}{1+2s}}$ if $s < 0$.

\subsection{ Multisolitons}
The main objective of this article is to study multisoliton solutions,
both by investigating the geometry of the set of multisoliton states
and by studying its stability under the family of commuting flows.
The first step is to define multisoliton solutions.

A natural venue to  define $N$-soliton solutions for NLS is to start with $n$  single solitons $Q_1, \cdots, Q_N$, with soliton parameters $(z_1,\kappa_1), \cdots (z_n,\kappa_N)$. Assuming that the soliton speeds $\real z_1, \cdots \real z_N$ are distinct, these solitons separate at infinity, and one can actually prove (see \cite{MR905674})  the existence of a unique solution $Q$ so that 
\begin{equation} \label{asymptoticdistance}
Q - (Q_1+ \cdots + Q_N) \to 0 \quad in \ L^2 \qquad \text{as} \ t \to \infty.
\end{equation} 

However, the above venue does not readily extend to solitons with equal speeds, and also it involves the time evolution. Instead, we will take advantage of the complete integrability of the problem, and 
use the spectral picture for the Lax operator in order to define 
$N$-multisoliton states:

\begin{definition}   A function $Q$ is an $N$-multisoliton state
with spectral parameters $z_j$, $1 \le j \le N$, if its transmission coefficient is 
\begin{equation}\label{T-SN}
T(z) = \prod_{j=1}^N \frac{z-\bar z_j}{z-z_j}. 
\end{equation}
The set of all $N$-multisolitons is denoted by $\M_N$.
\end{definition} 
The corresponding Lax operator has the values $z_j$ as eigenvalues, with multiplicity corresponding to the multiplicity of the $z_j$'s. 
If the spectral values $z_j$ are all different then we define the associated scattering parameters using the corresponding eigenfunctions by \eqref{scattering}. Alternatively one may define $N$ solitons as stationary solutions to a linear combination of the first $2N$ flows \cite{MR1992536,MR1636296}, an approach we do not pursue. 

We remark that when the definition based on \eqref{asymptoticdistance}
applies, it yields the same spectral parameters $z_j$ as in \eqref{T-SN}.
On the other hand, the scattering parameters predicted by \eqref{scattering} and \eqref{asymptoticdistance} are slightly different,
as a shift in the effective scattering parameters occurs when two solitons interact. This shift was approximatively computed by Faddeev and Takhtajan \cite{MR905674}, at least 
in the case of separated spectral parameters $z_j$.
Precisely,  if all the $z_j$ are distinct 
then the effective soliton position $\hat x_j$ and phase $ \hat \theta_j$ at infinity in \eqref{asymptoticdistance}  satisfy 
\[ \begin{split} 
\hat x_j - x_j \approx & \,  \frac1{2 \im z_j } \Big[  \sum_{x_k < x_j} \ln \left| \frac{ z_j-\bar z_k}{z_j-z_k}\right| - \sum_{x_j > x_k} \ln \left| \frac{z_j-\bar z_k}{z_j-z_k}\right| \Big]
\end{split} 
\] 
and, modulo $\pi$, 
\[ 
\begin{split} 
\hat \theta_j - \theta_j \approx &  \sum_{x_k< x_j} \arg \frac{z_j-\bar z_k}{z_j-z_k}
- \sum_{x_j > x_k}  \arg \frac{ z_j-\bar z_k}{z_j-z_k}   
\end{split} 
\] 
with errors that decay to zero as $x_j$'s separate.  On the other hand, these errors
grow as the $z_j$'s get closer to each other.
 
 We note that the $x_j$'s can always be separated by flowing far enough 
 along the combined NLS-mKdV flow. Indeed, suppose that the spectral parameters $(z_j)_{j \le N}$ are all simple and denote the scattering parameters by $\kappa_j$. Let $s$ resp $t$ be the times of the flow of NLS, respectively
 mKdV. Then 
\[
x_j(s,t)\im z_j  + i \theta_j(s,t) =  \kappa_j(s,t) =   \kappa_j +  2is z^2_j + 4it z_j^3,
\]
and we can choose $(s_n, t_n)$ so that the distance between the $x_j$'s tends to $\infty$. 

It follows from calculations as in Faddeev and Takhatajan ~\cite{MR905674} 
that the subset of multisolitons with simple eigenvalues is smoothly parameterized by the 
\[  \{ (z_j,\kappa_j) | z_j \in \{ \im z >0\}, \kappa_j \in \R / (\pi i \Z) , z_j \ne z_k \} 
\] 
On the other hand the singularity on the diagonal in the asymptotic formulas above reflects the fact, discussed in detail later on,  that the parametrization of the set of multisoliton states via the soliton parameters is singular near the diagonals $z_j = z_k$. This leads us to a very interesting question:
   
  \medskip 
  
   {\bf Is the set $\M_N$ of all $N$ multisoliton states a smooth manifold in some (any) reasonable Sobolev topology, or is it singular at the spectral values with higher multiplicity ?}

\medskip 

The first aim of this article is to provide an answer to
this question:

\begin{theorem}\label{uniformSol} 
  The set $\M_N$ is a uniformly smooth $4N$ dimensional symplectic 
  submanifold of $H^s$ for each $s > -\frac12$. 
Here uniformity holds with respect to spectral parameters restricted to a compact subset of   the upper half-space, but without any separation assumption.
  \end{theorem}  
  
  Further results concerning the structural properties of $N$-solitons are developed in Section~\ref{s:structure}.
 
 The first challenge in the proof of the theorem will be to resolve the apparent singularity near the diagonal $z_j = z_k$. But a second, equally difficult challenge is to establish the uniformity in the theorem.
 
 We now briefly outline the main steps in the construction of our
 smooth parametrization of the $N$-soliton
 manifold $\M_N$:
 
 \begin{itemize}
     \item To capture the symmetry of spectral parameters, we endow the set $\bfz = (z_1, \cdots z_N)$ of unordered spectral parameters
     with the smooth topology defined by the symmetric polynomials $\bfs$ in $\bfz$, 
     \[  s_j = \sum_{n=1}^N z_n^j \] 
     for $1 \le j \le N$. Let 
\[ W_{N} = \{ \bfs \in \C^N: \im z_n > 0 \}. \]      .
          \item Corresponding to scattering parameters $\kappa_j = 0$
we have a $N$-soliton state $Q_{\bfz,0}$, which is shown to depend smoothly on $\bfz$ in the above smooth topology.
\item We shift the scattering parameters away from $0$ using the first $2N$ commuting flows, see \eqref{flows}, in the form
  \[  P(z) = \sum_{n=0}^{2n-1} \beta_n z^n \] 
  to  define
    \[
    Q_{\bfz,\bbeta} = \Phi(P) Q(\bfz,0) = \prod_{n=0}^{2N-1}  \Phi_n( 2^{n-1} \beta_n  )   Q(\bfz,0).  
    \]
         \item \label{smoothness} The map 
    \[
    (\bfz,\bbeta) \to Q_{\bfz,\bbeta}
    \]
     provides locally a smooth parametrization of the $N$-soliton manifold
     $\M_N$.
 \end{itemize}
 
 We note that this parametrization is global, and corresponds to scattering parameters
 \[
\kappa_j = i P(z_j),  
 \]
 at least for distinct eigenvalues. It defines a diffeomorphism between the smooth manifold  
 \[ 
 \Phase^N = ( W_N \times \C^N)/ \{ \kappa_j \in i \pi \Z\} 
 \] 
 and $\M_N$.  However, it is not a uniform parametrization.

\subsection{ Soliton stability}

One can view the small data part of our earlier result in Theorem~\ref{energies} as a stability
statement in $H^s$ for the zero solution of the NLS or mKdV
equations. In the focusing case another interesting class of solutions
are the pure $N$-soliton solutions, which belong to all Sobolev spaces 
$H^s$.

The second goal of  the present article is to  establish the $H^s$ stability of 
these families of solutions. For convenience we state here 
a less precise form of the result, but which has the advantage 
that no further preliminaries are needed.
A more precise form is provided in Section~\ref{s:stable}.

\begin{theorem}\label{t:solitons}  
Let $s > -\dfrac12$, and $u_0 \in \M_N$ be a pure $N$-soliton 
state. Then

a)  There exist $\varepsilon_0  > 0$ and $C>0$ so that, for each 
initial data $w_0$ which satisfies
\begin{equation}
\| u_0 -w_0\|_{H^s} = \varepsilon < \varepsilon_0,  
\end{equation}
there exists another pure multisoliton data $v_0 \in \M_N$ so that 
the corresponding solutions for either NLS or mKdV satisfy
\begin{equation}
  \sup
  _{t \in \R} \| w(t) - v(t)\|_{H^s} \leq C \varepsilon.
\end{equation}

b) Furthermore, this result is uniform with respect to all $N$-soliton states $u_0$ with spectral parameters in a compact subset of the open upper half-plane, and holds for all commuting flows of the NLS hierarchy.   
\end{theorem}

To place this result into context, we note that the stability of solitons has been an intensely studied area. During the last decade there have been several stability results proved using the integrable structure: Hoffman and Wayne \cite{MR3059166} describe the use of the \Backlund transform to obtain stability of (multi)solitons.  
Mizumachi and Pelinovsky \cite{MR2920823} proved stability in $L^2$ and  Cuccagna and Pelinovsky  \cite{MR3180019} proved asymptotic stability of single solitons for localized initial data.

Turning our attention to multisolitons, we remark that a much more restrictive form of this theorem has been previously proved in $L^2$ or $H^1$ under the additional assumption that the spectral parameters of the multi-soliton are distinct; Alejo-Mu\~{n}oz \cite{MR3353827} considered the dynamics of solitons for complex mKdV and their stability and the stability of breathers \cite{MR3116324, MR2989213}.  Kapitula \cite{MR2307885} and  Contreras-Pelinovsky \cite{MR3214610} studied the stability of NLS multisolitons.

 Our result here drastically  improves these prior results as follows:

\begin{itemize}
\item We allow multisolitons with multiple spectral parameters.
\item We prove a result which is uniform near such multi-solitons.
\item We prove stability in a full range of Sobolev spaces.
\end{itemize}
The proof of the above theorem is  completed in Section~\ref{s:stable}.
A key ingredient in the proof is provided by the trace formula
in Proposition~\ref{prop:trace}. This allows us to find energies for which the infimum on the set of potentials with given eigenvalues $z_j$ with multiplicity is minimized for pure $N$-solitons with these spectral parameters. Moreover we may show that this energy is uniformly convex in 
the transverse direction in a uniform neighbourhood of this set of $N$-solitons. The stability then follows once we prove the uniform smoothness of the $N$-soliton manifold $\M_N$.

\subsection{Iterated \Backlund transforms}
A key tool in proving the results in this paper is the \Backlund transform, which adds or subtracts a soliton with given spectral and scattering parameter. 
On the level of the trace formula the action is very transparent: One adds or subtracts the contribution coming from this spectral value. 

Its use in a setting without multiplicities is described by Hoffman and Wayne \cite{MR3059166}.
Similarly, the iterated \Backlund transform adds or subtracts a 
multi-soliton  \cite{MR905674, MR1146435,MR3214610,MR2307885,MR3353827,zbMATH02209608}.  Unfortunately the naive iterated \Backlund transform deteriorates as spectral parameters get close. We overcome this difficulty by 

\begin{enumerate} 
\item Choosing an appropriate blow-up of the spectral and scattering coordinates near points with multiplicity, i.e. we prove that the parametrization described in Subsection \ref{smoothness} is smooth.  
\item Establishing a cancellation property for the iterated \Backlund transform, and using it to  verify  the surjectivity of this parametrization.
\end{enumerate} 

The outcome of this analysis is 

\begin{enumerate}[label = (\roman*)]
\item A soliton addition map
\[
B^N_+: H^s \times \M_N \to H^s,
\]
which nonlinearly adds an $N$-soliton $Q_{\bfz,\bbeta}$ to an $H^s$ state with no eigenvalues near $\bfz$.

\item A soliton removal map
\[
B^N_-: H^s \to H^s \times \M_N, 
\]
which reverses the above process, nonlinearly splitting an 
$H^s$ state near $\M_N$ into an $N$-soliton $Q_{\bfz,\bbeta}$ and 
an $H^s$ state with no eigenvalues near $\bfz$.
\end{enumerate}

For these maps we establish two smoothness and flow commutation properties:

\begin{itemize}
    \item They are smooth, inverse maps.
    \item They commute with the NLS and mKdV flow, and all other 
    commuting flows, which are viewed as acting separately on each of the inputs/outputs.
    
\end{itemize}

The arguments establishing this  are local, and yield the smoothness of the $N$ soliton set $\M_N$. To obtain the uniform smoothness of 
$\M_N$ we combine this with the uniform regularity of the energies and  with relations derived from the trace formula. 

Finally, we turn our attention to the full uniform smoothness of the soliton addition and removal maps. In full generality, this remains open:

\begin{conjecture}
The maps $B^N_+$ and $B^N_-$ are uniformly smooth on bounded sets in $H^s$.
\end{conjecture}

A more precise version of this conjecture is stated in Section~\ref{s:dress+}. A main difficulty in proving this is that 
our smooth $(\bfz,\bbeta)$ parametrization of the $N$-soliton
manifold deteriorates as the solitons separate (which corresponds to $\bbeta \to \infty$). Nevertheless, in the same section we prove some partial results in this direction for the soliton addition map;

\begin{itemize}
    \item The map $(u,\bfz,\bbeta) \to B_N^+(u,\bfz,\bbeta)$ is uniformly 
    smooth on compact sets in $\bbeta$.
    \item The same map is uniformly smooth in $(u,\bbeta)$
\end{itemize}
We also establish a uniform invertibility result for the soliton addition map with respect to its first argument, as follows:

\begin{theorem}\label{t:ureg} 
  The set of potentials $M_{\bfz,\bbeta}$ 
with spectral values $(\bfz,\bbeta)$ is a manifold of codimension $4N$. It is uniformly smooth in an $\varepsilon$ neighborhood of $\M_N$, and uniformly transversal to $\M_N$. 
  \end{theorem}

\subsection{An outline of the paper}
\label{s:outline} 
This paper aims to accomplish several goals, all having to do with 
multisoliton states and their perturbations in the 
completely integrable NLS and mKdV flows:
\begin{enumerate}
    \item Understand the structure of multisoliton states, with emphasis on close or multiple
    spectral parameters.
    \item Understand the geometry (smoothness and uniformity) of the multisoliton manifold.
    \item Study nearby states and their evolution  using the (multi)soliton addition and removal maps.
    \item Prove uniform orbital stability of the multisoliton manifold.
\end{enumerate}

Here we provide a road map for the reader:

\bigskip 

\emph {1. The Lax operator $\mathcal{L}$ and the associated spectral problem.}
 The goal of the next section is to provide an overview
of this spectral problem, in particular the left and right Jost
functions and more generally their linear combinations, which are
called \emph{wave functions}.  These are in turn used to define the
transmission coefficient $T(z)$ as a meromorphic function in the
upper half-space. Finally, the transmission coefficient $T$ has also
played a key role in the construction of the continuous family of
conservation laws; we provide an overview of these as well.

\bigskip
\emph {2. The  \Backlund transform.} 
This allows one to add or remove a soliton with given spectral and scattering 
parameters $(z,\kappa)$ to/from an existing state $u \in H^s$. It is, in turn,
constructed using the wave functions for Lax wave operator.
Section~\ref{s:backlund} is devoted to the \Backlund transform, 
both in the standard form and in an extended form; the latter is needed in order to better 
understand its symmetry properties. A new notion we introduce here it that of analytic families of wave 
functions, and their  associated \Backlund transform.

\bigskip

\emph{3. The  multisoliton addition and removal maps.} 
These allow one to add/subtract an $N$-soliton
to/from a given state, and are obtained by iterating
\Backlund transforms. They are classically defined relative to
solitons with distinct spectral parameters, and are described in
Section~\ref{s:dress}.

\bigskip

\emph{4. The smooth parametrization of the multisoliton addition and removal maps.}
Viewed as functions of the spectral and scattering parameters for $N$-solitons,
the addition and removal maps are singular at the diagonal, near multiple spectral 
parameters. A key result of this paper is that this is a singularity of the parametrization,
which can be removed by making a better choice for the parametrization of 
the sets of joint spectral and scattering parameters. This yields a smooth extension of 
the soliton addition and removal maps to solitons with higher multiplicity spectral parameters, 
and is naturally done using analytic families of wave functions. We introduce our 
extended spectral/scattering parameters, which in particular provide the 
desired reparametrization of multisoliton states in Section~\ref{s:dress+}. 
The smoothness of both the multisoliton addition  and removal  maps is 
proved in the next section.

\bigskip

\emph{5. The uniformity question.}
Ideally, one would like both the multisoliton addition  and removal  maps to be uniformly smooth 
when restricted to spectral parameters in a compact subset of the upper half-space. One major difficulty 
we encounter is that our smooth parametrization of spectral/scattering parameters, although 
natural, is not uniform. Nevertheless, in Section~\ref{s:dress+} we are able to prove several partial 
uniformity results. These in particular lead us to the proof of one of our main results, namely the uniform regularity of the multisoliton manifold.

\bigskip

\emph{6. The structure of multisolitons.}
One can think of multisoliton states as a collection of bump functions, with exponential 
decay away from these bumps. In Section~\ref{s:structure} we show that if sufficiently separated, 
these bump functions are exponentially close to lower rank multisolitons. This also leads to 
a good local uniform  description of the multisoliton manifold $\M_N$ as an approximate sum of $\M_{N_j}$'s in the single bump regime, where our parametrization is uniform.

\bigskip

\emph{7. Uniform orbital stability of multisolitons.} Section~\ref{s:stable} contains the proof of the stability result.  This relies heavily on our previous results on the uniformity properties for the soliton addition and removal maps, as well as on the conserved energies developed in our prior work.

\bigskip

\emph{8. $2$-solitons: a case study.} The aim of the last section
of the paper is to provide a complete analysis  for the case of $2$-solitons. In particular we
accurately describe both the $2$-soliton manifold, as well as the NLS
and mKdV flows on this manifold. This serves to both illustrate the concepts introduced in the rest of the paper, as well as to provide a full analysis of the interaction patterns of two solitons, both 
for the NLS and for the mKdV flows. Multiple pictures are also provided.

\subsection{Acknowledgements}
The first author was supported by  the DFG through the SFB 611. 
The second author was supported by the NSF grant DMS-1800294 as well as by a Simons Investigator grant from the Simons Foundation.

\section{An overview of the scattering transform}
\label{s:scattering}
\subsection{Lax pair, Jost solutions and scattering transform}  
Here we recall some basic facts about the inverse scattering transform
for NLS and mKdV.  Both the NLS evolution \eqref{nls} and the mKdV
evolution \eqref{mkdv} are completely integrable, so we have at our
disposal the inverse scattering transform conjugating the nonlinear
flow to the corresponding linear flow.  To describe their Lax pairs
 we consider the system
\begin{equation}\begin{split}
\psi_x = &\left(\begin{matrix} -iz & u \\ - \bar u  &iz \end{matrix} 
\right) \psi   \\
\psi_t = & i\left(\begin{matrix}-[2z^2 -|u|^2] & -2izu+u_x \\
+ 2 i z\bar u + \bar u_x  & 2z^2 -|u|^2\end{matrix}\right)\psi,  
\end{split} \label{laxpairnls}
\end{equation}
 where $z$ is a complex parameter. The focusing
NLS equation arises as a compatibility condition for the system
\eqref{laxpairnls}: For fixed $z$ there exist two unique solutions
$\psi_1$, $\psi_2$ to \eqref{laxpairnls} with $\psi_1(0,0)=(1,0)$ and
$\psi_2(0,0)=(0,1)$ if and only if $u$ satisfies the nonlinear
Schr\"odinger equation.  The above is  often  referred to in the literature as the Lax pair for NLS.

If instead we want the canonical form $\mathcal L,\mathcal P$ with 
\[
\mathcal L_t = [\mathcal P,\mathcal L],
\]
then we should view the first equation above as $L \psi = z \psi$
where
\[
\mathcal L = i \left(\begin{matrix} \partial_x & - u \\ -\bar u  & - \partial_x
 \end{matrix} 
\right) 
\]
and $\mathcal P$ is given by the second matrix in \eqref{laxpairnls} where $z$ has been eliminated 
using the relations $L \psi = z \psi$,
\[
\begin{split}
\mathcal P =
i \left(\begin{matrix} 2 \partial_x^2 + |u|^2  &  -  u \partial_x -   \partial_x u \\  - \bar u \partial_x - \partial_x \bar u & -2 \partial_x^2 - |u|^2   \end{matrix}\right).
\end{split}
\]
This is equivalent to the pair of Kappeler and Grebert
\cite{MR3203027}.  Much of this formalism can be found
in the seminal paper by Ablowitz, Kaup, Newell and Segur
\cite{MR0450815}.
The Lax operator $\mathcal L$ is the same for mKdV and for all
the other commuting flows. It is only the operator $\mathcal P$
that will change.

The scattering transform associated to both the focusing NLS and the
focusing mKdV is defined via the first equation of
\eqref{laxpairnls}  which we write as linear
system
\begin{equation}\label{scatter}
\left\{ 
\begin{array}{l}
\dfrac{d\psi_1}{dx} = -i z \psi_1 + u \psi_2 
\cr \cr
\dfrac{d\psi_2}{dx} =  i z \psi_2 - \bar u \psi_1.
\end{array}
\right.
\end{equation}

One part of the scattering data for this problem is
obtained for $z=\xi$, real, by considering the relation between the
asymptotics for $\psi$ at $\pm \infty$.  Precisely, one considers the
Jost solutions $\psi_l$ and $\psi_r$ with asymptotics
\[
\psi_l (\xi,x,t) = \left( \begin{array}{c}    e^{-i\xi x}  \cr 
0 \end{array}\right) + o(1) \ \ \text{ as } x \to -\infty,
\quad \psi_l (\xi,x,t) = \left( \begin{array}{c}   T^{-1}(\xi)  e^{-i\xi x}  \cr 
R(t,\xi) T^{-1}(\xi)  e^{i\xi x} \end{array}\right) + o(1) \ \ \text{ as } x \to \infty,
\] 
respectively
\[
\psi_r (\xi,x,t) = \left( \begin{array}{c}  L(t,\xi)  T^{-1}(\xi)  e^{-i\xi x}  \cr 
T^{-1}(\xi)  e^{i\xi x} \end{array}\right) + o(1) \ \text{ as }x \to -\infty,
\quad \psi_l (\xi,x,t) = \left( \begin{array}{c}   0  \cr 
e^{i\xi x} \end{array}\right) + o(1)  \ \text{ as }x \to \infty.
\] 
These are viewed as initial value problems with data at $-\infty$,
respectively $+\infty$.  We note that the $T$'s in the two solutions
$\psi_l$ and $\psi_r$ are the same since the Wronskian of the two
solutions is constant:
\[
\det ( \psi_l ,\psi_r) \to T^{-1}(\xi) \qquad \text{ for } x \to \pm \infty. 
\] 
The quantity $|\psi_1|^2 + |\psi_2|^2$ is also conserved,
which shows that on the real line we have
\[
|T| \geq 1, \qquad |T|^2 = 1+|R|^2 = 1+|L|^2.  
\]
 Further, we have the symmetry $(\psi_1,\psi_2) \to (\bar
\psi_2,-\bar \psi_1)$ which via the Wronskian leads to
\[
L \bar T = \bar R T.
\]

It is an immediate consequence of the existence of the Lax pair that
as $u$ evolves along the  NLS flow \eqref{nls}, the functions $L,R,T$
evolve according to
\begin{equation} \label{NLSflow} 
T_t = 0, \qquad L_t = - 4i \xi^2 L, \qquad R_t = 4i\xi^2 R,
\end{equation} 
if $u$ evolves according to the mKdV flow \eqref{mkdv} then 
\begin{equation} \label{mKdVflow} 
T_t = 0, \qquad L_t = - 8i \xi^3 L, \qquad R_t = 8i\xi^3 R,
\end{equation} 
and if $z$ evolves according to the $n$th flow 
\begin{equation} \label{nthflow} 
T_t = 0, \qquad L_t = - i (2\xi)^n L, \qquad R_t = i(2\xi)^n R.
\end{equation}

Thus one part of  scattering map for $u$ is given by 
\[
u \to R,
 \]
which maps  the NLS flow \eqref{nls} to the
(Fourier transform of) the linear Schr\"odinger evolution, and
simultaneously the mKdV flow to the linear Airy flow.

More generally, for any $z$ in the closed upper half plane there exist the Jost 
solutions 
\[
\psi_l (\xi,x,t) = \left( \begin{array}{c}    e^{-iz x}  \cr 
0 \end{array}\right)+ o(1) e^{\im z x}  \ \ \ \text{ as $x \to -\infty$},
\] 
\[ 
 \psi_l (\xi,x,t) = \left( \begin{array}{c}   T^{-1}(z)  e^{-iz x}  \cr 
0 \end{array}\right) + o(1)e^{\im zx}  \ \ \ \text{ as $x \to \infty$},
\] 
This provides a holomorphic extension of $T^{-1}$ to the upper half-space, and thus 
a meromorphic extension for $T$. Here  $T$ may have poles in the upper half-space, which 
 correspond to non-real eigenvalues of $\mathcal L$.  The poles of $T$ 
must be isolated in the open upper half space, though they can accumulate 
on the real line.


For data $u$ for which $T$ is holomorphic in the upper half-space, the
scattering data is fully described by the reflection coefficient $R$. If
instead $T$ is merely meromorphic, then the scattering data involves not only the function $R$ on the real line, but also at least the singular part of the
Laurent series of $T$ at the poles.  However, this still does not fully
describe the problem, as by the results of Zhou \cite{MR1008796},  $T$ may have poles in
the upper half space accumulating at the real axis even for Schwartz
functions $u$.

There is, however, one redeeming feature: All such poles are localized
in a strip near the real axis if $u \in L^2$, and more generally in a
polynomial neighbourhood $0 \leq \im z \lesssim_{\|u\|_{H^s}}
(1+|\real z|)^{-2s}$ of the real line if $u \in H^s$ with $-1/2 < s < 0$.
In the limiting case $s = -1/2$, smallness of $u$ in $l^2 DU^2$
guarantees the localization of the poles in   $0 \leq \im z \ll
(1+|\real z|)$.  

A key difference between real and nonreal $z$ is that for real $z$,
one essentially needs $u\in L^1$ in order to define the scattering
data $L(\xi)$ and $T(\xi)$ in a pointwise fashion. This restricts the
use of the inverse scattering transform to localized, rather than
$L^2$ data. On the other hand, for $z$ in the open upper half space it
suffices to have some $L^2$ type bound on $u$ in order to define
$T(z)$.

Reconstructing $u$ from the scattering data requires solving a
Riemann-Hilbert problem, see \cite{MR1207209} for this approach for
the modified Korteweg-de Vries equation. 

\subsection{Symmetries}

  The main symmetries are multiplication by a phase, translations in $x$, modulations resp. translations in frequency, and scaling. We define them simultaneously on distributions by
  \[   f \to   U_{\theta, \xi_0, x_0,  \lambda} f =  e^{i \theta + i x \xi}   \lambda f( \lambda(x-x_0)).   \]
  This fixes a representation on a central  extension of the Heisenberg group, a notion which we do not use in the sequel. On the Fourier side
  \[     \mathcal{F}\Big(  e^{i \theta + i x \xi_0}   \lambda f( \lambda(x-x_0))\Big)(\xi)  =  e^{i(\theta+ x_0 \xi_0)}  e^{-i x_0 \xi}          \hat f(\lambda^{-1}( \xi -\xi_0)).    \] 
We compute the  effect of the symmetries on the Lax operator and $z$ waves: 
 \[
 i \left(\begin{matrix} \partial_x & - e^{i\theta} u \\ - \overline{e^{i\theta}  u}  & - \partial_x  \end{matrix}  \right) 
 = \left( \begin{matrix} e^{i\theta/2}    & 0 \\ 0 & e^{-i\theta t/2} \end{matrix} \right)
    i \left(\begin{matrix} \partial_x & - u \\ - \bar u  & - \partial_x
    \end{matrix} \right)
    \left( \begin{matrix} e^{-i\theta t/2}    & 0 \\ 0 & e^{i\theta t/2} \end{matrix}  \right).
\]
    Let $V(h)f= f(x-h)$. Then
\[
 i \left(\begin{matrix} \partial_x & -  V(h) u \\ - \overline{ V(h) u}  & - \partial_x  \end{matrix}  \right) 
 = V(h) 
    i \left(\begin{matrix} \partial_x & - u \\ - \bar u  & - \partial_x
    \end{matrix} \right) V(-h).  
\]
Also,    
\[
 i \left(\begin{matrix} \partial_x & - e^{i x\xi } u \\ - \overline{e^{-i x\xi}  u}  & - \partial_x  \end{matrix}  \right) 
 = \left( \begin{matrix} e^{ix \xi/2}    & 0 \\ 0 & e^{-i x\xi/2} \end{matrix} \right)
  \left[   i \left(\begin{matrix} \partial_x & - u \\ - \bar u  & - \partial_x
    \end{matrix} \right) -  \left( \begin{matrix}  \xi/2 & 0 \\ 0 & \xi/2\end{matrix}   \right)    \right] 
    \left( \begin{matrix} e^{-ix\xi/2}    & 0 \\ 0 & e^{ix \xi/2} \end{matrix}  \right), 
\] 
and with $R(\lambda)f(x) = f(\lambda x)$, the scaling by $\lambda>0$
acts as follows:
\[
  i \left(\begin{matrix} \partial_x & -  \lambda u(\lambda x)  \\ - \overline{
      \lambda u(\lambda x) }  & - \partial_x  \end{matrix}  \right) 
 = 
    i  \lambda R_\lambda    \left(\begin{matrix} \partial_x & - u \\ - \bar u  & - \partial_x
    \end{matrix} \right) R_{\lambda^{-1}}.    
    \]
    We obtain
    \begin{lemma} We define
      \[ \tilde U_{\theta, \xi, x_0,\lambda}\left( \begin{matrix} \psi_1 \\ \psi_2 \end{matrix} \right)
      = \lambda^{-1/2}   \left( \begin{matrix} e^{-i(\theta+\xi x_0)/2} \psi_1(\lambda^{-1}(.-x_0)) \\ e^{i(\theta+\xi x_0)/2}
        \psi_2(\lambda^{1}(.-x_0))  \end{matrix} \right). \] 
      Then
 \[       i \left(\begin{matrix} \partial_x & -   U_{\theta, \xi_0, x_0, \lambda} u  \\ - \overline{
     \lambda U_{\theta, \xi_0, x_0, \lambda} u(\lambda x) }  & - \partial_x  \end{matrix}  \right) \Psi
 =    i      \left(\begin{matrix} \partial_x & - u \\ - \bar u  & - \partial_x
    \end{matrix} \right)   \tilde U_{\theta, \xi_0, x_0, \lambda} \Psi 
- \xi/2 \tilde U_{\theta, \xi_0, x_0, \lambda} \Psi.  
\] 
Moreover
\begin{equation}    T(U_{\theta, \xi_0, x_0, \lambda} u, z ) = T(u, \lambda^{-1} (z -\xi_0/2)) \end{equation}  
\end{lemma} 
The last equation expresses the mismatch between transformations of Fourier variables and spectral variables.

\subsection{The transmission coefficient in the upper 
half-plane and conservation laws} 

Our construction of fractional Sobolev conserved quantities in \cite{MR3874652} relies essentially on the 
fact that the transmission coefficient $T$ is preserved along both the NLS and mKdV 
flows. In principle this gives us immediate access to infinitely many conservation
laws, but the question is whether one can relate (some of) them nicely to the standard scale 
of Sobolev spaces. 

If $u$ is a Schwartz function then $\ln |T|$ is a Schwartz
function on the real line, and has a  Taylor expansion 
\begin{equation}\label{expand-f} 
\ln T(z) \approx  \frac1{2\pi i}   \sum_{j=0}^{\infty}  H_{j}(2z)^{-j-1}.
\end{equation}
If $T$ has no poles ith the upper half-space then by the residue theorem 
the conserved energies $H_j$ can be expressed in terms of the 
values of $T$ on the real axis,
\begin{equation}  \label{energy} 
H_k=\int \xi^k  \ln |T(-\xi/2)|d\xi.   
\end{equation} 
However, $\ln T$ may have poles in the upper half plane, and the right hand side  in the
formula \eqref{energy} above has to be modified to account for the residues at the
poles. Precisely, if the poles of $T$ are located at $z_j$ with multiplicities $m_j$
then the counterpart of the relation \eqref{energy} is
\begin{equation}  
H_k= \int \xi^k  \ln |T(-\xi/2)|d\xi + 2 \sum_j  \frac{1}{k+1} m_j \im (2z_j)^{k+1}. 
\label{energy-f}   
\end{equation} 
This is clear if $T$ has finitely many poles away from the real line, but can also be
 justified in general by  interpreting the trace of $\ln|T|$ on the real line as a non-negative measure.

More generally,  for any function  $\eta : \R\to \R$  the expression
\[
 \int \eta(\xi) \ln |T(-\xi/2)|\,   d\xi 
\]
is formally conserved.  

Thus a natural candidate for a fractional Sobolev conservation law 
may be obtained by choosing any (real) function $\eta$ so that 
\[
\eta(\xi) \approx (1 + \xi^2)^s.
\]

However, there are two issues with such a general choice. First,
it is quite difficult to get precise estimates for $\log |T|$ on the
real line without assuming any integrability condition on
$u$. Secondly, in the focusing case such a choice would still miss the
poles of the transmission coefficient.

To remedy both of these issues, it is natural to use much more precise real weights which 
have a holomorphic extension at least in a strip around the real line. Our choice in \cite{MR3874652} was
to use the weights 
\[
\eta_s (\xi) = (1 + \xi^2)^s, \qquad s > -\frac12.
\]
which not only have the appropriate size on the real axis, but can also be extended as holomorphic functions to the subdomain $D = U \setminus i[1,\infty)$ 
of the upper half-space $U$.  

In the absence of poles for $T$ in the upper half-space one can formally  define the 
conserved energies by 
\[
E_s (u)=  \ \int (1+\xi^2)^{s} \real \ln T(\xi/2) d\xi.
\]
By Cauchy's theorem the integral can be switched to the half-line 
$i[1,\infty)$ to give
\begin{equation}\label{deff+} 
\begin{split}  
E_s(u) = 4\sin(\pi s)\! \int_1^\infty \! (\tau^2 -1)^ s\Big[\!-\!\real  \ln T(i\tau/2)  +\frac1{2\pi} 
\sum_{j=0}^N (-1)^j  H_{2j} \tau^{-2j-1}\Big] d\tau 
+ \sum_{j=0}^N \binom{s}{j}  H_{2j} 
\end{split} 
\end{equation} 
Here the conserved integer energies $H_{2j}$ are used to remove the
leading terms in the expansion at infinity, which is needed in order 
to insure the absolute convergence of the integral. 

One key advantage to switching the integral into the upper halfspace
is that the transmission coefficient is more robust there, depending
only on Sobolev norms of $u$. For this reason, in \cite{MR3874652} we adopt the 
formula \eqref{deff+} as the definition of the conserved energy $E_s$. 

This works also in the case when the transmision coefficient $T$ has poles in the upper half-space,
Then  $T$ may have only finitely many poles on the half-line $i[1,\infty)$. 
We also note the  role played by the smallness condition  for $u$ in $l^2 DU^2$,
which is present in Theorem~\ref{energies}. This guarantees that $T$ has a convergent 
multilinear expansion on the  half-line $i[1,\infty)$, and in particular  has no poles 
  there.


\section{The B\"acklund transform} 
\label{s:backlund} 

The central object in this section is the intertwining operator, which is related to the one in the work of Cascaval, Gesztesy, Helge and Latushkin \cite{zbMATH02209608}. We also heavily exploit complex differentiability here, which brings in the tools of complex analysis. Unfortunately the dependence on the state $u$ is not holomorphic, since the complex conjugate occurs in the Lax operator.  To rectify this, it turns  out to be useful to relax the relation between the off-diagonal entries of the Lax operator and to consider the generalized spectral problem
\begin{equation} \label{eigen}    
\psi_x = \left(\begin{matrix} -iz & u_1  \\ -u_2 & iz \end{matrix} \right)\psi. 
\end{equation}

\subsection{Regularity of Jost solutions} 
 We characterize the regularity of the Jost functions  in the following  summary of results  of \cite{MR3874652}, where, for the left Jost function,  we solve the system of integral equations
\[ 
\begin{aligned}
\phi_1(x) = & 1+ \int_{-\infty}^x  u_1(y) \phi_2(y) dy,
\\
\phi_2(x) =& \int_{-\infty}^x e^{2iz(x-y)}   u_2(y)  \phi_1 dy.
\end{aligned}
\]   
for the renormalized functions
\[
(\phi_1,\phi_2) = e^{iz x} (\psi_{l,1},  \psi_{l,2}).
\]

\begin{lemma}\label{hsjost} Let ${\bf u}= (u_1,u_2) \in H^s$, $s>-\frac12$, and   $\im z >0$.
Then the left Jost function $\psi_l$ satisfies
\[   
(e^{iz x} \psi_l)' \in H^s,\qquad  e^{izx} \phi_{l,2} \in H^{s+1}, 
\] 
\[ \lim_{x\to -\infty} e^{izx} \psi_l = \left( \begin{matrix} 1 \\ 0 \end{matrix} \right), 
\] 
\[
\lim_{x\to \infty} e^{izx} \psi_{l,1} = T^{-1}(z). 
\] 
Moreover, 
\[
\Vert (e^{izx} \psi_{l,1})' \Vert_{H^s} + \Vert \psi_{l,2} \Vert_{H^{s+1}}
\le c \Vert (u_1,u_2) \Vert_{l^2DU^2} (\Vert u_1 \Vert_{H^s} + \Vert u_2 \Vert_{H^s}).
\] 
The map $(u_1,u_2,z) \to e^{izx} \psi_l $   is holomorphic in $z$, $u_1$ and $u_2$ with all derivatives bounded by
  \[ c( \Vert (u_1,u_2) \Vert_{l^2DU^2} )(1+ \Vert u_1 \Vert_{H^s} + \Vert u_2 \Vert_{H^s} ) \] 
for $z$ in a compact region in the upper half plane. 
  The differential of $e^{izx} \psi_l$ at ${\bf u}=0$ is given by
  \[ 
  \left( \begin{matrix}  0 \\  \int_{-\infty}^x e^{2iz(x-y)} u_2 (y) dy \end{matrix} \right),  
  \]
and the differential  of $e^{-izx} \psi_r$ by 
\[  \left( \begin{matrix}
   \int_{x}^\infty e^{-2iz(x-y)} u_1 (y) dy \\ 0 \end{matrix}  \right).  
  \]
  The map
  \[ 
  l^2 DU^2 \times l^2 DU^2 \ni (u_1,u_2,z) \to 1/T(z)= W(\psi_l,\psi_r)  \in \C 
  \]
  is holomorphic with derivatives bounded by  $C(\Vert u_1 \Vert_{l^2DU^2} , \Vert u_2 \Vert_{l^2DU^2} ) $ for $z$ in a compact domain of the upper half plane.
  The expansion of $T^{-1}$ at ${\bf u}=0$ is given by
  \begin{equation} \label{expansionofT}  T^{-1}(z) = 1 - \int_{x<y} e^{-2iz(x-y)} u_1(y)u_2(x) dy dx + O( \Vert (e^{2i\real zx} u_1,e^{-2i\real z x} u_2) \Vert_{l^2_{\im z} DU^2}^4) . \end{equation}  
   \end{lemma}

\subsection{The spectrum of the Lax operator, 
wave functions  and eigenfunctions} 

We consider the scattering transform for the focusing nonlinear
Schr\"odinger equation. The first equation of the zero curvature
formulation is
\begin{equation} \label{eigenl}    \psi_x = \left(\begin{matrix} -iz & u  \\ -\bar u & iz \end{matrix} \right)\psi . \end{equation}               
Here we take $u\in l^2 DU^2$.  
Its transmission  
coefficient is a meromorphic function $T$ in $\C \setminus \R$,  with  singularities (poles) at
eigenvalues $z$ of the Lax operator
\begin{equation} \label{lax} 
\psi \to \mathcal L\psi= \mathcal L(u) \psi=  i\left( \begin{matrix} \partial  & -u \\ -\bar u &  -\partial \end{matrix} \right) \psi, 
\end{equation}  
or equivalently, if there is an $L^2$ function $\psi$ which satisfies
\eqref{eigen}. The operator $\mathcal L$ is not selfadjoint, however it satisfies the conjugation relation
\begin{equation}\label{M0} 
\mathcal L^* = M_0 \mathcal L M_0^{-1}, \qquad M_0 = \left(\begin{matrix} 1  & 0 \\ 0 & -1 \end{matrix} \right). 
\end{equation} 
Hence if $z$ is an eigenvalue for $L$ with eigenfunction $\phi$, then it is also an eigenvalue for 
$\mathcal L^*$ with eigenfunction $M_0 \phi$. 

On the other hand,  by conjugation it follows that $\bar z$ 
is an eigenvalue for both $\mathcal L$ and $\mathcal L^*$, with eigenfunctions
$M \bar \phi$, respectively $M_0 M \bar \phi$, where 
\begin{equation} \label{M} 
M = \left(\begin{matrix} 0 & 1 \\ -1 & 0 \end{matrix} \right) 
\end{equation} 
We recall that if  $\Vert u \Vert_{l^2DU^2}$ is small then the eigenvalues all satisfy 
\[  
\im z \le \varepsilon \langle \real z \rangle, 
\]
see Corollary 5.11 of \cite{MR3874652}, or,  by \eqref{expansionofT} and
\[ \Vert e^{-i\real  z x}  u \Vert_{l^2_{\im  z}DU^2} \lesssim \frac{\langle \real z \rangle}{\im z} \Vert u \Vert_{l^2_{\im z} DU^2} \lesssim  \frac{\langle \real z \rangle}{\im z}\langle  1/\im z \rangle^{1/2}  \Vert u \Vert_{l^2_{1} DU^2}.
\]

For $z$ in the upper half-space we seek to describe solutions to
$(\mathcal{L}(u)-z) \phi = 0$, which form a two dimensional vector space; these are
called {\em wave functions}.  For this it is convenient to use the
left and right Jost functions $\psi_l$, $\psi_r$ which decay
exponentially at $-\infty$, respectively $+\infty$.  The transmission coefficient is defined to be the meromorphic function in the upper half plane given by the inverse of their Wronskian,
\[
T(z) = (\det( \psi_l, \psi_r) )^{-1}. 
\] 
Then we
distinguish two scenarios:

\begin{itemize} 
\item $z$ is not an eigenvalue. Then $\psi_l$ and $\psi_r$ are linearly independent,
 and form a basis in the space of solutions.
\item $z$ is an eigenvalue. Then $\psi_l$ and $\psi_r$ are linearly dependent,
and both are  eigenfunctions.
\end{itemize}

In the first case we want to parametrize the wave functions which are
unbounded as $x\to \pm \infty$, up to the multiplication by a complex
number. We parametrize them by  
\begin{equation} \label{wavepara} 
 \psi = T(z) \left( e^{-i (\beta_0 +\beta_1 z) }   \psi_l +  e^{i (\beta_0+\beta_1  z)   }   \psi_r  \right), 
\end{equation}
where $\beta_0 \in \R/\pi \Z$ and $\beta_1  \in \R$. 

We can also relate this notation to our notation $\kappa$ for scattering parameters
for solitons, by setting
\begin{equation}\label{def-xt}
i (\beta_0 + \beta_1 z)  =  \kappa =   \im z x_0 +i \theta. 
\end{equation}
Here   $ x_0 \in \mathbb{R}$ and $\theta \in \R / \pi \Z$ can be thought of as the center 
point and the phase associated to $\psi$, and $\kappa$ will be naturally interpreted 
later on as a scattering parameter in the context of the \Backlund transform.

Moving  $u$ along the NLS flow corresponds to moving $\psi$ along the $\mathcal P$ flow.
  It is not difficult to determine the dependence on time of 
  the unbounded wave function parameters $x_0$ and $\theta$
when we evolve wave functions along the $\mathcal P$ flow. We recall that
the leading part of $\mathcal{P}$ is $\left( \begin{matrix}
    2i \partial_x^2 & 0 \\ 0 & -2i\partial_x^2 \end{matrix} \right) $
and hence, for the solution to
\[   
\psi_t =  \mathcal{P} \psi, 
\] 
the leading term near $-\infty$ is 
\[
e^{2itz^2+  i (\beta_0+ \beta_1 z)}   \left( \begin{array}{c} 0 \\ 
e^{izx} \end{array}\right) = 
   e^{\im z (x_0 -4t\real z) + i(\theta +2 t (\real^2 z-\im^2 z)) } 
   \left( \begin{array}{c} 0 \\ 
e^{izx} \end{array}\right),     
\] 
and on the right it is 
\[ 
e^{-2iz^2t - i (\beta_0 + \beta_1 z) }    \left( \begin{array}{c}  
e^{-izx}\\ 0  \end{array}\right) =  e^{- \im z (x_0-4t\real z) - i(\theta+2 t (\real^2 z-\im^2 z)) } 
   \left( \begin{array}{c}  e^{-izx}  \\ 0 
\end{array}\right).  
\]  
Thus 
\begin{equation} \label{t-dep}   
x_0(t) = x_0 - 4 \real  z t  , \qquad \theta(t) = \theta + 2(|\real z|^2 -|\im z|^2)   t . 
\end{equation} 
or with the complex notation, as for the pure soliton, 
\begin{equation}\label{t-dep-c}   
  \kappa(t) = i (\beta_0 + \beta_1 z + 2t z^2). 
\end{equation}

A similar computation can be carried out for the mKdV flow, as well as for all of the other
commuting flows.

Moreover, suppressing the time dependence for the rest of this section  and setting $t=0$, if $\zeta$ is neither real nor an eigenvalue  then the inverse  of $\mathcal{L}(u)-\zeta$ is given by 
\begin{equation}\label{inverse-B} 
\begin{split}  
  (\mathcal{L}(u)-\zeta)^{-1}f(x)  = & T(\zeta)^{-1} 
\Big(  \psi_r(x)     \int_{-\infty}^x -\psi_{l,2}(y)f_1(y)+ \psi_{l,1}(y)f_2(y) \,  dy  
\\  & + \psi_l(x) \int_x^\infty   \psi_{r,2}(y)   f_1(y)- \psi_{r,1}(y) f_2(y)   \,  dy\Big) . 
\\  = &  T(\zeta)^{-1} \left( \psi_l(x) \int_x^\infty M \psi_r \cdot f dy - \psi_r(x) 
\int_{-\infty}^x M \psi_l \cdot f dy \right). 
\end{split} 
\end{equation} 

Similarly we normalize eigenfunctions $\psi$ so that  
\begin{equation} \label{eigenpara}
  \psi = - e^{-i(\beta_0 + \beta_1 z)}    \psi_l = e^{i(\beta_0+\beta_1 z)}  \psi_r. 
\end{equation}  
Together with  \eqref{def-xt}, this allows one to understand $\kappa$, respectively $x_0$ and $\theta$ as scattering parameters, and we interpret heuristically   ``$u$ contains a soliton with scale $ \lambda$, modulation $\xi$, center $x_0$ and phase $ \theta$"  as the statement that the Lax operator has an eigenvalue $z= -\xi/2 + i \lambda$ with scattering parameter $ \kappa$ given by $x_0$ and $ \theta$ through \eqref{def-xt}.     
This will become more clear when we discuss
the \Backlund transform later on.

The multiplicity of eigenvalues is discussed next:

\begin{lemma} Suppose that $\zeta$ is an eigenvalue for $\mathcal
  L(u)$. Then the geometric multiplicity of $\zeta$ is $1$. Let $\psi$
  be a $\zeta$ eigenfunction. Then the algebraic multiplicity of
  $\zeta$ is $1$ if and only if
\begin{equation} \label{nondeg}  \int \psi_1  \psi_2 \, dx \ne 0. \end{equation}  
Further, we have
\begin{equation}\label{nondeg+}
 2i \int \psi_1  \psi_2 dx = \frac{d}{dz} T^{-1}(z).
\end{equation}
 \end{lemma}

\begin{proof} If the geometric multiplicity were $2$ then all
  solutions to \eqref{eigen} were bounded, and hence would decay
  exponentially at $\pm \infty$. This contradicts the fact that near $
  \infty$ there is one characteristic exponent with positive real
  part.  The eigenvalue is simple if the equation 
\[ 
\mathcal{L}(u)\phi  -\zeta \phi = \psi 
\] 
is not solvable in $L^2$. Since, by the Fredholm alternative,  
\[  
(\mathcal{L}^* -\bar \zeta) MM_0 \bar \psi = 0, 
\]  
the above equation  is not solvable iff 
\[ 
  \int \psi_1 \psi_2 dx \ne 0. 
\] 
To verify \eqref{nondeg+} we differentiate the system \eqref{eigen} with respect to the parameter $z$
where $\psi = \psi_l$ is the left Jost function. Denoting $\tilde \psi = \dfrac{d}{dz} \psi_l$, it solves the system
\[
 \tilde \psi_x = \left(\begin{matrix} -iz & u  \\ -\bar u & iz \end{matrix} \right)\tilde \psi - iM_0 \psi_l,
\]
with initial and terminal data
\[
\tilde \psi(-\infty) = -ixe^{-izx} \left( \begin{matrix} 1 \\ 0 \end{matrix}\right), \qquad \tilde \psi(\infty) = e^{-izx} 
\left(\begin{matrix} \partial_z T^{-1}(z) \\ 0 \end{matrix}\right),
\]
since $z$ is an eigenvalue.  We recall that 
\[
T(z)^{-1} = W(\psi_l,\psi_r) 
\] 
Then the relation \eqref{nondeg+} is obtained from 
\[
\lim_{x\to -\infty} W(\psi_l,\tilde \psi)= 0, \quad \lim_{x\to \infty} W(\psi_l,\tilde \psi) = \partial_z T^{-1}(z),
\] 
and 
\[ 
\partial_x W(\psi_l(x), \tilde \psi(x)) = 2i (\psi_l)_1 (\psi_l)_2 
\]  
by the fundamental theorem of calculus. 
\end{proof}

Unbounded wave functions will play a crucial role also in the case when $z$ is an 
eigenvalue. 
Suppose now that $\phi$ is an eigenfunction to the eigenvalue $z$ of $L(u)$. If the wave function $\psi$ to the same eigenvalue $z$ is unbounded on one side then the same is true on the other side, and $\phi$ and $\psi$ are a fundamental
system. We may normalize $\psi$ so that 
\[ \psi \sim e^{-\im z x_0- i \theta} e^{-izx} \left(\begin{matrix} 1\\ 0\end{matrix}\right)  \] 
as $x \to \infty$ and 
\[ \psi \sim e^{\im z x_0+i \theta} e^{izx} \left(\begin{matrix} 0\\ 1\end{matrix} \right) \] 
as $ x\to  -\infty$. In contrast to the previous case (when $z$ is not an eigenvalue), here
$x_0$ and $\theta$ are uniquely determined, and we have the same normalization for all unbounded wave functions. With this convention 
\begin{equation}  W(\phi,\psi) = 1, \label{wronski1} \end{equation}  
and $x_0$ and $\theta$ are the same for both. 


\begin{lemma} \label{l:double}
Suppose that $z$ is an eigenvalue with $\phi$ an eigenfunction and $\psi$ an unbounded wave function to the eigenvalue $z$. Then 
the limit
\[ 
\lim_{\varepsilon\to 0 } \int_{-\infty}^\infty e^{-\varepsilon (x-x_0)^2} (\phi_1\psi_2+\phi_2\psi_1)\, dx 
\] 
exists. The general unbounded wave function is given by 
\[
\psi+ \zeta  \phi. 
\] 
If $z$ is simple then there is a unique wave function so that the limit is $0$.
The limit  defines a bijection between unbounded wave
functions and $\C$. If the eigenvalue has higher multiplicity then it
does not depend on the unbounded wave function. 
\end{lemma} 

As a consequence we obtain a natural parametrization of unbounded wave functions in the case of a simple eigenvalue.

\begin{proof} By Lemma \ref{hsjost}  the limits 
\[ \lim_{\varepsilon\to 0} \int_{-\infty}^{x_0}  e^{-\varepsilon (x-x_0)^2} 
\phi_2 \psi_1 dx \] 
and 
\[ \lim_{\varepsilon\to 0} \int_{x_0}^\infty  e^{-\varepsilon (x-x_0)^2} 
\phi_1 \psi_2 dx \]
exist. Since also $W(\phi,\psi) = \phi_1\psi_2-\phi_2 \psi_1 = 1$ we obtain 
\[ 
 \lim_{\varepsilon \to 0 } \int_{\R} e^{-\varepsilon(x-x_0)^2} \phi_1\psi_2+\phi_2\psi_1 dx =   2\lim_{\varepsilon\to 0} 
\left[\int_{-\infty}^{x_0}e^{-\varepsilon(x-x_0)^2}  \phi_2 \psi_1 dx 
+ \int_{x_0}^\infty e^{-\varepsilon(x-x_0)^2}  \phi_1 \psi_2 dx\right] . \]    
\end{proof} 

Finding a natural parametrization of unbounded $z$ waves is important in the sequel. We will obtain implicitly a natural parametrization also for higher multiplicity.

\subsection{The  intertwining operator}

The main tool in understanding the \Backlund transform is the intertwining operator $\mathcal{D}(u)$. Given $u \in H^s(\R)$ with $s>-1/2$ we recall that $\mathcal{L}(u)$ is the associated Lax operator. Let $z$ be a point in the upper half plane, and $\psi$ an unbounded $z$ wave. Then $\tilde \psi= \left( \begin{matrix} \bar \psi_2 \\ -\bar \psi_1  \end{matrix}\right) $ is a $\bar z $ wave. The intertwining operator is the unique operator 
of the form 
\begin{equation} \label{intertwining} 
\mathcal{D}(u) = \mathcal{L}(u) + A(x) 
\end{equation}  
where $A: \R \to \C^{2\times 2}$ is chosen so that $\mathcal D$ annihilates $\psi$ and $\tilde \psi$. The matrix $A$ is uniquely determined by this requirement. It turns out that there is a unique function $ v \in H^s(\R)$ so that
the intertwining relation
\[
\mathcal{L}(v) \mathcal{D}(u) = \mathcal{D}(u) \mathcal{L}(u) 
\] 
holds. The map 
\[
(u,z,\psi) \to v 
\]
is called the \emph{\Backlund transform}. The construction is remarkable. It can be iterated,  it gives useful formulas for the addition of multiple solitons, it works for multiple eigenvalues and it can be inverted by an intertwining operator based on eigenfunctions instead of unbounded $z$ waves. 
 
 We want to trace the dependence of multiple \Backlund transforms on the data. For that it turns out to be useful to relax the relation between $u$ and $\bar u$, $z$ and $\bar z$, and $\psi$ and $\tilde \psi$: we consider a Lax operator of the form
 \[ 
 i \left( \begin{matrix} \partial &- u_1 \\ - u_2 & -\partial \end{matrix}\right), 
 \]
 two different values   $z_1, z_2\in \C \backslash \R$, and associated $z_j$ waves 
 $\psi_1$ and $\psi_2$.  We define the intertwining operators  - this time on intervals - by the requirement that the intertwining operator is of the form \eqref{intertwining} and it has both $\psi_j$'s in its null space. 
 
 The crucial benefit of this extension  is that the iterated \Backlund transform is easily seen to be invariant under exchanging any set of indices, which immediately implies a regular dependence of the iterated \Backlund transform  on the elementary symmetric polynomials of the $z_j$ in the NLS/mKdV case.

 \bigskip

  We consider a pair of function ${\bf u} =(u_1,u_2)$ and the corresponding Lax operator
\begin{equation} \label{laxbfu} 
  \psi \to \mathcal L\psi= \mathcal L({ \bf u}) \psi=  i\left( \begin{matrix} \partial  & -u_1 \\ -u_2  &  -\partial \end{matrix} \right) \psi.
\end{equation}  
We define $z$ waves in the same fashion as for the Lax operator in the remaining part of this section. 
\begin{definition} We denote the Wronskian by $W( ., .)$. 
  Let $\zeta_1\ne \zeta_2 \in \C \backslash \R $ and  let
  $\psi_j$ be  $\zeta_j$- wave functions associated to $\bf u$ and $I \subset \R$ an open set so that $ W(\psi_1,\psi_2) \ne 0$. 
 We define
  the intertwining operator on $I$ by 
\begin{equation}\label{D-def}
D\psi= D({\bf u}, \bzeta, \bpsi)\psi  = \big(\mathcal L({\bf u}) - \zeta_2\big) \psi  + (\zeta_2-\zeta_1)\frac{W(\psi_2, \psi)}{W(\psi_2,\psi_1)} \psi_1.  
\end{equation}
\end{definition}

It is not hard to determine the kernel of this operator if $I$ is an interval. 
\begin{lemma}\label{first}  
Let  $\bf u$, $\bzeta$, $\bpsi$ and $I$ as above. Then 
\begin{equation}\label{nullspace}  D({\bf u},\bzeta, \bpsi) \psi_j = 0. \end{equation}
The operator $\mathcal{D}$ is symmetric under exchanging the indices,
\begin{equation} \label{conjugate} D({\bf u},(\zeta_2,\zeta_1) , (\psi_2,\psi_1) ) = D(\bf u,\bzeta, \bpsi). \end{equation}
\end{lemma} 

\begin{proof}
  It is easy to see that $\psi_1$ is in the null space. Assuming \eqref{conjugate} we can argue in the same way for $\psi_2$. We turn to the proof of  \eqref{conjugate} and  use  the trilinear algebraic identity
  \begin{equation} \label{sumform} W(\psi_1, \psi_2)\psi_3 + W(\psi_2, \psi_3) \psi_1     + W(\psi_3,\psi_1) \psi_2 = 0.
  \end{equation}
It implies 
  \[ \zeta_2 W(\psi_1,\psi_2)\psi + (\zeta_2-\zeta_1)W(\psi_2, \psi) \psi_1 = - \zeta_1{W(\psi_1,\psi_2)} \psi  + (\zeta_2-\zeta_1)W(\psi_1, \psi) \psi_2. 
\] 
 We divide by $W(\psi_1,\psi_2)$ to obtain \eqref{conjugate}. 
\end{proof}

The next construction   is a crucial piece of the puzzle. We search for a function $ {\bf v} = B({\bf u} , \bzeta, \bpsi)$ so that
\begin{equation} \label{intertwining2}  \mathcal{L}({\bf v}) D({\bf u},  \bzeta,  \bpsi) = D({\bf u},  \bzeta,  \bpsi) \mathcal{L}({\bf u}). 
\end{equation} 
  Both sides are second order operators with identical second order terms. We rewrite both sides of \eqref{intertwining2} as 
\[ (\mathcal{L}({\bf u}))^2 + A_1 \mathcal{L}({\bf u})   + A_0 =
(\mathcal{L}({\bf u}))^2 + B_1 \mathcal{L}({\bf u}) + B_0. \]
where
\[
\begin{split} 
     A_1 -  B_1     
\,  &  =
  i\left( \begin{matrix} 0 & u_1-v_1 \\ u_2-v_2  & 0  \end{matrix} \right)  + \frac{\zeta_2-\zeta_1}{W(\psi_2,\psi_1)}
  \left[ \left( \begin{matrix} 1 & 0 \\ 0 & -1 \end{matrix} \right),   \left(\begin{matrix} -\psi_{2,2} \psi_{1,1} & \psi_{2,1} \psi_{1,1} \\
        -\psi_{2,2} \psi_{1,2} & \psi_{2,1} \psi_{1,2} \end{matrix} \right) \right]\left( \begin{matrix} 1 & 0 \\ 0 & -1 \end{matrix} \right) 
\\ &=  i\left( \begin{matrix} 0 & u_1-v_1- 2i\frac{\zeta_2-\zeta_1}{W(\psi_2,\psi_1) } \psi_{2,1} \psi_{1,1}    \\ u_2-v_2+2i\frac{\zeta_2-\zeta_1}{W(\psi_2,\psi_1) } \psi_{2,2} \psi_{1,2}  & 0  \end{matrix} \right)
\end{split}\]

 Thus   $A_1=B_1$ is equivalent to 
  \begin{equation} \label{defbaecklund}   {\bf v} : = B({\bf u}, \bzeta,  \bpsi):=  {\bf u} + i\frac{2(\zeta_2-\zeta_1)}{W(\psi_2,\psi_1)} \left( \begin{matrix} \psi_{1,1} \psi_{2,1} \\ -  \psi_{1,2} \psi_{2,2} \end{matrix} \right). 
   \end{equation} 
With this choice we see,   using  Lemma \ref{first},  that $\psi_j $, $j=1,2$  are in the null space of the both  sides. 
  Thus $A_0=B_0$, and we have proved the intertwining relation  \eqref{intertwining}.
  This computation motivates the following:
  
  \begin{definition} 
  We define the \Backlund operator $B$ by \eqref{defbaecklund}.
    \end{definition} 

It also leads to the next result:

\begin{theorem}\label{th_commuting} 
  a)
$D(\bf u, \bzeta,  \bpsi)$ maps $z$ waves of $\mathcal{L}(\bf u)$ to $z$ waves of $\mathcal{L}(B(\bf u,\bzeta, \bpsi))$. 
 
 b)  We have 
 \begin{equation}\label{Lveigen}
   \mathcal{L}(B({\bf u}, \bzeta, \bpsi))\frac{1}{W(\psi_2,\psi_1)}\psi_1  =  \zeta_2 \frac{1}{W(\psi_2,\psi_1)}\psi_1 .
   \end{equation}
c)  For all functions ${\bf u}$, pairwise disjoint $\zeta_j$  and $\zeta_j$-waves $\psi_j$, $j=1,2,3,4$,  the commutation relation  
\begin{equation}\label{commuting} \begin{split} 
  & D\Big(B({\bf u}, (\zeta_1,\zeta_2)  , (\psi_1,\psi_2)) , (\zeta_3,\zeta_4), D({\bf u}, (\zeta_1,\zeta_2) ,( \psi_1,\psi_2))(\psi_3,\psi_4)
    \Big) D({\bf u}, (\zeta_1,\zeta_2)_,(\psi_1,\psi_2))\hspace{-11cm} 
   \\ &  = 
D\Big(B({\bf u}, (\zeta_1,\zeta_4)  , (\psi_1,\psi_4)) , (\zeta_3,\zeta_2)  ,
D({\bf u}, (\zeta_1,\zeta_4) ,( \psi_1,\psi_4) )(\psi_3,\psi_2)\Big) D({\bf u}, (\zeta_1,\zeta_4)_,(\psi_1,\psi_4))
\end{split} 
\end{equation}
holds, and hence the iterated \Backlund transform is symmetric in all indices. 
\end{theorem}
\begin{proof}
  Part a) is an immediate consequence of   \eqref{intertwining2}.
  To see  Part b) let $\phi$ be a $\zeta_2$ wave. Then by the definition of the intertwining operator \eqref{D-def} 
 \[ 
   D({\bf u}, \bzeta, \bpsi) \phi =  (\zeta_2-\zeta_1)\frac{W(\psi_2,\phi)}{W(\psi_2,\psi_1)} \psi_1,
 \] 
   where $W(\psi_2,\phi)$ is constant and zero iff $\phi$ is a multiple of $ \psi_2$. We choose $ \phi$ linearly independent from $\psi_2$. Then the right hand side does not vanish. By the intertwining property,
   \[ \begin{split} 
   \mathcal{L}(B({\bf u}, \bzeta, \bpsi)) \frac{1}{W(\psi_2,\psi_1)} \psi_1
   \, &  =  \frac{1}{(\zeta_2-\zeta_1)W(\psi_2,\phi) }   \mathcal{L}(B({\bf u}, \bzeta, \bpsi)) D({\bf u}, \bzeta, \bpsi) \phi
\\ & = \frac{1}{(\zeta_2-\zeta_1)W(\psi_2,\phi) } D({\bf u}, \bzeta, \bpsi) \mathcal{L}(u)  \phi
\\ & = \zeta_1 \frac1{W(\psi_2,\psi_1)} \psi_1.
\end{split} 
\]

   To prove the commutation relation \eqref{commuting} we
  observe that both sides are second order operators with the same
  leading part. All the  $\psi_j$ are in the null space, and hence they are the
  same. 
\end{proof} 

We can now obtain the following inversion result by a simple direct computation:

\begin{lemma}\label{intertwining_symmetry} 
Assume that 
\[
\bf v = B({\bf u},\bzeta,\bpsi),
\]
and let
\[ \tilde \psi_2 = \frac1{W(\psi_2,\psi_1)} \psi_1, \quad \tilde \psi_1 = \frac1{W(\psi_1,\psi_2)} \psi_2.  \]
Then 
  \begin{equation} \label{inversebacklund} 
\bf u = B({\bf v},\bzeta ,  \tilde \bpsi) 
\end{equation} 
and
\begin{equation}\label{inverse}   D({\bf v},  \bzeta, \tilde  \bpsi  )
  D({\bf u}, \bzeta ,\bpsi) = (\mathcal{L}({\bf u}) - \zeta_1)(\mathcal{L}({\bf u} )-\zeta_2).
  \end{equation} 
  Moreover,
  \begin{equation}\label{utov} 
    D({\bf v} , \bzeta,  \tilde \bpsi)\phi 
    = (\mathcal{L}({\bf u})  -\zeta_2) \phi  -2(\zeta_2-\zeta_1)\frac{W(\psi_2,\phi)}{W(\psi_2,\psi_1)} \psi_1 .
    \end{equation} 
\end{lemma}

\begin{proof}
  The identity \eqref{inversebacklund} is  a consequence of \eqref{Lveigen} in Theorem \ref{th_commuting} and of the definition of $D({\bf u},\bzeta, \bpsi)$. 
Both sides of \eqref{inverse} map $z$ waves of $\mathcal{L}({\bf u}) $ to $z$ waves of the same operator  $\mathcal{L}({\bf u})$. The 
kernel of the right hand side is spanned by the $\zeta_1$ and $\zeta_2$ waves of $\mathcal{L}({\bf u}) $. Since every $\zeta_j$ wave is mapped by $D({\bf u}, \bzeta,\bpsi)$ into the null space of $D({\bf v},\tilde \bzeta)$ they also span the null space of the left hand side.  This implies the formula \eqref{inverse}. 
Finally \eqref{utov} is a consequence of \eqref{sumform}.  
\end{proof}

We may iterate the B\"acklund transform as follows. Let $u_1, u_2 \in H^s(\R)$, $s >-\frac12$, $\zeta_{j1}, \zeta_{j2}$, $j= \overline{1,N}$, pairwise disjoint complex numbers,
and associated wave functions $\psi_{j1}, \psi_{j2}$ for $\L({\bf u})$.  On a set where $W(\psi_{11}, \psi_{12})\ne 0$ we apply the corresponding \Backlund transform for 
${\bf u}$ via \eqref{defbaecklund}, as well as transform the other wave functions by  
\[ 
\psi_{j1}^1 = D \psi_{j1} , \qquad \psi_{j2}^1= D \psi_{j2} 
\]
for $ j \ge 2$. Then we repeat the process $N$ times.

By Theorem \ref{th_commuting}, the iterated \Backlund transforms  are symmetric in all the indices - of
course on a set where all the Wronskians are nonzero.
 So it is natural to
seek a direct description for them.  To achieve that we start with the
$N \times N$ matrix $M$ with complex entries
\begin{equation} \label{Mjk}  
M_{jk} = \frac{i W( \psi_{2k} , \psi_{2j-1} )}{\zeta_{2k}  - \zeta_{2j-1}}.
\end{equation} 
We define the map
\[ Q(\phi_1,\phi_2) (\psi) = W(\phi_2,\psi) \phi_1. 
\]

We assume that $M$ is invertible and denote 
 $m = M^{-1}$. Then we have the following:

\begin{theorem} \label{t:dress2}
The following properties hold for the iterated \Backlund transform:

\begin{enumerate}[label=(\alph*)] 
\item \label{casea2} 
The operator $D^N = D^N({\bf u},\bzeta,\bpsi)$ is given by
\begin{equation}\label{Dn2}
D^N  =\left( I + \sum_{j,k=1}^N m_{kj}   Q(\psi_{2j-1}, \psi_{2k})  (\L({\bf u})  - \zeta_{2k}) ^{-1} \right)  \prod_{k=1}^N  (\L({\bf u}) -\zeta_{2k} ). 
\end{equation}
\item  \label{caseb2} 
  The output function $ {\bf v}= \Dress(u,\bzeta,\bpsi)$ is given by
\begin{equation}\label{Dn-state2}
  {\bf v}  = {\bf u}  + 2 \sum_{kj} m_{kj}\left( \begin{matrix}
    \psi_{2j-1,1}   \psi_{2k,1} \\ -\psi_{2j-1,2} \psi_{2k,2}\end{matrix}  \right). 
\end{equation}
\item  \label{caseb2a} 
In particular, the image of a $z$-wave $\psi$ for ${\bf u}$   is a $z$-wave $D^N \psi$ for ${\bf v} $ where
\begin{equation}\label{Dn-wave2}
D^N\psi = \prod_{\ell=1}^N (z-\zeta_{2\ell}) \left( \phi + \sum_{j,k=1}^N \frac{1}{z - \zeta_{2k} }  m_{kj}  W(\psi_{2k},\phi)   \psi_{2j+1} \right),
\end{equation}
provided $ z $ is not equal to one of the $\zeta_{2j}$ - otherwise we swap the odd and even indices. 

\item \label{casec2} 
  The functions 
\begin{equation}\label{eigen-m2}
\phi_{2k} = \sum_{j=1}^N m_{jk} \psi_{2j-1}
\end{equation}
are $z_{2k}$ waves for $\L({\bf v})$, and similarly with odd and even indices swapped.  We obtain the concise formula for the iterated \Backlund transform
\begin{equation} \label{conciseinverse} 
\mathbf{v} = \mathbf{u} +  \sum_{k=1}^N \left( \begin{matrix} \phi_{2k,1} \psi_{2k,1} \\  -\phi_{2k-1,2}  \psi_{2k-1,2}    \end{matrix} \right).  
 \end{equation} 

\end{enumerate} 
\end{theorem}

\begin{proof} \ \ref{casea2} We begin with the product formula, where
  we remark that the operator $D^N$ is an order $N$ nondegenerate
  differential operator acting on $2$ vectors, therefore it admits a
  system of $2N$ fundamental solutions, and is uniquely determined by
  such a system. The iterated \Backlund transform is another $N$-th order operator with the same coefficient of the leading term. The null space of $D({\bf u}, (\zeta_{2l-1},\zeta_{2l}), {\psi_{2l-1},\psi_{2l}}) $ is spanned by $\psi_{2l-1}$ and $\psi_{2l}$.
  Therefore by the
  iteration relation \eqref{commuting} it follows that the functions
  $\psi_j$ , $1\le j \le 2N $ form a fundamental system for $D^N$.
  Hence, it remains to show that the expression in \eqref{Dn2}
  vanishes when applied to $\psi_j$. 

Indeed, we have
\[
\begin{split}
D^N \psi_{2\ell-1} = & \  
\prod_{k=1}^N (z_{2\ell-1}- z_{2k})
\Big( \psi_{2\ell-1}  + \sum_{k,j=1}^N \frac{1}{z_{2\ell-1} - z_{2k}} m_{kj}
  W( \psi_{2k} , \psi_{2\ell-1})  \psi_{2j-1} \Big) 
\\
= & \  
 \prod_{k=1}^N (z_{2\ell-1}-z_{2k} ) \left( \psi_{2\ell-1}  - \sum_{k,j=1}^N m_{kj}  M_{\ell k}   \psi_{2j-1} \right) = 0,  
\end{split}
\] 
and   
\begin{equation}  \label{nullspace2} 
\begin{split} 
  D^N \psi_{2\ell} = &  \prod_{k\ne \ell }  (z_{2k}-z_{2\ell})\big[  (\mathcal{L}({\bf u}) - z_{2\ell})\psi_{2\ell}  +   \sum_{j=1}^N m_{j\ell }  Q( \psi_{2\ell}, \psi_{2j-1})    \psi_{2\ell}\Big]
  \\   = & \prod_{k \neq \ell}  (z_{2k}-z_{2\ell})\sum_{j=1}^N m_{j\ell} W( \psi_{2\ell} ,\psi_{2\ell}) \psi_{2j-1} = 0
\end{split}
  \end{equation} 
since the Wronskian vanishes.

\bigskip \ref{caseb2} Next we verify the formula \eqref{Dn-state2}. For
this we use the intertwining relation
\[
\L({\bf v}) D^N = D^N \L({\bf u}),
\]
where for $D^N$ we use \eqref{Dn2}. The expression on the right admits an expansion in terms of powers of $\L({\bf u}) $,
\[
D^N \L({\bf u})  = \L({\bf u})^{N+1}  + \left ( \sum_{k=1}^N - z_{2k}  + \sum_{j,k=1}^N m_{kj}  Q(\psi_{2k}, \psi_{2j-1})  \right) \L({\bf u})^N + \cdots
\] 
so we compute a similar expansion on the left,
\[
\L({\bf v})  D^N  = \L({\bf u})  D^N + \left(\begin{matrix} 0 & u-v \\ \bar u - \bar v  & 0  \end{matrix}\right) D^N.
\]
We use the expression for $D^N$, commute and identify the coefficients of $\L({\bf u})^N$.
This yields
\[
  \Big[\left( \begin{matrix} 1 & 0 \\ 0 & -1 \end{matrix} \right) , \sum_{j,k=1}^N m_{kj} \left( \begin{matrix} -\psi_{2k,2} \psi_{2j-1,1} & \psi_{2k,1} \psi_{2j-1,1} \\ -\psi_{2k,2} \psi_{2j-1,2} & \psi_{2k,1}  \psi_{2j-1,2} \end{matrix} \right) \Big] =  
\left(\begin{matrix} 0 & u_1- v_1 \\ u_2 - v_2  & 0  \end{matrix}\right),
\]
which leads to the desired formula \eqref{Dn-state2}.

\bigskip \ref{caseb2a} The image of a $z$-wave $\psi$ for ${\bf u}$   is a $z$-wave $D^N \psi$ for ${\bf v} $ by iterated application of Theorem~\ref{th_commuting} (a).
The formula \eqref{Dn-wave2} is a direct consequence of \eqref{Dn2}.

\bigskip

\ref{casec2}  Here we consider the eigenfunction formula \eqref{eigen-m2}.
Let $ \phi$ be a $\zeta_{2\ell} $ wave. Then $\phi$ is mapped to a $\zeta_{2\ell} $ wave. As in \eqref{nullspace} we get 
\[   D^N \phi   = W(\psi_{2\ell}, \phi)   \prod_{k \neq \ell}  (z_{2k}-z_{2\ell})\sum_{j=1}^N m_{j\ell}  \psi_{2j-1}. \] 
Formula \eqref{eigen-m2} follows since the Wronskian is constant, and we may swap the odd and even indices. 

\end{proof}

\subsection{The intertwining operator for NLS} 
Our main interest is in the NLS equation, where
\[
u_1= u, \quad u_2 = \bar u , \quad  z_1=\zeta, \quad z_2= \bar \zeta , \quad \psi_1=\phi, \quad \psi_2= M \bar \phi = \left(  \begin{matrix} -\bar \phi_2 \\ \bar \phi_1 \end{matrix} \right),
\]
and $ \phi$ is a $\zeta$-wave. 
Then the Wronskian is always nonzero,
\[ 
W(\psi_1,\psi_2) = \det\left( \begin{matrix}  \phi_1 &  -\bar \phi_2\\ \phi_2 & \bar \phi_1 \end{matrix} \right) = |\phi|^2, 
\]  
so all the formulas in the previous subsection apply on the full real line.
The intertwining operator becomes 
\begin{equation}\label{D-def2}
D\psi= D(u,\zeta,\phi)\psi  = \Big(\mathcal L(u) -\bar \zeta -2i \im \zeta \frac{\phi  \phi^*}{|\phi|^2} 
\Big) \psi. 
\end{equation}
and the \Backlund transform becomes
\begin{equation}\label{v-def}
v = B(u,\zeta,\phi) :=u +4 \im \zeta \frac{\phi_1 \bar \phi_2}{|\phi|^2}. 
\end{equation}
The wave function $ \phi $ and $ M \bar \phi$ are in the null space of $D(u,\zeta,\psi)$, and
\[ 
D(u,\bar \zeta, M \bar \phi) = D(u,\zeta, \psi) , \qquad B(u,\zeta, \phi) = B(u, \bar \zeta,\bar \phi), 
\]
which can be written out as 
  \begin{equation}\label{utov2} 
    D(B(u,\zeta,\phi),\bar \zeta, \frac{\phi}{|\phi|^2} )
    = \mathcal{L}(u) -\bar \zeta -2i\im \zeta \frac{M\phi(M\phi)^*}{|\phi|^2}.
    \end{equation}
The intertwing  relation \eqref{intertwining2} becomes 
\begin{equation}\label{intertwining3}    \mathcal{L}(B(u,\zeta,\phi)) D(u,\zeta,\phi) = D(u,\zeta, \phi) \mathcal{L}(u) \end{equation} 
and we obtain   
  \begin{equation}\label{Lveigen2}
    \mathcal{L}(B(u,\zeta,\phi))\frac{\phi}{|\phi|^2} = \bar \zeta \frac{\phi}{|\phi|^2}.     \end{equation}

If 
\[
v = B(u,\zeta,\phi),
\]
then we have 
  \begin{equation} \label{inversebacklund2} 
u = B(v,\bar \zeta ,   \dfrac{1}{|\phi|^2} \phi) = B(v, \zeta, \dfrac{1}{|\phi|^2} M \bar \phi).
\end{equation}

We will use the intertwining operator $D$ in two cases:
\begin{itemize}
\item when $\phi$ is a wave function which is unbounded at both ends.
\item when $\phi$ is an eigenfunction.
\end{itemize}
The remaining case when $\phi= \psi_l$ is not an eigenfunction is also
of interest, but not relevant here.

We begin our discussion with the first case.  Let $\psi$ be a wave function for $u$, at the spectral parameter $\zeta$, and  which is unbounded at $\pm \infty$.
We normalize it so that 
\[ 
\lim_{x\to -\infty}  e^{-i\zeta x} \psi_2(x) = e^{\kappa}, 
\]
\[ 
\lim_{x\to \infty}  e^{i\zeta x} \psi_1(x) = e^{-\kappa}. 
\]
for a unique choice
\[ 
\kappa \in \mathbb{C} \backslash (\pi i \Z). 
\]  
We recall that $\kappa$ is uniquely determined if $z$ is an eigenvalue, but 
can be chosen arbitrarily otherwise. 
 
We know that the Lax operator for $v = B(u,\zeta,\psi)$ has an eigenvalue at $\zeta$ with associated eigenfunction 
\[ 
\phi = \frac{1}{|\psi|^2} \left( \begin{matrix} -\overline{\psi_2}\\[1mm] \overline{\psi_1} \end{matrix} \right). 
\]
Then a brief calculation shows that 
\begin{equation}  \label{kappa-backlund}
\phi =  -e^{-\kappa}  \psi_l = e^{\kappa} \psi_r 
\end{equation} 
where $\phi_l$ and $\phi_r$ are the left resp. right Jost function for $v$.

\begin{remark}
This property is what allows us to identify our use of $\kappa$ as a notation for a scattering parameter, in the first section, to the current use of $\kappa$ as a parameter
for the unbounded eigenfunctions.
\end{remark}

It will often be convenient to use the alternative notation 
\[
e^{\kappa} = e^{i (\beta_0+  \beta_1\zeta )},  
\] 
with $ \beta_0, \beta_1 \in \R$.

   \begin{lemma}   
  Let  $\psi_l(u)$, $\psi_l(v)$,  $\psi_r(u)$ and $\psi_r(v)$
  are left resp. right Jost functions for $u$ resp. $v$ to the spectral parameter $z$  and let $ \phi$ be an unbounded $\zeta$ wave. Then
  \begin{equation}
\label{Luwavemap} 
D(u,\zeta,\phi)\psi_l(u)= (z -  \bar \zeta) \psi_l(v),
\qquad  
D(u,\zeta,\phi)\psi_r(u)=
 (z -  \bar \zeta) \psi_r(v),
\end{equation} 
  and, if the Jost functions at $\zeta$ are unbounded (resp. $ \zeta$ is not a pole for $T$, equivalently $\zeta$ is not an eigenvalue), and
\[
\phi =  T(\zeta)^{-1} \Big( e^{-i(\beta_0+\beta_1 \zeta)} \psi_l(\zeta)  + e^{i(\beta_0+\beta_1 \zeta)} \psi_r(\zeta) \Big) ,
\]

   \begin{equation} \frac1{|\phi|^2} \left( \begin{matrix} -\overline{\phi_1} \\[1mm]
      \overline{\phi_2}\end{matrix} \right)   = -  e^{-i(\beta_0+ \beta_1 \zeta)}   \psi_l(\zeta,v) =  e^{i(\beta_0+ \beta_1 \zeta)}  \psi_r(\zeta, v).   
      \end{equation}
Moreover
\begin{equation}\label{inverse2}   D(B(u,\zeta,\phi), \bar \zeta, \frac{\phi}{|\phi|^2} ) D(u, \zeta ,\phi) = (\mathcal{L}(u) - \zeta)(\mathcal{L}(u)-\bar \zeta).
  \end{equation}

  \end{lemma}

\begin{proof} 
Since the operator $D(u,\zeta,\psi)$ becomes $\mathcal L(u)-\bar \zeta$ at infinity, we have 
\[  
\lim_{x\to \infty} e^{izx} (D(u,\zeta,\phi) \psi_l(u))_1  =   (z-\bar \zeta)   T(z,u)^{-1},  
\] 
which, together with the same calculation for the right Jost functions,  implies \eqref{Luwavemap}. The second formula is an immediate consequence.
For the last formula we use \eqref{utov2}.  
\end{proof}

If $\zeta$ is an eigenvalue, then
$\phi_l$ and $\phi_r$ coincide up to a
constant as above, see \eqref{eigenpara}. Hence to characterize the
normalized eigenfunction $\phi$ we can use the properties
in the above Lemma for $\psi_l$ for $x < x_0$, and for $\psi_r$ for $x
> x_0$.   We combine Lemma \ref{hsjost} with the previous constructions.

\begin{proposition}\label{p:wave}  
\begin{enumerate} 
\item  If $\phi$ is unbounded as $x\to \pm \infty$ then 
$D(u,\zeta, \phi) : H^{s+1} \to H^s$ is injective and has closed range of codimension $2$,
with orthogonal complement spanned by $|\phi|^{-2} M_0\phi$ and $ |\phi|^{-2} MM_0\bar \phi$
(which are the $\zeta$, respectively $\bar \zeta$ eigenfunctions of $L(v)^*$). 
Further, $|\phi|^{-2} M \bar \phi$ is a $\zeta$ eigenfunction for $\mathcal L(v)$,
and $|\phi|^{-2} \phi$ is a $\bar \zeta$ eigenfunction for $\mathcal L(v)$, and
\[
(z-\zeta)T(u,z)  = (z-\bar \zeta) T(v,z).
\]
\item  If $\phi$ is an eigenfunction then $D(u,\zeta,\phi) :H^{s+1}\to H^s$ is surjective, with null space  spanned by 
$\phi$ and $M \bar \phi$.  Moreover, 
\[
(z-\bar \zeta) T(u,z) = (z-\zeta) T(v,z). 
\] 

\item \label{waveana} If $\zeta $ is in the resolvent set of $\mathcal L(u)$  and $\phi$
is a $\zeta$ wave function for $\mathcal L(u)$ as in  \eqref{wavepara}
then the maps 
\[
\zeta \times (x_0,\theta) \times u \to D(u,\zeta, \phi) \in L(H^{s+1}, H^s)
\] 
and 
\[
H^s\ni u \to B(u,\zeta, \phi) -u \in H^{s+1} 
\]
are  analytic and separately holomorphic as functions of $\zeta$, $\bar \zeta$, $u$ and $\bar u$, as discussed in the beginning of this section. 
 They and their  derivatives are uniformly bounded on the set 
\[ \{ \delta \langle \real \zeta \rangle < \im \zeta \} \times \R \times (\R/ \pi \Z) \times 
(H^s\cap \{ u : \Vert u \Vert_{l^2DU^2} < \delta/C\}), \]
for some $C>0$. 
\item If $ \zeta$ is an simple eigenvalue of $\mathcal{L}(v) $ with eigenfunction $\phi$ 
  then the maps
\[ v \to \zeta( \mathcal{L}(v)) , \] 
\[   v \to D(v,\zeta, \phi)  \in L(H^{s+1},H^s), \] 
and 
\[ H^s \ni v \to B(v,\zeta,\phi)-v \in H^{s+1} \] 
are  analytic and holomorphic as functions of $\zeta$, $\bar \zeta$, $u$ and $\bar u$.
\end{enumerate}

\end{proposition}

\begin{proof} 
To prove the claims we suppose that $ f \in H^s$ and we study solutions to 
\begin{equation} \label{inversion}   D(u,\zeta, \phi)  \psi  = f. \end{equation}  
Let $ \phi$ be an unbounded wave function. Then 
\[
\lim_{x\to - \infty} \frac{\phi \bar \phi}{|\phi|^2}= \left( \begin{matrix} 0 & 0 \\ 0 & 1 \end{matrix} \right),  
\] 
and the equation at $-\infty $ becomes 
\[ 
\left( \begin{matrix} \partial + iz  & 0 \\ 0 & -\partial+ i\bar z  \end{matrix}\right)  \psi  = f. 
\] 
Hence we obtain the unique solution (if it exists) by integration from $-\infty$. 
Similarly  
\[ 
\lim_{x\to + \infty} \frac{\phi \bar \phi}{|\phi|^2}=
\left( \begin{matrix} 1 & 0 \\ 0 & 0 \end{matrix} \right) \] and we
find the solution (if it exists) solving from $\infty$.  Both
solutions have to coincide at $x=0$, which shows that we can solve
\eqref{inversion} on a set of $f$ of codimension $2$. 

By the previous lemma we know that $|\phi|^{-2} M\bar \phi$ is a $\zeta$ eigenfunction
for $\mathcal{L}(v)$. Then by symmetries, $|\phi|^{-2}  \phi$ is a $\bar \zeta$ eigenfunction
for $L(v)$, and  $|\phi|^{-2} MM_0\bar \phi$, respectively $|\phi|^{-2}  M_0 \phi$
are eigenfunctions for $\mathcal{L}(v)^*$ associated to the eigenvalues $z$, respectively $\bar z$.

To identify the co-kernel we compute the adjoint
\[
D^* = \mathcal L^*(u)  - \zeta + 2i \im z \frac{ \phi \bar\phi^T}{|\phi|^2} =  \mathcal L^*(v) 
 -  \zeta + 2i \im z \frac{  M_0\phi (M_0\phi)^* }{|\phi|^2}.
\]
Inserting the two eigenfunctions above for $\L(v)^*$  in this formula yields the desired 
basis for the kernel of $D^*$.

If $ \phi$ is an eigenfunction then 
\[ 
\lim_{x\to - \infty} \frac{\phi \bar \phi}{|\phi|^2}=
\left( \begin{matrix} 1 & 0 \\ 0 & 0 \end{matrix} \right), 
\]
\[
\lim_{x\to \infty} \frac{\phi \bar \phi}{|\phi|^2}=
\left( \begin{matrix} 0 & 0 \\ 0 & 1 \end{matrix} \right), 
\] 
and every solution to the initial value problem \eqref{inversion} with
prescribed initial data $ \psi(0)= \psi_0$ is in $H^{s+1}$. We obtain
the null space by choosing $f=0$. It is an easy verification that $\phi$ and $M\phi$ 
span the kernel of $D(u,\zeta,\phi)$.
\end{proof}

Of course we can relax the connection and use $\mathbf{u}$ and two nonreal numbers $z_1$ and $z_2$. The crucial additional condition is that the Wronskian
\[     \psi_{11} \psi_{22} - \psi_{12} \psi_{21} \]
does not vanish. This is certainly true if

\[ \Vert u_j \Vert_{l^2(DU^2)} +  \Vert u_2 \Vert_{l^2(DU^2)} < \delta,
\quad |z_1-\bar z_2| < \delta \text{  for some positive number } \delta.\]

\subsection{The \Backlund transform associated to holomorphic  families of wave functions}

Here we introduce a key generalization of the previous discussion
of the \Backlund transform, which will be critical later in the context of
iterated \Backlund transforms. Precisely, starting with an initial 
state $u$ and some $z$ with $\im z > 0$, instead of single 
wave functions we consider a holomorphic family $\psi = \psi(x,\zeta)$ 
of $\zeta$-wave functions, for $\zeta$ near $z$, or more generally 
for $\zeta$ in an open subset of the half-plane and $x $ in an interval. 

In the NLS case $ u_2 = \bar u_1$ we consider holomorphic families of unbounded wave functions on $\R$. Away from the spectrum of $\mathcal{L}(u)$, the next lemma identifies such unbounded wave functions with a holomorphic function $\alpha$ which relate it to the left, respectively the right Jost functions:
\begin{lemma}
\label{l:all-alpha} Let $u \in H^s$, $s>-\frac12$ and $U$ an open set whose  closure is compact in the upper half plane, without eigenvalues for $\L(u)$. Let
  $\alpha$ be a holomorphic function on $U$. Then there exists a
  unique holomorphic family of unbounded wave functions $\psi(u,z)$
  with
  \[ 
  \lim_{x\to -\infty} e^{-izx} \psi_2(x,z) = e^{i\alpha(z)}, 
  \]
  \[
  \lim_{x\to \infty} e^{izx} \psi_1(x,z) = e^{-i\alpha(z)}. 
  \]
The same is true in the case $\bfu=(u_1,u_2)$ without assuming $ u_2 = \bar u_1$. 
\end{lemma}

\begin{proof}
  We define
  \[ 
  \psi(x,z) = T(z) ( e^{-i\alpha(z)} \psi_l(x,z) + e^{i\alpha(z)} \psi_r(x,z) ). 
  \]
Uniqueness is easy to see. 
  \end{proof}

By contrast, at  eigenvalues of $\mathcal{L}(u)$,  holomorphic families of unbounded wave functions $\psi(u,z)$
the Taylor expansions are uniquely determined up to an order given by the multiplicity.

\begin{lemma} 
\label{l:find-alpha}
Let $\zeta$ be a zero of $T^{-1}$ of multiplicity $N$. Then there exist 
$\alpha_j $, $0 \le j < N$  so that for every holomorphic family of unbounded wave functions near $z = \zeta$,  and $\alpha$ defined  as above, 
 we must have 
 \[ 
\alpha_0 - \alpha(\zeta) \in 2\pi \Z, 
\]
\[ \alpha_j = \alpha^{(j)}(\zeta), \qquad 1 \leq j \leq N-1, 
\] 
for $1 \le j \ <N$. Conversely, given $\alpha_j$ there exists a unique polynomial 
$\alpha$ of degree at most $N-1$ satisfying the above conditions,
along with an associated family of holomorphic unbounded wave functions.
  \end{lemma}

\begin{proof}
  Let $\psi_0= \psi_l(\zeta)$ be the $\zeta$ eigenfunction for the
  eigenvalue $\zeta$ of multiplicity $N$. We choose $\psi(0)$
  linearly independent of $\psi_0(0)$. There exists a unique solution to
  \[ 
\psi_x = \left( \begin{matrix} -i\zeta  & u \\ -\bar u & i\zeta \end{matrix} \right) \psi 
\] 
with these initial data. The Wronskian satisfies
$W(\psi_0(0),\psi(0))\ne 0$ and it is constant. Since $\psi_0$ decays
exponentially as $x \to \pm \infty$, it follows that $\psi$ is
unbounded as $x\to \pm \infty$. We obtain a holomorphic family of unbounded wave functions
near $z = \zeta$
by solving
\[
  \psi_x = \left( \begin{matrix} -iz  & u \\ -\bar u & i z  \end{matrix} \right) \psi 
\] 
with the same initial data for $\psi(0,z)$. After multiplication by a
holomorphic function we may assume that
\[
 \lim_{x\to \infty} e^{izx} \psi_1(x,z)=: e^{- i\alpha(z)} 
\]
and 
\[
 \lim_{x\to \red{-} \infty} e^{-izx} \psi_2(x,z)= e^{i\alpha(z) } 
\]
for some holomorphic function $\alpha$ (we chose this normalization instead of the initial condition).

Let $\tilde \psi$ be another  holomorphic family of wave functions near $z = \zeta$. Then
we can represent it as 
\[   
\tilde \psi(0,z) = \lambda(z) \psi(0,z) + \mu(z) \psi_l(0,z) 
\] 
Both sides are solutions and hence this relation holds for all $x$. In
particular, if $\tilde \psi(.,\zeta)$ is unbounded then
$\lambda(\zeta)\ne 0$. Choosing a smaller neighborhood if necessary we divide by $\lambda(z)$ and, by an abuse of notation we obtain
$\lambda(z)=1$ and
\[
  \tilde \psi(x,z) = \psi(x,z) + \mu(z) \psi_l(x,z). 
\]  
Clearly each choice of $\mu$ gives a holomorphic family of unbounded wave functions
near $\zeta$. Now
\[ 
\lim_{x\to \infty} e^{izx} (\psi_1(x,z)+ \mu (z)\psi_{l,1}(x,z) ) =
e^{- i\alpha(z)}  + T^{-1}(z)\mu(z)  
\]
and 
\[ 
\lim_{x\to -\infty} e^{-izx} (\psi_2(x,z)+ \mu(z)\psi_{l,2}(x,z) ) =
\lim_{x\to -\infty} e^{-izx} \psi_l(x,z) = e^{i\alpha(z)}. 
\]
Hence the defining function $\tilde \alpha$ for $\tilde \psi$ is given by
\begin{equation}\label{alpha-vs-talpha}
\tilde \alpha(z) = \alpha(z) + \frac12 \ln (  1+ e^{-\alpha(z)} T^{-1}(z)\mu(z) ).
\end{equation}
where the logarithm exists in a neighborhood  of $\zeta$ since $T^{-1}(\zeta)=0$ by an abuse of notation. Here $T^{-1}$ vanishes of order $N$ at $\zeta$.
Then $\alpha(\zeta)-\tilde \alpha(\zeta) \in 2\pi   \Z$, and   for $0< j < N$ we must have 
\[
\tilde \alpha^{(j)}(\zeta)= \alpha^{(j)}(\zeta). 
\]  
\medskip

Conversely, let $ \hat \alpha$ be a holomorphic function with $\hat \alpha^{(j)}(\zeta) = \alpha^{(j)}(\zeta) $ for $0\le j < N$.
Then
\[ \hat \mu(z) =   \exp(\hat \alpha(z)-\alpha(z) )(1+e^{-\alpha(z)}T^{-1}(z) )^{-2}  \]
yields $\hat \alpha$. 
\end{proof} 

We further remark that the above class of functions $\alpha$ express the order $N$ matching 
between $\psi_l$ and $\psi_r$ at the pole. The Wronskian relation
\[
W(\psi_l,\psi_r) = T^{-1}
\]
shows that at $\zeta$ as in the lemma the vectors $\psi_l$ and $\psi_r$ agree exactly to order $N$
up to a multiplicative factor.  Away from the pole we must have  
\[
\psi(z) =  T^{-1}(z) ( e^{-i\alpha(z)} \psi_l(z) + e^{i\alpha(z)} \psi_r(z)).
\]
So this multiplicative factor is exactly determined by $\alpha$,
\[
e^{-i\alpha(z)} \psi_l(z) = - e^{i\alpha(z)} \psi_r(z) + O((z-\zeta)^N).
\]

\bigskip

The interesting feature of working with a holomorphic family of
unbounded wave functions $\psi(u,z)$ is that we can propagate it
across any associated \Backlund transform in a way that carries full
information. Again we consider the general case, but we also specialize to the NLS case. 

\begin{lemma} \label{l:hol-B}
  Let $I$ be an interval, $ \zeta_1,\zeta_2  \in \C \backslash \R$
  with $\zeta_1 \neq \zeta_2$, $\psi_2$ a $\zeta_2$ wave and $\psi(x,z)$ a holomorphic family of wave functions for $x \in I$ and $z$ in a neighborhood of $\zeta_1$, $\psi_1= \psi(.,\zeta_1)$. 
  Then
  
\begin{equation}  \label{def:psiv}
\psi_v(x,z)= D(u,\zeta_1,\zeta_2,\psi_1,\psi_2) \left\{ \begin{array}{rl} \dfrac{\psi(.,z)-\psi(.,\zeta_1)}{z-\zeta_{1}}  \qquad z \neq \zeta_1
\\ \partial_z \psi(.,z)\qquad  z = \zeta_1 \end{array}\right.   
\end{equation}

is a holomorphic family of wave functions for
 \begin{equation} \label{def:v}
v = B(u,\zeta_1,\zeta_2, \psi_1,\psi_2). 
 \end{equation}
 Let 
\[ 
\phi_1(x)= \frac1{W(\psi_1,\psi_2)} \psi_2\qquad 
\phi_2(x)= \frac1{W(\psi_1,\psi_2)} \psi_1. 
\]
Then
\begin{equation}\label{re:psi,u}
 \psi(.,z) = \frac{1}{z-\zeta_{2}} D(v,\zeta_1,\zeta_2, \phi_1,\phi_2) \psi_v(.,z), \qquad u = B(v,\zeta_1,\zeta_2,\phi_1,\phi_2).
\end{equation}

\end{lemma}

\begin{proof}
  It is obvious that $\psi_v(x,z)$ is holomorphic in $z$. By Theorem
  \ref{th_commuting} $\psi_v(x,z)$ is an unbounded wave function for
  $v$ if $z \ne \zeta_{1}$. By continuity the same is true for
  $z=\zeta_{1}$. A direct calculation shows that $\alpha$ does not
  change. The final assertion about inversion follows by
  \eqref{inverse}.
\end{proof}

 If $I=\R$, $u_2=\bar u_1$, $\zeta_1= \bar \zeta_1$, $\psi_2 =
 \left(\begin{matrix} \bar \psi_2 \\ - \bar \psi_1 \end{matrix}
 \right) $ and if $ \psi$ is an unbounded family parametrized by the
 same holomorphic function $\alpha(z)$ then $\psi_v$ is also
 parametrized by $ \alpha(z)$.

Conversely, we can start from $v$ and $\psi_v$ and recover $u$ and $\psi$:

\begin{lemma}\label{l:hol-B-re}
Let $\psi_v$ be a holomorphic unbounded family of wave functions for $v$, and $\zeta$
an eigenvalue for $\L(v)$. Then with $u$ and $\psi$ defined  by \eqref{re:psi,u}, the 
relations \eqref{def:psiv}, \eqref{def:v} hold.

\end{lemma}

\section{The soliton addition and removal  maps}
\label{s:dress}

In the previous section we have shown how to add one soliton to an
existing state by applying a \Backlund transform with respect to an
unbounded wave function, and, in reverse, how to remove a soliton by
applying a \Backlund transform with respect to an eigenfunction.

The B\"acklund transforms can be iterated to add multiple solitons. Theorem \ref{t:dress2} provides compact formulas provided the matrix $M_{jk}$ is invertible. But this is always true for the focusing case, see Lemma \ref{posdef} below. 

 Our aim in 
this section is to spell out the results of the last section for adding and subtracting mutiple solitons, and to provide algebraic proofs for the properties of the matrices $m$ and $M$. Finally Lemma  \ref{l:trace} will provide a sharp estimate of the uniform  norm of multiple pure solitons.

To keep the analysis simple, for the computations in this section we
only consider the case of distinct eigenvalues (spectral
parameters). The soliton addition map $\Dress$ will add $N$ prescribed
solitons to a given state, and the soliton removal map $\Undress$ will
remove $n$ existing solitons from a given state.

Given an open subset $U$ with compact closure of the complex upper
half-space we define the nondegenerate phase space for $N$-solitons as
\[
\Phase^{N,0}_U = \{ \bfs = (\bfz,\bfk) \in   U^N \times (\C/i\pi \Z)^N; \ z_i \neq z_j\},
\]
where
\[
\bfz = (z_1, \cdots, z_N), \qquad \bfk= (\kappa_1,\cdots,\kappa_N).
\]

\subsection{ The soliton addition map}

We will view the soliton addition map $\Dress$ as a map 
\[
\Dress: H^s \times \Phase_U^{N,0}  \to H^s.
\]
 We denote the output by 
\[
H^s \ni u
\to  v = \Dress(u,(\bfz,\bfk)) \in H^s.
\]
To define it we impose some natural restrictions, namely that the
$z_k$ are not poles for $T_u$. These will be satisfied for instance if
the spectral parameters $\bfz$ are localized in a compact subset of
the upper half-space and $u$ is sufficiently small in $H^s$.

To describe it we start with the $N$ distinct spectral data $\bfz =
(z_1,\cdots,z_N)$ in $U$ and corresponding scattering data $\bfk =
(\kappa_1,\cdots,\kappa_N)$. We denote $\bfs = (\bfz,\bfk)$ the
corresponding element of $\Phase_U^{N,0}$.  We consider the associated
left and right $z$-waves $\psi_{k,l}$ and $\psi_{k,r}$, and use them
to define the unbounded wave functions $(z_j, \psi_j)$ for $v$ which
have spectral parameters $\kappa_j$ (see \eqref{wavepara} and \eqref{def-xt}),
\[
\psi_j = e^{-\kappa_j} \psi_{j,l} + e^{\kappa_j} \psi_{j,r}.
\]

We inductively apply $n$ \Backlund transforms as follows,
\[
\psi_j^{(k+1)} = D(u^{(k)},\psi_{k}^{(k)},z_k) \psi_j^{(k)}, \qquad u^{(k+1)} = B(u^{(k)},\psi_{k}^{(k)},z_k) 
\]
where we initialize 
\[
\psi_{k}^{(1)} = \psi_k, \qquad  u^{(1)} = v
\]
Then we define the soliton addition map as
\[
 \Dress(v,\bfs):= u^{(N+1)}
\]
We will also denote the iterated \Backlund transform as 
\begin{equation}\label{iteration}  
B^N(u,\bfz,\bfk)  = \prod_{k=1}^N D(u^{(k)},\psi_{k}^{(k)},z_k). 
\end{equation} 

We specialize the formulas of Theorem \ref{t:dress2}. 
We start with the
symmetric matrix $M$ with complex entries
\[
M_{jk} = \frac{i \psi_k^{*} \psi_j}{\bar z_k - z_j}
\]
\begin{lemma}\label{posdef}  Suppose that the imaginary part of $z_j$ is positive, that the
  $z_j$ are pairwise disjoint and that the $\psi_j\in \C^2$ are nonzero.
  Then $M_{jk}$ is positive definite.
\end{lemma}
\begin{proof} We define the nonzero  functions
  \[  \Psi_j(t) = e^{iz_jt } \psi_j \in L^2((0,\infty);\C^2). \]
which are linearly independent since $z_j$ are distinct.
  Then $M$ is their Gramian matrix.
  \end{proof} 
  We denote by $m$ the inverse matrix  $m = M^{-1}$. 
Then the formulas in Theorem \ref{t:dress2} take the following form: 
\begin{enumerate} 
\item  
The iterated intertwining operator is given by 
\begin{equation}\label{Dn}
D^N(\bfz, \bfk)  =\left( I + \sum_{j,k=1}^N m_{kj}  \psi_{j} \psi_k^* (\L_v - \bar z_k)^{-1} \right)  \prod_{l=1}^N  (\L_v-\bar z_l). 
\end{equation}
The image of a $z$-wave $\psi$ for $v$   is a $z$-wave $D^N \psi$ for $u$ where
\begin{equation}\label{Dn-wave}
D^N = \prod_{l=1}^N (z-\bar z_l) \left( I + \sum_{j,k=1}^N \frac{1}{z - \bar z_k} m_{kj}  \psi_{j}  \psi_{k}^* \right). 
\end{equation}
\item   The output function $ v= \Dress(u,\bfz,\bfk)$ is given by
\begin{equation}\label{Dn-state}
v = u +  2m_{kj} \psi_j^1   \bar \psi_k^2.
\end{equation}
\item \label{casec} 
  The functions 
\begin{equation}\label{eigen-m}
\phi_j = m_{jk} \psi_k
\end{equation}
are $\bar z_j$ eigenfunctions for $\L_v$ with scattering parameters $\kappa_j$.
Moreover,
\[ 
v = u +   \sum_{j=1}^N  \phi_j^1 \bar \psi_j^2 .    
\]      
\end{enumerate} 

We can represent the iterated \Backlund transform as follows:
\begin{equation}
  \begin{split} 
D^N = & \sum_{j=1}^N \Big(1 - \sum_{k=1}^N  \phi_k \psi_k^* (\L_u -\bar z_k)^{-1}\Big) \prod_{\ell=1}^n (\L_u-\bar z_\ell)
\\ =  & \prod_{l=1}^N (\L_v-\bar z_\ell)\Big(1 -\sum_{k=1}^N (\L_v-\bar z_k)^{-1} \phi_k \psi_k^*\Big)
\\ = & \sum_{j=1}^N \Big(1 - \sum_{k=1}^N  M\bar \phi_k M \psi_k^t (\L_u -z_k)^{-1}\Big) \prod_{\ell=1}^n (\L_u-z_\ell).
\end{split} 
\end{equation}

\subsection{ The soliton removal map}

We will view the soliton removal map $\Undress$ as a map 
\[
\Undress: \V_U^{N,0}  \to H^s \times \Phase_U^{N,0}. 
\]
 We denote the output by 
\[
H^s \ni v \to   \Undress(v) = (u,\bfz,\bfk) \in \Phase_U^{N,0}  \times H^s.   
\]
To define it we again impose some natural  restrictions on $v \in H^s$, namely 
we select an open subset $U$ inside the upper half-space with compact closure,
and assume that the transmission coefficient $T_v(z)$ has exactly 
$N$ simple poles $z_j$ in $U$.

Now the spectral parameters $\bfz$ are defined as the poles of $T_v$ within $K$.
These will be simple eigenvalues of $\L_v$; then their conjugates $\bar{\bfz}$ will also 
be simple eigenvalues of $\L_v$.

We denote by $(\phi_1, \cdots, \phi_N)$ a corresponding set of
eigenfunctions for $ \bfz$. The scattering parameters $ \bfk$ will be
determined by the relations
\[
\phi_j = -e^{- \kappa_j} \psi_{j,l} = e^{\kappa_j} \psi_{j,r},
\]
comparing the left and right wave functions to the eigenfunction.

Then we define the rest of the soliton removal map $\Undress$ exactly as we have previously 
defined the soliton addition map $\Dress$,  but starting
from $v$ and the $\bfz$ eigenfunctions $(\phi_1, \cdots, \phi_N)$.
To describe the soliton removal map  we start with the symmetric matrix $m$ with entries
\[
m_{jk} = \frac{i \phi_k^* \phi_j}{ \bar z_j - z_k}.
\]
We denote $M = m^{-1}$. Then we have the following formulas for the removal map

\begin{enumerate}
  \item 
The operator $D^N$ is given by
\begin{equation}\label{Dn+}
D^N =\left( I + \sum_{j,k=1}^N M_{kj}  \phi_{j}  \phi_{k}^* (\L_v - \bar z_k)^{-1} \right)  
\prod_{\ell=1}^N  (\L_v - \bar z_\ell). 
\end{equation}
In particular the image of a $z$-wave $\psi$ for $v$   is a $z$-wave $D^n \psi$ for $u$ where
\begin{equation}\label{Dn-wave+}
D^n = \prod_{\ell=1}^n  (z- z_\ell) \Big( I - \sum_{j,k=1}^n \frac{1}{z -  \bar z_k} M_{kj}  \phi_{j}  \phi_{k}^* \Big). 
\end{equation}
\item   The output function $u = \Undress(v,\bfs,\bfk)$ is given by
\begin{equation}\label{Dn-state+}
u = v -   M_{jk} \phi_k^1  \bar \phi_j^2.
\end{equation}
\item  The functions
\begin{equation}
\psi_j = M_{jk} \phi_k
\end{equation}
are $z_j$ wave functions for $\L_u$ with scattering parameters $\kappa_j$.
\end{enumerate}

\subsection{Connecting the two maps}

Here we briefly discuss the relation between the soliton addition and removal maps
in the context of isolated eigenvalues. It follows from the corresponding result for 
single simple eigenvalues that the maps are inverses.

\begin{theorem}
We have 
\begin{equation}
  \Undress \circ \Dress = Id
\end{equation}
for a finite number $N$ of simple eigenvalues. 
\end{theorem}

 The soliton addition and removal
operations are symmetric with the roles of $\phi$ and $\psi$
essentially reversed.  This is a consequence of the construction by
iterative \Backlund transforms, but it is also a consequence of a
purely algebraic relation.

\begin{lemma}
For nonzero  $\psi_j \in \C^2$ nonzero define 
\[
M_{jk} = \frac{i \psi_k^{*} \psi_j}{\bar z_k - z_j}, \qquad m = M^{-1}, \qquad \phi_j = m_{jk} \psi_k.
\]
Then we have the converse relation
\[
m_{jk} = \frac{i \phi_k^{*} \phi_j}{z_k - \bar  z_j}.
\]
\end{lemma}

\begin{proof}
Let $z$ be a complex number, different from the $z_k$ and $\bar z_k$ and considering the $2\times 2$ matrices
\[
\chi = 1 + \sum_{k=1}^n \frac{i \phi_k \psi_k^*}{z- \bar z_k},
\]
\[
\chi^+ = 1 - \sum_{k=1}^n \frac{i \psi_k \phi_k^*}{z- z_k},
\]
we compute 
\[
\begin{split}
  \chi \chi^+ = & \  I + \sum_{k=1}^n  \frac{i \phi_k \psi_k^*}{z- \bar z_k} -
  \sum_{k=1}^n  \frac{i \psi_k \phi_k^*}{z- z_k}
+ \sum_{k,j=1}^n \frac{ \phi_k \psi_k^*}{z- \bar z_k} \frac{ \psi_j \phi_j^*}{z-  z_j}
\\ = & \  I + \sum_{k=1}^n  \frac{i \phi_k \psi_k^*}{z- z_k} -  \sum_{k=1}^n \frac{i \psi_k \phi_k^*}{z- z_k}
+ \sum_{j,k=1}^n  \phi_k \psi_k^* \psi_j \phi_j^* \frac{1}{\bar z_k -  z_j} \left( \frac{1}{z- \bar z_k} - \frac{1}{z-z_j}\right)
\\ = & \  I + \sum_{k=1}^n \frac{i \phi_k \psi_k^*}{z- z_k} -  \sum_{k=1}^n \frac{i \psi_k \phi_k^*}{z- z_k}
+ \sum_{j,k=1}^n  M_{jk} \phi_k \phi_j^*  \left( \frac{1}{z- z_j} - \frac{1}{z-\bar z_k}\right)
\\ = & \  I + \sum_{k=1}^n \frac{i \phi_k \psi_k^*}{z- z_k} -  \sum_{k=1}^n \frac{i \psi_k \phi_k^*}{z- z_k}
+ \sum_{j=1}^n   \frac1{z-z_j}  \psi_j  \phi_j^*
- \sum_{k=1}^n \frac1{z-z_k} \phi_k \psi_k^* 
\\ = & \  I.
\end{split}
\]
This implies that $\chi^+ \chi = 1$ which yields
\[
\sum_{k=1}^n \frac{i \phi_k \psi_k^*}{z- \bar z_k} -  \sum_{k=1}^n \frac{i \psi_k \phi_k^*}{z- z_k} =  
-\sum_{j,k=1}^n  \frac{ \psi_j \phi_j^*}{z-  z_j} \frac{ \phi_k \psi_k^*}{z- \bar z_k} = 
-\sum_{j,k=1}^n \psi_j \phi_j^* \phi_k \psi_k^*  \frac{1}{\bar z_k -  z_j} \left( \frac{1}{z- \bar z_k} - \frac{1}{z-z_j}\right).
\]
Identifying the residues  we obtain the relations
\[
\phi_k \psi_k^* = i\sum_{j=1}^n   \frac{  \phi_j^* \phi_k}{  \bar z_k -  z_j} \psi_j  \psi_k^*, 
\]
or equivalently
\[
\phi_k = i\sum_{j=1}^n   \frac{  \phi_j^* \phi_k}{  \bar z_k -  z_j} \psi_j, 
\]
which leads to the desired conclusion.

\end{proof}

For later use we include here another algebraic relation related
to the soliton addition/removal transforms.  Precisely, consider the
Hermitian $2 \times 2$ positive definite matrix
\[
A = \sum_{j,k=1}^nm_{kj}  \psi_j  \psi_k^*
\]
where we note that the difference between $u$ and $v$ is one of the
off-diagonal entries of this matrix.  Rather than trying to bound that
particular entry, we produce a bound for the entire matrix, via its
trace.

\begin{lemma}\label{l:trace}
We have 
\begin{equation}
Tr A = 2 \sum_{j=1}^n \im z_j.
\end{equation}
\end{lemma}

As a consequence, we obtain a uniform bound for $|u-v|$.

\begin{proof}
This lemma seems to have little to do with the context of our problem. 
Writing the trace of $A$ in the form
\[
\tr A = \sum_{j,k=1}^n m_{kj}  \psi_k^*   \psi_j,   
\]
this becomes a statement which only involves the (complex) dot
products of $\psi_j$ and $\psi_k$.  We consider first the case when
$\psi_j$ take values in $\C^n$ assuming that they are linearly
independent. The statement of the lemma follows then by continuity.

To prove the above trace property we  represent the Gram matrix as
\[
( \psi_k^*  \psi_j ) = \diag (z_j) M - M \diag (\bar z_k).
\]
Thus our trace becomes
\[
\begin{split} 
  \tr A = &  \sum_{j,k=1}^n m_{kj} \psi_k^*\psi_j 
\\  
= & \tr_{\C^n}  (  (\diag (z_j) M - M \diag (\bar z_k)) M^{-1} ) = \tr_{\C^n}   (\diag (z_j)-  \diag (\bar z_k))
\\ = &  
  2 \sum_{j=1}^n  \im z_j.
  \end{split} 
\]
We remark that here the $\psi$'s can be in an arbitrary Hilbert space, therefore  the Gram matrix $  \psi_j  \bar \psi_k^t  $ can be any arbitrary symmetric non-negative matrix.
 \end{proof}

\section{The extended soliton addition and removal 
 maps} \label{s:extended}

So far, we have only considered the iterated \Backlund transform
corresponding to isolated eigenvalues, which can be viewed as a smooth
map
\[
\Dress: H^s \times \Phase_U^{N,0} \to H^s, \qquad u \times \bfs \to v.
\]
restricted to states $u$ with no eigenvalues at $\bfz$. The problem
with this setting is that when we endow $\Phase_U^{N,0}$ with the obvious
smooth structure derived from $(\Phase_1)^N$, the soliton addition  map does not
admit a smooth extension to the diagonal with multiple eigenvalues.

Our contention here is that this does not reflect an inherent lack of
smoothness for the soliton addition map at multiple eigenvalues, but rather the
fact that we are using the wrong smooth structure on $\Phase_U^{N,0}$. To
rectify that, our first step is to consider the iterated \Backlund
transform associated to holomorphic families of unbounded wave
functions.  Precisely, we start with
\begin{itemize}
\item A compact set $U$ in the upper half-space,
\item  A state $u \in H^s$ with no eigenvalues in $U$
\item A holomorphic family of unbounded wave functions $\psi(z)$ in a neighbourhood of $U$, associated to $u$, also with associated $\alpha$ as in Lemma~\ref{l:all-alpha}. 
\end{itemize}

Let $z_j\in U$ be pairwise disjoint for $1 \le j \le J$. Let $n_j \ge
1$ for $1 \le j \le J$ and $N= \sum n_j$. By an iterated \Backlund
transform corresponding to the holomorphic unbounded wave function $\psi$,
we add $N$ solitons at $z_j$ with corresponding multiplicities $n_j$. 
By an iterated application of
Lemma~\ref{l:hol-B}, the result only depends on the $z_j$ and
\begin{equation}  
(\partial^k \alpha(z_j))_{1\le j \le J,\  0\le k < n_j}. 
\end{equation} 
We obtain an associated soliton addition map
\[
\Dresso_{\bfz,\psi} : H^s \times U^N \to H^s.
\]

A-priori this map, as a function of spectral parameters $\bfz$, is
smooth, indeed real analytic, as well as symmetric. However the eigenvalues $\bfz$ of the Lax operator are only determined up to permutations, and they 
do  not depend smoothly on the Lax operator in the case of multiplicities.
Instead the elementary symmetric polynomials
\[
s_j = \sum_{k=1}^N z_k^j , \qquad 1 \le j \le N  
\]
depend smoothly on the potential, and we will see that the soliton addition map
is invariant under such a permutation and depends smoothly on the elementary symmetric polynomials. Secondly, we will parametrize the holomorphic wave function $\alpha$
recorded at $\bfz$ via $2N$ real variables $\beta_0, \cdots, \beta_{2N-1}$, which turn out to be naturally associated to the first $2N$ commuting flows.
These enhancements are  the topic of this section.

We seek to study
further its regularity properties as well as its parametrization.  For
this we consider separately the spectral parameters $\bfz$ and the
unbounded wave functions.

\subsection{The spectral data and the characteristic
 polynomial}
The spectral data $\bfz$ for the soliton addition map can be encoded 
  the characteristic polynomial 
\[
P_{\bfz}(z) = \prod_{j = 1}^n (z-z_j) = z^N + \sum_{k = 1}^{N} (-1)^k s_k z^{N-k}
\]
where $\bfs = \{s_k: 0 \le k \le N\}$ are the elementary symmetric polynomials in $\bfz$ with  $s_0=1$.

Since the soliton addition map is symmetric as a function of $\bfz$, it is
natural to seek to view it as a smooth function of $s_k$'s, rather
than separately in each individual eigenvalues. Because of this, on
the space of spectral data $\bfz$ we will not use the product
topology. 

Instead, we will denote the space of spectral data by
$\C^N_{sym}$, and interpret it as the space of unordered $N$-uples of
complex numbers with the smooth topology defined by the elementary
symmetric polynomials $s_k$.

The correspondence between the two topologies
is continuous but not smooth.

\begin{lemma} \label{l:z-vs-s}
  A) Let $\mathbf{z} \in \C^N$. Then
  \[ |\mathbf{s}(\mathbf{z})| \le (1+|\mathbf{z}|)^N \]  
  and
  \[ |\mathbf{z}| \le  \sqrt{2} |\mathbf{s}(\mathbf{z})|. \]  
For all $\mathbf{s}$ there exists $\mathbf{z}$ with $\mathbf{s}= \mathbf{s}(\mathbf{z})$.
  
  B) Let $\mathbf{z}, \mathbf{w} \in \C^N$. Then
  \[ |\mathbf{s}(\mathbf{z}) - \mathbf{s}(\mathbf{w}) | \le  c_N (1+ |\mathbf{z}|+\mathbf{w}|)^{N-1}  |\mathbf{z}-\mathbf{w} |. \]
  C) Let $\mathbf{s}, \mathbf{\sigma} \in \C^N$. Then there exist $\mathbf{z}$ and $\mathbf{w} $  in $\C^N$ with $\mathbf{s} = \mathbf{s}(\mathbf{z}) $,
  $\mathbf{\sigma}= \mathbf{s}(\mathbf{w}) $  and
  \[ |\mathbf{z}-\mathbf{w} | \le  C(|\mathbf{s}|,|\mathbf{\sigma} | )  |\mathbf{s}-\mathbf{\sigma} |^{\frac1N}.    \] 
\end{lemma} 
\begin{proof} The first inequality in A) is immediate with the $l^1$ norm instead of the $L^2$ norm, which implies the bound in the $l^2$ norm. The 
  components of $\mathbf{z}$ are the roots of
  \[ \sum_{n=0}^N s_{N-n} z^n = 0. \] which are contained in the open
  disc with the given radius by the theorem of Gerschgorin.  Part B is an immediate calculation. For
  Part C we study the dependence of roots on the coefficients of a
  polynomial.
\end{proof}

\subsection{The scattering data and  holomorphic families of
unbounded wave functions}

We have seen that a holomorphic family of unbounded wave functions $\psi$ can
be uniquely described (up to a multiplicative constant) via a
holomorphic function $\alpha(z)$; because of this, we will identify the notations
  $\Dresso_{\bfz,\psi}$ and  $\Dresso_{\bfz,\alpha}$.
In the case of distinct eigenvalues
$\bfz$, the associated soliton addition map depends only of
$\kappa_j = i\alpha(z_j)$. Suppose now that we have multiple eigenvalues $z_k$
with multiplicity $n_k$. In view of Lemma~\ref{l:hol-B}, the
associated soliton addition map may depend on
\[
\partial^j \alpha(z_j), \qquad j = 0,n_k-1.
\]
Thus we can use the equivalence relation
\begin{definition}
For two holomorphic functions $\alpha$ and $\tilde \alpha$ we say that 
\[
\alpha = \tilde \alpha \ \ ( \mod P_{\bfz} )
\]
if there exists a holomorphic function $q$ so that
\[
\alpha(z) - \tilde \alpha(z) = q(z) P_\bfz(z).
\]
\end{definition}

Then we can rephrase the above discussion as 
\begin{lemma}
Assume that $\alpha = \tilde \alpha \ \ (\mod P_{\bfz})$. 
Then $\Dresso_{\bfz,\alpha} = \Dresso_{\bfz,\tilde \alpha}$.
 \end{lemma}

This equivalence relation will allow us to replace the holomorphic function 
$\alpha$ by an equivalent polynomial with degree at most $N-1$.

\begin{lemma} \label{l:find-r} Let $P_{\bfz}$ be the characteristic
  polynomial and $\alpha$ a holomorphic function in a neighbourhood of
  $\bfz$. Then there exists an unique polynomial (remainder)
\[
\tilde \alpha(z) = \sum_{k = 0}^{N-1} \alpha_j z^j
\]
so that 
\[
\tilde \alpha = \alpha \ \ (\mod P_{\bfz}).
\]
Furthermore, $\tilde \alpha$ depends holomorphically on the symmetric polynomials $\bfs$.
\end{lemma}
\begin{proof}
We consider a contour $\gamma$ around the zeroes $\bfz$ of $P_{\bfz}$. We must have 
\[
\int_{\gamma} z^j \frac{\alpha(z) -\tilde \alpha(z)}{P(z)}\, dz = 0, \qquad j \geq 0,
\]
but only the first $N$ such relations are independent.
This yields 
\[
\int_{\gamma} z^j \frac{\alpha(z)}{P(z)}\, dz = \int_{\gamma} z^j \frac{\tilde \alpha(z)}{P(z)}\, dz, \qquad j = 0,N-1.
\]
The left hand side is determined by $\alpha$, and depends holomorphically on 
the symmetric polynomials in $\bfz$ (which are the coefficients of $P_{\bfz}$).
This  in turn uniquely determines
the first $N$ coefficients in the Taylor series for $\dfrac{\tilde \alpha(z)}{P(z)}$ at infinity,
which in turn uniquely determines $\tilde \alpha(z)$.
\end{proof}

The $N-1$ degree polynomial $\tilde \alpha$ can be viewed as our generalized scattering
 parameter, and is identified via its (complex) coefficients $\balpha = (\alpha_0, \cdots,\alpha^{N-1})$,
\[
\alpha(z) = \sum_{j = 0}^{N-1} \alpha_j z^j. 
\] 
This is endowed with the smooth topology of $\C^N$.

However, there is also an equivalent alternative choice, which we will
give preference to in this paper. Precisely, instead of working with 
complex polynomials of degree $N-1$, it is sometimes more convenient 
to work with real polynomials of degree $2N-1$. 

If two real polynomials are equal modulo $P_{\bfz}$ then they are also equal modulo 
$P_{\bar{\bfz}}$ so they must\footnote{Here we recall that all the $z_j$'s in $\bfz$
are in the upper halfspace, so $P_\bfz$ and $P_{\bar{\bfz}}$ have no common roots.}
be equal modulo $P_{\bfz} P_{\bar{\bfz}}$. Thus the 
above lemma concerning $\alpha$ is replaced by 

\begin{lemma} \label{l:find-r-beta} Let $P_{\bfz}$ be the characteristic
  polynomial and $\alpha$ a holomorphic function in a neighbourhood of
  $\bfz$. Then there exists an unique real polynomial 
\[
\beta(z) = \sum_{k = 0}^{2N-1} \beta_j z^j
\]
so that 
\[
\beta = \alpha \ \ (\!\!\!\mod P_{\bfz})
\]
Furthermore, $\bbeta:= (\beta_0,\cdots,\beta_{2N-1})$
depends analytically on the symmetric polynomials $\bfs$.
\end{lemma}

There is a (real) linear one-to one connection between $\balpha$ and $\bbeta$,
which is analytic in the symmetric polynomials $\bfs$. Thus the topologies determined by 
the $\balpha$, respectively the $\bbeta$ representations of the scattering 
parameters are equivalent.

\subsection{The smooth soliton parameters  }

Based on the previous discussion, it is natural to define the phase
space $\Phase_U^N$ associated to an open set $U$  with compact closure in the open upper half-space as
\[
 \Phase_U^{N} = \{( \bfz, \bbeta); \bfz \in \C^N_{sym},\ \bfz \subset U, \ \ \bbeta \in \R^{2N} \} 
\]
with the smooth topology given by the symmetric polynomials $\bfs$ for $\bfz$ and
the smooth topology in $\R^{2N}$  for $\bbeta$.  

It is easily seen that we have a smooth  embedding of the $N$ soliton set with pairwise different eigenvalues 
\[
\Phase_U^{N,0} \subset \Phase_U^N,
\]
which is provided by the matching
\[
\kappa_j = i\bbeta(z_j).
\]
Thus, one can view $\Phase_U^N$ as the completion of $\Phase_U^{N,0}$
with respect to the above topology.  Our contention is that this is
the correct smooth parametrization for extending the soliton addition and
removal maps as smooth inverse maps to spectral parameters $\bfz$
with multiplicity.

Another symmetry which is readily seen at the level of $\kappa_j$ is that $\kappa_j$ are only 
uniquely determined modulo $\pi i$. At the level of $\bbeta$, this yields the equivalence relation,
denoted by $A$, defined
\[
\bbeta_1 \equiv \bbeta_2 \quad \text{iff}  \quad \bbeta_1(z_j) = \bbeta_2(z_j) \ (\mod \pi i).
\] 
This relations  is now $\bfz$ dependent.  There are two interesting observations to make:

\begin{itemize}
\item This is a discrete relation, uniformly in $\bfz \in U$. Thus the local smooth topologies 
on $\Phase_U^N$ and $\Phase_U^N/A$ coincide.

\item The dimension of the symmetry lattice depends on the multiplicities in $\bfz$.
This corresponds to some periods approaching infinity as eigenvalues collapse.

\end{itemize}

In this section we carry out the first step of the analysis, and show that

\begin{proposition}\label{p:trace}
The soliton addition map 
\[
(u, (\bfz,\bbeta)) \to v = \Dress(u,\bfz,\bbeta) := \Dresso_{\bfz,\bbeta}(u)  
\]
is one to one in a suitable setting, and commutes with every flow of the NLS hierarchy whenever this flow is a continuous extension of the flow on Schwartz functions. 
Moreover we have the energy relation (trace formula)
\begin{equation}\label{trace}
E_s(v) =  E_s(u) + 2 \sum_{k=1}^N m_k \Xi_s(2z_k)  
\end{equation}
and in particular 
\begin{equation}\label{trace-L2}
\Vert v \Vert_{L^2}^2 = \Vert u \Vert_{L^2}^2 + 2 \sum_{k=1}^N  \im z_k. 
\end{equation}
\end{proposition}

Here for the flow of $\bbeta$ we use the induced linear maps determined by the relations \eqref{nflow-S1} where for $\kappa$ we use the $\bbeta$ representation
$\kappa = i \sum \beta_n z^n$. This is also explicitely spelled out later in \eqref{phase-flow}.

\begin{proof}
  The parameters $\bbeta$ yield a unique well-defined holomorphic
  family of unbounded wave functions $\psi_u = \psi_{u}(z,\bbeta)$
  associated to $u$. Then by iteratively applying $N$ \Backlund transforms
  to the pair $(u,\psi_u)$ corresponding to the eigenvalues $\bfz$, we
  obtain the pair $(v,\psi_v)$ where $\psi_v$ is another
  unbounded wave function with the same parameter $\bbeta$.  By
  Lemma~\ref{l:find-alpha}, it follows that $\balpha$ and thus $\bbeta$ is uniquely
  determined by $v$ modulo $\pi \Z$.  This proves injectivity.

Conversely,  let $v$ be a potential with eigenvalues $\bfz$, possibly with multiplicities. 
Then by Lemma~\ref{l:find-alpha}  applied to each eigenvalue
there exists a unique polynomial $\alpha$ of degree at most $N-1$ 
which generates a family of holomorphic wave functions $\psi_v$ associated to $v$.
Following Lemma~\ref{l:hol-B-re} we successively remove poles while propagating back the 
unbounded family of wave functions, until after $N$ steps we obtain a pair $(u,\psi_u)$
without any eigenvalues for $\L(u)$ in $U$. Then, by Lemmas~\ref{l:hol-B},~\ref{l:hol-B-re},
$(v,\psi_v)$  is the image of $(u,\psi_u)$ through the iterated \Backlund transformation.
The assertion on the norms is an immediate consequence of the trace formula \eqref{deff++}.
\end{proof}

\section{The regularity of soliton addition and removal} \label{s:dress+}

With the setting of the previous section in place, we return to the
question of the regularity of the soliton addition map with
multiplicities. As a consequence we obtain a precise description of the
pure soliton manifolds and the structure of the phase space and its
dynamics.

\subsection{ The results}

Our goal here is to study the regularity properties of the soliton
addition and the soliton removal maps.  We begin by describing our
setup, which requires the following elements:

\begin{itemize}
\item An open subset $U$ of the upper half-space with compact closure
  in the open upper half plane.

\item The set of $N$ tuples $\bfz$ of complex numbers in $U$, up to
  permutations.  We consider it as an analytic manifold with the
  analytic structure given by the elementary symmetric polynomials
  $\bfs$ in $N$ variables, and use the notation $\C^N_{sym}$. Each $\bfz$
 can also be identified with its characteristic polynomial
\[
P_\bfz(z) = \prod_{k = 1}^N  (z - z_k).
\]
and also, equivalently, with the real polynomial $P_\bfz P_{\bar{\bfz}}$.

 \item The associated soliton phase space $\Phase^N_U$ defined by 
\[
\Phase^N_U = \{ (\bfz,\bbeta) \subset \C^N_{sym} \times \R^{2N}, \, \bfz \subset U \}/A.
\] 
We identify $\bbeta$ with the polynomial 
\[ 
\bbeta(z) =  \sum_{j=0}^{2N-1} \beta_j z^j. 
\] 
The set $A$ is a discrete equivalence relation: We  identify
$\bbeta_1$ and $\bbeta_2$ if for $z_j \in \bfz$ we have
\begin{equation}  \bbeta_1(z_j) - \bbeta_2(z_j) \in \pi  \Z
  \text{ and } \bbeta_2^{(m)} (z_j)= \bbeta_1^{(m)}(z_j) \text{ for }
  1\le m < m_j = \text{ the multiplicity of $z_j$. }
\end{equation}
The set $\Phase^N_U$ carries the natural analytic structure defined by
the analytic structure of the $\bfz$ and the Euclidean structure of
the $\bbeta$.  The $n$-th flow acts on $\Phase^N_U$ by 
\begin{equation}\label{phase-flow}
\dot \bfz = 0, \quad \dot  \bbeta(z) =  2^{n-1}  z^n \ (\mod\ P_\bfz P_{\bar{\bfz}}). 
\end{equation}
This can be viewed as Hamiltonian flows on the phase space
endowed with the symplectic form
\[
\omega = \sum_{k=0}^{2N-1} \beta_k \wedge d \sum_{j=1}^{N}  \im z_j^{k+1},
\]
generated by the Hamiltonians 
\[
H_n(\bfz,\bbeta) =  \frac{2^n}{n+1} \im \sum_{j=1}^{N}  z_j^{n+1}.
\]

\item The $N$ pure soliton set $\M^N_U$ of pure $N$ solitons with
  spectral parameters in $U$. There is a natural map
  \[ 
\M^N_U \ni u \to \bfz, 
\]
 whose fibers are denoted by $M_{\bfz }$. 
  \item The space $\V^0_U \subset H^s$ of states with no spectral parameters in $U$,
\[  
\V^0_U= \{ u \in H^s: \sigma( \mathcal{L}(u)) \cap U = \emptyset \}.  
\]

\item The space $\V_U^N \subset H^s$ of $N$  soliton states with
  spectral parameters in $U$,
\[  
\V_U^N = \{ u \in H^s: \#\sigma(\mathcal{L}(u)) \cap U) = N\}. 
\]
There is the obvious natural map 
\[
\V_N \ni u \to \bfz \in \C^N_{sym},
\]
which is easily seen to be analytic.

\item The $m$-th Hamiltonian  of the hierarchy is a sum
  \[ H_m(u) = \sum_{j=1}^{m/2} H_{m,j} (u), \]
  with
  \[
  H_{2m,2}(u) = \int  |u^{(m)}|^2 dx, 
  \]
  \[
  H_{2m+1,2}(u) = \frac1{i} \int u^{(m)}\partial_x
  \overline{u^{(m)}} dx. 
  \] 
  The second index is half of the
  homogeneity in $u$. In $H_{m,j}$ there are $m+2-2j$ derivatives
  distributed over the $2j$ terms.  The $m$-th Hamiltonian defines a
  flow on Schwartz space, and in particular on pure solitons $\M^N_U$.
  
  $H_2$ is the Hamiltonian of the Schr\"odinger equation and $H_3$ is
  the Hamiltonian of mKdV.
\end{itemize}

In this context we consider the soliton addition map
\[
\Dress: \V_U^0 \times \Phase^N_U \to \V_U^N,
\]
and the soliton removal map
\[
\Undress: V_U^N \to V_U^0  \times \Phase^N_U.
\]

In the previous section we have seen that these maps are inverse maps.
Here we study their regularity.  We begin with the soliton addition map.

\begin{theorem} \label{t:dress-reg} a) Let $s > -1/2$. The soliton addition  map
\[  
\V_U^0 \times \Phase^N_U  \ni (u, \bfz,\bbeta) \to   v=\Dress(u, \bfz,\bbeta) \in \V_U^N 
\]
of adding $N$ solitons is smooth, uniformly on compact sets in $\bbeta$, for $\bfz \in K^N$ for a compact subset $K \subset U$.

b) The  soliton addition  map  $\Dress$ is   uniformly smooth globally  in $u$ and $\bbeta$,  for  $u$ restricted to a  bounded set in $H^s$ and $\bfz$ restricted as above. 

\end{theorem} 

This result provides the proper context  to study  the pure $N$-soliton set $\M_U^N$,
which, by Proposition~\ref{p:trace},  can be described as
\begin{equation}
\label{pure-N}
\M_U^N = \Dress(0,\Phase_U^N).
\end{equation}
For this set we will prove Theorem~\ref{uniformSol}, which we restate 
for convenience in the following:

\begin{theorem}\label{t:pure}  Let $ s > -\frac12$, $U$ and $N$ as
  above. Then the pure $N$ soliton set $\M_U^N$ is a uniformly smooth $4N$ dimensional
  Riemannian submanifold of $H^s$. 
\end{theorem}   

The pure $N$ solitons belong to all $H^s$ spaces, and using a different $s$ 
will yield an equivalent Riemannian structure in the metric sense. By construction
there exists a smooth natural  diffeomorphism $\Phase_U^N \to \M^N_U$,
which commutes with the first $2N$ flows. This gives $\Phase_U^N$ a smooth
Riemannian structure.

It is also natural to consider the foliation of $\M^N_U$ relative to the spectral parameter 
$\bfz$. It is not difficult to see that this is a smooth foliation, as we show later 
that $\bfz$ is a smooth nondegenerate function on the entire space $\V_U^N$.
We conjecture the following:

\begin{conjecture}
 The fibers $\M_{\bfz}$ provide a uniformly smooth foliation of $\M^N_U$.
\end{conjecture}

The global structure of $\M_{\bfz}$ depends on the multiplicities of
the spectrum. If all eigenvalues are simple then it is diffeomorphic
to $ (\R \times \mathbf{S}^1)^N$.  Moreover the induced
diffeomorphisms obtained by choosing a function in ${\M}_{\bfz}$ and flowing 
it with the first $2N$ flows are uniformly
smooth.

The topology of the fiber is different if there are multiplicities:
Some of the $S^1$ components became real lines as the spectral
parameters approach multiplicities. As a consequence the smooth
structure defined by the flow maps cannot give a uniformly smooth
parametrization as we approach multiple eigenvalues.

All the complications above occur already for the case of two solitons, which  we will study the two
soliton case in depth in Section \ref{doubleeigenvalue}. The two
soliton manifold with a double eigenvalue $z_0$ is diffeomorphic to
\[ \R^3 \times \mathbb{S}^1. \] 
For instance if $n = 2$, $z_0=i$ and $|\beta_2|+|\beta_3| \gg 1$  then the   soliton distance is about
\[
R+i \theta \approx \log(-2(\beta_2+i\beta_3)).
\] 
Here $R$ denotes the distance between bump locations, and $\theta$ is
the phase shift between the two bumps. Hence $(\beta_2,\beta_3)$ with the Euclidean topology 
cannot uniformly describe the soliton distance. 

\bigskip
We continue with the properties of soliton removal  map:

\begin{theorem} \label{t:undress-reg} a) The soliton removal  map
\[  
\V_U^N \ni v \to  \Undress(v) = (u, \bfz,\bbeta) \in \V_U^0  \times \Phase^N_U 
\]
of removing $N$ solitons is smooth, uniformly on compact sets in
$\bbeta$ and $\bfz$ in a compact subset of $U^N$.  The map commutes with the flows in the same sense as for
the soliton addition map.
\end{theorem}

As a corollary of these two results, we have:

\begin{theorem}  \label{t:diffeomorphism} The soliton addition  map
\[  
\Dress: \V_U^0 \times \Phase^N_U  \to \V_U^N 
\]
is a local diffeomorphism with respect to the smooth structure of $H^s$ for all $s > -\frac12$.
\end{theorem}

A natural question to ask here is whether the soliton addition and removal maps are 
uniformly smooth globally, for $u$ restricted to a bounded set. For this to be meaningful, 
one has to use the phase space $\Phase_U^N$ endowed with the Riemannian metric
induced from the pure $N$-soliton manifold $\M_U^N$. We conjecture the following:

\begin{conjecture}
Identifying $\Phase_U^N$ and $\M_U^N$, the soliton addition and removal maps are  uniformly smooth globally, 
\[
\Dress: \V_U^0 \times \M_U^N \to \V_U^N, \qquad \Undress: \V_U^N \to 
\V_U^0 \times \M_U^N
\]
for $u$ restricted to a bounded set in $H^s$ and the spectral parameters $\bfz$ restricted to a compact subset of $U$.
\end{conjecture}
\bigskip

The remainder of this section contains  proofs of these results, after a preliminary discussion of symmetric functions.

\subsection{Symmetric functions and elementary symmetric polynomials}

\begin{lemma}\label{elementarysymmetric}  Let $U \subset \C$ and $V \subset X$ be open,
where $X$ is a Banach space. Let
$f: U^N\times V\to \C$ be a continuous ($C^k$, $C^\infty$, analytic) function, so that for every $ x\in V$
  \begin{enumerate}
  \item $U^N \ni \bfz \to f(\bfz,x) $ is invariant under permutations.
  \item $U^N \ni \bfz \to p(\bfz,x)$ is holomorphic.
  \end{enumerate}
  Let $ \bfs(U^N) $ be the (open) range under the map to the first $N$ elementary symmetric functions. Then there exists a continuous ($C^k$, $C^\infty$, analytic) function $\tilde f:\bfs(U^N) \times V \to \C$ so that for every $x \in X$
  \begin{enumerate}
  \item $\bfs(U^N) \ni \bfs \to \tilde f(\bfs,x) $ is holomorphic.
  \item $ \tilde f( \bfs(\bfz)) = f(\bfz,x) $.
  \end{enumerate}
  \end{lemma} 
  \begin{proof}
    Let $U^{N,0} \subset U^{N}$ be the set with pairwise disjoint complex numbers. The map
    \[ U^0 \ni \bfz \to \bfs \]
    is locally biholomorphic, in which case the claim is trivial. 

  Let $\mathbf z \in U^N$ be a point where all
  the variables are identical.  The Taylor series for $f$ at the special
  point converges uniformly in a neighborhood. The partial sums up to
  degree $M$ are symmetric polynomials $f_M$ which can be written as
  \[ 
  f_M= q_M( {\bf s}). 
  \]
  The $q_M$'s converge uniformly in a
  neighborhood of $\bfs(\bfz)$ since the same is true for the $f_M$. The limit is a
  holomorphic function $\tilde f$ in a neighborhood of $\bfs(\bfz)$.

  If the $z$'s  are grouped into separated clusters, we can do the same with the elementary   symmetric functions for the clusters:  Let $U_j \subset \C$ be open sets with compact pairwise disjoint closures.  Let $ \{z_n\}_{n\le N}$ be $N$ points in the union of the $U_j$ and $(s_j)_{1\le j \le N}$ be the elementary symmetric polynomials. Let $N_j$ be the number of $z$'s in $U_j$. We claim that the map
  \[ 
  (s_n) \to (s^j_m)_{1\le j \le J, 1\le m \le N_j} 
  \]
  is holomorphic.
 To see that we note that an integration over a suitable contour yields
\[ 
\lambda^l_{m} := \sum_{1\le m \le N_j} (z^j_m)^l    =  \frac1{2\pi i }\int_{\gamma}   z^l  \frac{ f'(z)}{f(z)}  dz   
\]  
where $\{s_n^l\}$ and $\{\lambda_n^l\}$ are algebraically diffeomorphic.

   The inverse is given by 
  \[
  \sum_{n=0}^N (-1)^n s_n z^{N-n}  = \prod_{m=1}^M \sum_{n=0}^{N_m} (-1)^n s^n_m z^{N_m-n}. 
  \]
These definitions immediately carry over to the setting with $X$.

  \end{proof}

\subsection{A first proof of Theorem \ref{t:dress-reg}(a) }

Let $ u$ be as in the theorem, $u_1=u$ and $ u_2 = \bar u$ and $
\bfz_0 \in U$.
 
By Lemma \ref{hsjost}  there exists a neighborhood $ V \in H^s(\R; \C^2)$
of  $\bfu$ and a neighborhood
\[ V_{\bfz}  =  \{ z :  |z -z_{j,0})| < \varepsilon \} \subset    U \] 
such that the map
\[ V \times V_{\bfz} \ni ( v_1,v_2, \bfz) \to (e^{ - \real z x } \psi_l ,  e^{\real z x} \psi_r ) \in (L^\infty \cap D H^s \times H^{s+1}) \times ( H^{s+1} \times L^\infty \cap D H^s) \]
is analytic with uniformly bounded derivatives.  It is an immediate
consequence that the iterated \Backlund transform is  analytic
in the spectral parameters $\bfz$, $\bbeta$ and holomorphic  in
$u$, uniformly for bounded $\bfu$, $\bfz$ and $\bbeta$ in compact sets.  

This is weaker than the statement of Theorem \ref{t:dress-reg} since we claim smoothness in the elementary symmetric polynomials. To see this let
\[ \psi_1(x,z_1)  = e^{-i \alpha(z_1) } \psi_l(x,z_1) + e^{i \alpha(z_1) } \psi_r(x,z_1) \]
and for $ z \in \bar V_{\bfz} $ 
\[ \psi_2(x,z_2) =  \left( \begin{matrix} \overline{ \psi_1^2(x,\bar z} \\ -\overline{ \psi_1^1(x,\bar z) }  \end{matrix} \right). \]

The iterated \Backlund transform is now analytic in $(u_1,u_2)$, $\bfz_1=  \bfz$ and $\bfz_2=  \overline{\bfz}$.  Fix $(u_1,u_2)$, $ \bfz_1$ and $ \bfz_2$ pairwise disjoint.  The set where the Wronskian vanishes, 
\[ \left\{ x:   W(\psi_1(x, z_{1,j} ), \psi_2(x,z_{2,k}) \ne 0 \right\}, \]
is the complement of a finite set. By Theorem \ref{t:dress2} the iterated \Backlund transform is symmetric under symmetric permutations of the $\bfz_1$ and $ \bfz_2$ separately for these values of $x$, and by continuity for all $x$.

By Lemma \ref{elementarysymmetric} the iterated \Backlund transform is a holomorphic function of the elementary symmetric polynomials separately in $ \bfz_1$ and $\bfz_2$\ - more precisely we have to remove a small neighborhood of the set where some Wronskian vanish. The derivatives with respect to the elementary symmetric polynomials are bounded by the derivatives with respect to $\bfz_1$ reps. $\bfz_2$, hence we obtain a smooth dependence on the elementary symmetric polynomials.

\subsection{ The key regularity lemma}

Here we turn our attention to a second proof of Theorem~\ref{t:dress-reg}, where we aim to provide a more algebraic argument. 
This proof depends on a result on holomorphic
functions, which is virtually independent from the problem at hand. 
Let $U \subset \{ z \in \C: \im z >0\} $ be
an open subset with compact closure in the open upper half plane and $
W \subset X$ open. We consider functions $ \psi: U\times W
\to \C^N$, holomorphic in both arguments and 
such that
\[ |\psi(z_1,w_1)| \le 2 |\psi(z_2,w_2)| \]
for $|z_1-z_2| \ll 1$ and $\Vert w_1-w_2 \Vert_X \ll  1$. 
Using the Cauchy integral we deduce that
\begin{equation}  |\partial^\gamma_{z,w} \psi(z,w)| \le C_{|\gamma|}(1+ d^{-|\gamma|} )|\psi(z,w)| \end{equation} 
where $d$ is the distance from $(z,w)$ to the complement of $U\times W$. 
Let $\delta>0$, $\phi_1, \phi_2 : U\times W \to \C $ holomorphic with
\begin{equation}  
|\phi_1(z,w)| \le \delta |\psi(z,w)|\qquad |\phi_2(z,w)|\le \delta |\psi(z,w)|. 
\end{equation} 
Let $U_0\subset U^N$ the subset of pairwise disjoint tuples and
$V_0 = {\mathbf{s}}(U_0)$ resp $V = {\mathbf{s}}(U)$.
We define the Hermitian matrix $M$ by 
\[
M_{jk} = \frac{i \psi^*(z_k,w) \psi(z_j,w)}{\bar z_k -z_j},
\]
and denote by $m$ its inverse. Then we define the function
$g : U_0 \times W \to \C$ by 
\begin{equation}  
g(\mathbf{s},w) = \sum_{j,k=1}^N \phi^*_1(z_j,w) m_{jk} \phi_2(z_k,w), \end{equation} 
where $\mathbf{s}$ denotes the elementary symmetric functions.

\begin{lemma}\label{l:zhol}
  The function $g$ has a unique analytic extension to $U^N \times W$. Moreover
\[ 
|\partial^\gamma_{{\mathbf{s}},w} g(\mathbf{s},w)| \le c_{N,|\gamma|}  \delta(1+d^{-N|\gamma|}).  
\]   
  \end{lemma} 
Here the power of $d$ is controlled via the Cauchy integral and Lemma~\ref{l:z-vs-s}.
\begin{proof}

  At a given point $(\mathbf{z}_0, w_0)$ we divide both $\phi_j$ and $\psi$
  by $|\psi(z_0,w_0)|$ and we may assume that $|\psi(z_0,w_0)|=1$ in the sequel.

  We substitute $\tilde z_j$ for $\bar z_j$ and $\tilde w $ for $\bar  w$ and define
  \[  
M_{jk} = i\frac{ (\overline{\psi(\overline{\bar z_k},\overline{\tilde w} ) } \cdot 
 \psi(z_j,w)}{\tilde z_k-z_j}, 
\]
which is a holomorphic function of $\mathbf{z}, \widetilde
{\mathbf{z}},w$ and $\tilde w$.  Similarly we extend the function $g$.
Let $C$ be the cofactor matrix of $M$. Then we can write
\[
m = \frac{1}{\det M} C^T 
\]
and
\[
g(\bfs,\tilde \bfs,w,\tilde w) = \overline{ \phi^2 (\overline{\tilde z_k)}} m_{kj}\phi^1(z_j) =
 \frac{ \overline{\phi^2 (\overline{\tilde z_k},\tilde w)} C^T_{kj}  \phi^1(z_j,w)}{\det M}.
\]

Now we consider the symmetry properties of both the numerator and the
denominator.  Permuting two values $z_j$ and $z_k$ exchanges to rows
in $M$ and permuting two values $\tilde z_j$ and $\tilde z_k$
exchanges two columns. Under either operation $\det(M)$ changes
sign. Since
\[ M C^T= C^TM = \det(M) 1 \] we see that interchanging $z_j$ and
$z_k$ exchanges two rows in $C$ and the whole sign of $C$. As a
consequence both $\det M$ and
  \[ 
\overline{\phi^2 (\overline{\tilde z_k})} C^T_{kj}  \phi^1(z_j), 
\]
considered as holomorphic functions of $\widetilde{\mathbf{z}} $ and $\mathbf{z}$,
  are antisymmetric in the $z_j$. Then we can smoothly factor
\[
\overline{\phi^2 (\overline{\tilde z_k},\overline{\tilde  w}) } C^T_{kj}   \phi^1(z_j,w) =
G(\bfz,\tilde \bfz,w,  \tilde w )  \prod_{j\ne k}  (z_j - z_k) \prod_{j\ne k}  (\tilde  z_j - \tilde  z_k).
\]
The same applies for $\det M$, 
\[
\det M = H(\bfz, \widetilde{\bfz},w,\tilde w)  \prod_{j\ne k}  (z_j - z_k) \prod_{j\ne k}  (\tilde z_j - \tilde z_k). 
\]
Here the functions $G$ and $H$ are holomorphic functions in $\bfz$ and
$\tilde \bfz$, and separately symmetric in both $\bfz$ and $\tilde
\bfz$ (this is the reason we separated the variables $\bfz$ and $\bar
\bfz$ in the first place).  By Lemma \ref{elementarysymmetric}, every
symmetric holomorphic function is a holomorphic function in the
elementary symmetric polynomials.  Hence by a slight abuse of
notations we will write
\[
G(\bfz,\tilde \bfz,w,  \tilde w ) = G(\bfs,\tilde \bfs,w,  \tilde w ), \qquad 
H(\bfz,\tilde \bfz,w,  \tilde w ) = H(\bfs,\tilde \bfs,w,  \tilde w ).
\]
We still need to divide these two functions $G$ and $H$.
To do this we return to the diagonal $\tilde \bfz = \bar \bfz$, $\tilde w = \bar w$
and claim  that there we have  
\begin{equation}\label{M-pos}
H(\bfz  ,\overline{\bfz},w,\bar w) \gtrsim 1.
\end{equation}
But there we can take advantage of the positivity of $M$. We can write $M$
at $w=\bar z$  as the sum of $M$  positive 
matrices which correspond to the components of the $\phi$'s,
\[
M_{jk} = \sum_{m=1}^M  \frac{i \phi^m(z_j) \overline{\phi^m(z_k)}}{\bar z_k -  z_j} =: \sum_{m=1}^M  M^m_{jk}.
\]
Here, to insure nondegeneracy,  we first rotate the so that all  components
$\phi^m(z_j)$ are 
of comparable size, which we can do if $m \ge 2$.

Then
\[       M^m =  \left( \begin{matrix} \phi^m(z_1) & 0 & \dots & 0 \\
  0 & \phi^m(z_2) &\dots & 0 \\
  \vdots & \vdots & \ddots & \vdots \\
   0  & 0 & \dots & \phi^m(z_M) 
\end{matrix} \right)        \Big(\frac{i}{\bar z_k-z_j}\Big)_{jk} 
\left( \begin{matrix} \overline{\phi^m(z_1)} & 0 & \dots & 0 \\
  0 & \overline{\phi^m(z_2)} &\dots & 0 \\
  \vdots & \vdots & \ddots & \vdots \\
   0  & 0 & \dots & \overline{\phi^m(z_M)}  
          \end{matrix} \right) 
\] 
and
\[ \det M^m = \prod_{j=1}^n |\phi^m(z_j)|^2  \frac{ \prod_{j< k} (z_j-z_k) (\bar z_j-\bar z_k)}{\prod_{j,k} (\bar z_j-z_k )} \]   

 Since we have bound one of these these determinants from below, we can also bound 
from below the determinant of the sum, thereby proving our claim \eqref{M-pos}.

Now we write, again on the diagonal,
\[
g(\bfz,w) = \frac{G(\bfs,\bar \bfs,w,  \bar w )}{H(\bfs,\bar \bfs,w,  \bar w )}
\]
where both the numerator and denominator are smooth and the
denominator is bounded from below. The conclusion of the Lemma follows.

\end{proof}

\subsection{The pointwise regularity of soliton addition}

This is the first part of the second proof of  both part (a) and (b) of  Theorem~\ref{t:dress-reg}, where we consider the pointwise regularity of the soliton addition map.  We consider $u \in \V_U^0$ and 
$(\bfz,\bbeta) \in \Phase^N_U$, and let $v = \Dress(u,\bfz,\bbeta)$.
 For arbitrary fixed $x \in \R$ we study the dependence of $v(x)$ on $u,\bfz,\bbeta$.
Here $v(x)$ is obtained via the following steps:

\begin{enumerate}
\item  Since the transmission coefficient $T(u)$ has no poles in $U$, it follows that the 
left and right Jost  functions $\psi_l(u,z)$ and $\psi_r(u,z)$ are uniformly independent (the Wronskian is the inverse of $T(z)$),
analytic in $u$  and uniformly holomorphic in $z$. 

\item  Given $\bbeta \in \R^{2N}$, we produce a holomorphic family of wave functions by setting
\[
\psi(u,z,\bbeta) = e^{-i \bbeta(z)} \psi_l(u,z) + e^{i \bbeta(z)} \psi_r(u,z).
\]

\item  Based on our interpretation of the soliton addition map in terms of the holomorphic wave function, 
it follows that
\[
{\Dress}(u,\bfz,\bbeta) = \Dresso_{\bfz,\psi} (u). 
\]
The latter expression can be viewed as the outcome of $N$ \Backlund transforms, so 
it is analytic in the $z_j$'s, in $u$ and $\bar u$, and in $\alpha$ (via $\psi$).

\item  If the $z$'s are distinct, then we can use the results in Section~\ref{s:dress} to obtain 
an expression for $v-u$ as 
\[
v - u = \sum_{j,k=1}^N \psi^*_1(z_j,u) m_{jk} \psi_2(z_k,u), 
\]
where $m$ is the inverse of the matrix $M$ given by
\[
M_{jk} = \frac{i \psi^*(z_k,u) \psi(z_j,u)}{\bar z_k -z_j}.
\]

\item Now we use Lemma~\ref{l:zhol} to conclude that this expression
  has an uniformly smooth extension to the diagonal, as a function of
  $\bfs$, $\bbeta$ and $u$. 
\end{enumerate} 

By the above considerations we have defined a soliton addition map $\Dress$
on $\V_U^0 \times \Phase_U^N$ which is  smooth  for fixed $x$.
At this point we are still lacking uniformity both with respect to $x$, $u$, $\bfz$ and $\bbeta$.

We now consider the question of uniformity. To start with, we need to describe 
more accurately the Jost functions $\psi_l$, $\psi_r$.  Consider $\psi_l$ for instance.
Taking out the exponentials, we set $\tpsi_l = e^{iz x} \psi_l$,
for which we have the ode's 
\begin{equation} \label{eq:tpsil} 
 \left\{
 \begin{array}{rl} 
\dot \tpsi_l^1 = &  v \tpsi_l^2
\cr
\dot \tpsi_l^2 = &  2 i z \tpsi_l^2 + \bar v \tpsi_l^1
\end{array}
\right. .
\end{equation} 
The map
\[ 
H^s\ni u \to \tpsi_l^1 \in C_b(\R) 
\]
is uniformly smooth in $u$ (see Lemma \ref{hsjost}), and also uniformly holomorphic in $z$.
In a symmetric way  we set $\tpsi_r = e^{- iz x} \psi_r$ with similar properties.

With these notations, our holomorphic family of wave functions becomes
\[
\psi(u,z,\bbeta) = e^{-i( \bbeta(z)+ z x) } \tpsi_l(u,z) + e^{i (\bbeta(z)+zx)} \psi_r(u,z).
\]

To gain uniformity, we first assume that $\bbeta$ is in a compact set. Then the 
second term above is leading if $x <0$ while the first term is dominant when $x > 0$.
Suppose for instance that $x > 0$. Then we take out the second exponential factor, 
and redefine the holomorphic family of wave functions as
\[
\psi(u,z,\bbeta) =  \tpsi_l(u,z) + e^{2i(\bbeta(z)+2zx)} \psi_r(u,z).
  \]
Our choice of $x$ now insures that for this family we have uniform 
regularity at $x$  with respect to all parameters. Then, by Lemma~\ref{l:zhol},
we obtain the corresponding uniform regularity for $v(x)$, as stated in  part (a) 
of the theorem.

Next we move to part (b) of the theorem. We fix $\bfz$, and consider the question of uniform regularity 
in $u$ and $\bbeta$.  Differentiating $v-u$ we obtain a representation
\[
 \partial_u (v-u)(x) = \partial_u \psi_1^* m \psi_2 +  \psi_1^* m \partial_u \psi_1
-  \psi_1^* m   \partial_u M m  \psi_1,
\]
and similarly for the $\bbeta$ derivatives.  For higher derivatives
with respect to $\bbeta$ and $u$ we obtain a similar but longer
expansion but with more instances of $m$ separated by differentiated
$M$.  For each differentiated $\partial^k M$ (in either $u$ or
$\bbeta$) we can separate the variables $\bfz$ and $\bar \bfz$ (e.g. using
the exponential representation) and represent them as rapidly
convergent sums (integrals) of terms of the form
\[
\partial^{k_1} \psi_j g(z_j) \otimes \partial^{k_2} \psi^*_j g(\bar z_j). 
\]
By Cauchy-Schwartz, it remains to obtain a uniform bound for 
expressions of the form
\[
\partial^k \psi^*_j \overline{g( z_j)} \, m_{jn}\, \partial^k \psi_n {g( z_n)}.
\]

For the two components of $\psi$ we have the regularity
\[
\partial^k \left[e^{\bbeta(z)} \psi_l(u,z) \right]  =
e^{\bbeta(z)+i zx} f^k(z,u),
\]
where $f^k$ is uniformly holomorphic in $z$.

So we need to bound uniformly an expression of the form
\[
\overline{e^{\bbeta(z_j)+i z_jx} f(z_j)}  \,m_{jn}\,  e^{\bbeta(z_n)+i z_nx} f(z_n),
\]
where $f$ is holomorphic.

If we could simply discard the $\psi_r$ component of $\psi$, then we 
would  factor out the phase $ e^{\bbeta(z)+i zx}$ and then just apply 
Lemma~\ref{l:zhol}. As it is, we can still use each of the two components of $\psi$ 
to define its own non-negative matrix $M^1$, respectively $M^2$ so that $M = M^1+ M^2$.
correspondingly we get $m \leq m^1$ and $m \leq m^2$. Applying 
Lemma~\ref{l:zhol} to each of these components, it follows that we could 
bound $m$ on vectors of the form
\[
b_1(z) \psi_1, \qquad b_2(z) \psi_2,
\]
with $b_1$, $b_2$ holomorphic. It remains to see that we can obtain a representation
\[
e^{\bbeta(z)+i zx} f(z) = b_1(z) \psi^1 + b_2(z) \psi^2.
\]
Cancelling phases this is equivalent to
\[
f(z) = b_1(z) ( \tpsi_l^1 + e^{-2\bbeta(z)-2i zx} \tpsi_r^1) + b_2(z) ( \tpsi_l^2 
+ e^{-2\bbeta(z)-2i zx} \tpsi_r^2).
\]
Here we do not want $b_1$ or $b_2$ to depend on the exponentials,
for that would likely make them unbounded. So we strengthen the above relation to a system
\[
\begin{split}
b_1(z)  \tpsi_l^1  + b_2(z) \tpsi_l^2 = & \ f,
\\
b_1(z)  \tpsi_r^1  + b_2(z) \tpsi_r^2 = & \ 0,
\end{split}
\]
This is uniformly solvable since the Wronskian of $\tpsi_l$ and $\tpsi_r$ is 
constant and of size $O(1)$ for $z \in U$, because $u \in \V_U^N$.

\subsection{The \texorpdfstring{$H^s$}{} regularity of soliton addition}

Theorem \ref{t:dress-reg} claims uniform smoothness
for the soliton addition map  as a map to $H^s$ in two contexts,
corresponding to part (a) and part (b).
At this point we know that, in both contexts, for each $x$ the map 
\[
\V_U^0 \times \Phase_U^N   \ni (u, \bfs,\bbeta) \to v(x) - u(x) \in \C 
\]
is smooth, with appropriate uniformity statements. The next step is to
prove similar $H^s$ bounds for $u-v$ and its linearization.  To
achieve this, we divide and conquer.  We split the real axis into unit
intervals, and seek to understand the $H^s$ regularity within each
interval. For a reference point $x$, we study the $H^s$ regularity of
$u-v$ in the interval $I = (x-1,x)$. 

We follow the analysis in the previous subsection, but working on unit 
intervals instead of at a fixed point $x$. One can think of the construction
as having two stages:

\medskip
(i) From the data $u$ to the renormalized Jost functions $\tpsi_l$, $\tpsi_r$.

\medskip
(ii) From the renormalized Jost functions to $v-u$.

\medskip
As long as $u \in \V_U^0$ and $z \in U$, the map 
\[
H^s \ni u \to \tpsi_l, \tpsi_r \in H^{s+1}(I)
\]
 is holomorphic in $z$ and analytic in $u$. Then the same argument as in the previous 
subsection shows that the soliton addition map
\[
\V_U^0 \times \Phase_U^N \ni (u,\bfs,\bbeta) \to v -u \in H^{s+1}(I)
\]
is analytic, with the same uniformity statements as before.

The new difficulty we face here is in the transition from the local
$H^s$ regularity to the global $H^s$ regularity. For this we need to
gain the $\ell^2$ summation with respect to unit intervals.
We remark that this gain is not straightforward, i.e. it does not happen at the level of 
$\tpsi_l, \tpsi_r$. Instead, the best we can say is that we have the localized bounds
\[ 
\Vert \tpsi_l^2 \Vert_{H^{s+1}(x-1,x)} \le c \Vert  e^{\im z (y-x)} u(y) \Vert_{H^s(-\infty,x)} 
\]
and
\[ 
\Vert \tpsi_l^1(x)- \tpsi_l^1(.) \Vert_{H^{s+1}(x-1,x)} \le c\Vert e^{\im z (y-x)} u(y) \Vert_{H^s(-\infty, x) }, 
\]
with an implicit constant depending only on $\Vert u \Vert_{H^s}$.
Similar bounds will hold for the linearizations. Excluding finitely many intervals
where $u$ might concentrate, we can assume that we also have smallness,
\[
\Vert \tpsi_l^2 \Vert_{H^{s+1}(x-1,x)} + \Vert \tpsi_l^1(x)- \tpsi_l^1(.) \Vert_{H^{s+1}(x-1,x)} \ll 1
\]
This in turn gives pointwise smallness, and thus a bound from below
\[
|\tpsi_l^1(x)| \gtrsim 1.
\]
Similarly, we will have 
\[
|\tpsi_r^2(x)| \gtrsim 1.
\]
Hence, on the interval $(x-1,x)$ it is natural to compare the
renormalized Jost functions $\tpsi_l$ and $\tpsi_r$ with $
\tpsi_l^1(x) e_1$, respectively $\tpsi_r^2(x) e_2$.  Thus, within the
interval $I$ we arrive at a reference configuration which corresponds
to a pure soliton.  However, this is not the soliton with parameters
$(\bfz,\bbeta)$; that would correspond to having $\tpsi_l^1 = 1$ and
$\tpsi_r^1 = 1$. Instead $\bbeta$ is readjusted to
\[
\tilde \bbeta(x) = \bbeta + \frac12 \log(\tpsi_l^1(x)/\tpsi_r^2(x)).
\]
This also should be seen as a function of $z$ and $u$.

We denote by $Q_{\bfz,\bbeta}$ the pure soliton with parameters $(\bfz,\bbeta)$.
Then our analysis above allows us to  conclude that we have the localized bound
\[
\| v - u - Q_{\bfz,\tilde \bbeta(x)}\|_{H^{s+1}(I)} \lesssim \| \sech [\delta(y-x)] u(y) \|_{H^s_y}, \qquad 0 < \delta < \min \{ \im z_j\}.
\]
Similar bounds will also hold for the linearization.

To conclude, we need to show that the square summability in $I$ survives 
as we vary the soliton parameter $\tilde \bbeta$. The key property here 
is that $\tilde \bbeta$ does not vary much, $|\tilde \bbeta - \bbeta| \lesssim 1$.
Hence it suffices to verify the property
\[
\sum_{I} \sup_{|\tilde \bbeta - \bbeta| \lesssim 1} \| Q_{\bfz,\tilde \bbeta}\|_{H^{s+1}(I)}^2 \lesssim 1
\]
Here the value of $s$ is not important. But this is easy to see, as the pure $N$-solitons 
with spectral parameters in $U$ are 
uniformly bounded in all $H^s$ spaces, while the change in $\beta$ corresponds 
to the flow along the first $2N$ commuting flows, so the $\beta$ derivative 
of $Q_{\bfz,\bbeta}$ is also uniformly bounded in all $H^s$ norms. In effect  in the next section we prove that the $N$-solitons are exponentially decaying 
away from at most $N$ bumps, so the same applies to the localized norms in the above formula.

\subsection{ The multisoliton manifold}
\label{multisoliton} 
This subsection is devoted to the proof of Theorem \ref{uniformSol} resp. \ref{t:pure}.
We begin with the case $s=0$, where the notations are simpler.
The argument for other $H^s$ spaces with $s > -\frac12$ is similar,
and is outlined at the end of the section.

We recall that the family $\M_U^N$ of pure $N$-solitons can be described using 
the soliton addition map $\Dress$,
\[
\M_U^N = \{ \Dress(0,\bfs,\bbeta); \ (\bfs,\bbeta) \in \Phase_U^N \}.
\]
as a subset of the set of $N$-soliton states $\V_U^N$
\[
\V_U^N = \{ v = \Dress(u,\bfs,\bbeta); \ u \in \V_U^0, \ (\bfs,\bbeta) \in \Phase_U^N \}.
\]

  On $\V_U^N$ we  define the real valued map
\begin{equation}  
F : \V_U^N \ni v \to   F(v) = E_0(v) - 2\sum_{k}  \im z_k = 
  \Vert v \Vert_{L^2}^2 - \sum_{k} 2 \im z_k. 
\end{equation}
which gives the soliton free $L^2$ energy of $v$. 
In view of the trace formula \eqref{prop:trace}, the map $F$ is obviously
uniformly smooth, non-negative and it vanishes on pure solitons. Thus
also its derivative vanishes at pure $N$ solitons.
We claim that the following three properties hold:

\begin{enumerate}
\item the $N$-soliton manifold can be described as
\begin{equation} \label{solitonmanifold} 
  \M^N_U = \{ v \in \V_U^N: DF(v) = 0 \},
\end{equation}
\item  the Hessian of $F$ evaluated at $v = \Dress(0,\bfz,\bbeta) \in \M^N_U$ 
is nondegenerate, 
\begin{equation} \label{implicitfunction} 
 D^2F(v)[w,w] \ge C^{-1} \Vert w \Vert_{L^2}^2, \qquad w   \in \Ran(D_u \Dress(0,\bfz,\bbeta)), 
\end{equation} 
\item  the $u$ differential of $\Dress$ at $u = 0$ is nondegenerate,
\begin{equation} \label{lowerbound} 
  \Vert w \Vert_{L^2} \le C \Vert D_u D^N(0,\bfz,\bbeta)w\Vert_{L^2}, \qquad w \in L^2.
\end{equation} 
\end{enumerate}

We proceed to prove these three claims. Since $F$ vanishes
quadratically at pure solitons, we must have $DF(v)=0$ whenever $v$ is
a pure soliton. Now suppose that $ DF(v)=0$ with $v =
\Dress(u,\bfz,\bbeta)$. The trace identities imply that
\begin{equation}\label{F-dress}
F( \Dress(u,\bfz,\bbeta)) = \Vert u \Vert_{L^2}^2, 
\end{equation}
hence, differentiating in the $w$ direction,
\[ 
2\real \int u w\, dx  = D_u  (F \circ \Dress(u,\bfz,\bbeta))|_{u=0} (w)  =  DF(\Dress(u,\bfz, \beta)) 
D_u \Dress(u,\bfz,\beta) (w),   
\]
which vanishes for $w=u$ only if $u=0$, or, equivalently, if $D^N(u,\bfz,\bbeta)$ is a pure $N$ soliton. This implies the claim \eqref{solitonmanifold}.

Moreover, we can also calculate the Hessian in \eqref{F-dress} as
\[
\begin{split} 
2 \Vert w \Vert^2\ & = D^2_u (F \circ \Dress(u,\bfz, \bbeta )) [w,w] 
\\ & =   D^2_v F( \Dress(u,\bfz, \bbeta ))]  [D_u \Dress(u,\bfz,\bbeta)w,    
  D_u \Dress(u,\bfz,\bbeta)w   ]
\\ & \qquad + D_v F( \Dress(u,\bfz,\bbeta))     D^2_u \Dress(u,\bfz,\bbeta)[w,w].
\end{split} 
\]   
We evaluate this formula at pure $N$ solitons, using the fact that $DF$ vanishes there: 
\[ 
\begin{split}
2 \Vert w \Vert^2 = & \  D^2_v F \circ \Dress(0,\bfs,\bbeta) [D_u \Dress(0,\bfz,\bbeta)w,    
D_u \Dress(0,\bfz,\bbeta)w]
\\  = & \ 
\Vert   D_u \Dress(0,\bfz,\bbeta)w \Vert^2   -
2 D^2 \im s_1 (D_u \Dress(0,\bfz,\bbeta)w,D_u \Dress(0,\bfz,\bbeta)w).
\end{split}
\]   
The map to the elementary symmetric functions is given by  a nondegenerate contour 
integral of the transmission coefficient, and hence it is uniformly smooth. 
Thus we obtain \eqref{lowerbound}, By Theorem~\ref{t:dress-reg} the two norms in \eqref{lowerbound} must be equivalent, so  \eqref{implicitfunction} also follows.

This gives important information on the uniformly smooth maps
$v \to F(v)$ and $v \to DF(v)$. Let $R(\bfz,\bbeta)$ be the range of $D_u
D^N(0,\bfz,\bbeta)$, which by \eqref{lowerbound} is a closed subspace of $L^2$,
with codimension $4N$.  By \eqref{implicitfunction} $D^2F$ is positive 
definite on this subspace.  Thus $D^2F$ defines a linear map from $L^2$ to
$L^2$ with a $4N$ dimensional null space. The restriction to $R(\bfz ,
\bbeta) $ defines a uniformly invertible operator.
Then by  the implicit function theorem the set $\{v \in L^2;\  DF(v) = 0 \} $ is a
uniformly smooth $4N$-dimensional manifold, which concludes the proof of the theorem.

  \subsection{Regularity of soliton removal}
  Here we give the proof of Theorem~\ref{t:undress-reg}.

  \subsubsection{The spectrum} 
  Let $v \in \V^N_U$ be as above, or equivalently, assume that $\L(v)$
  has exactly $N$ eigenvalues (counting with multiplicity) in $U$.
  These are denoted by $\bfz = \{z_j\}$ and can be described as the
  poles of $T$, not necessarily distinct.
 We call  
\[
 P(z) = \prod_{j=1}^N (z-z_j) = \sum_{n=0}^N (-1)^n s_n z^{N-n} 
 \]
 the characteristic polynomial. 
 As seen in Lemma~\ref{l:z-vs-s}, the relation between the 
symmetric polynomials and the roots is H\"older continuous but not smooth.
The next lemma shows that these polynomials can be smoothly recovered from $T$,
which in turn depends smoothly on $v$ away from the poles.

\begin{lemma}\label{zeroes1} 
  Let $U \subset \C$ be open, $X$ be a complex Banach space, $W \subset X$
  open and
  \[ f: U \times W \to \C\]
  holomorphic. Let $K \subset U$ be compact, $w_0\in  W$ such that
  $f(z,w_0) \ne 0$ for $z \in \partial K$. Then there exists $\varepsilon>0$ so that $f(z,w)$ does not vanish for $z \in \partial K$, $|w-w_0|< \varepsilon$. The number of zeroes in $K$ of $f(.,w)$ is independent of $w$.  Let $(s_n(w))_{n\le N}$ be the elementary symmetric polynomials of the roots. Then
  \[ B_\varepsilon(w_0) \ni w \to s_n(w) \in \C \]
  is holomorphic.
\end{lemma}

\begin{proof} We may assume that $\partial K $ is a union of closed
  nonintersecting positively oriented $C^1$ Jordan curves $\gamma$. Then
    \[ 
    \lambda_k:= \sum_{n=1}^N z_n^k  =    \frac{1}{2\pi i}
  \int_{\gamma} \zeta^k \frac{\partial_zf(\zeta, w)}{ f(\zeta, w)}  d\zeta.  
  \] 
   Then each $s_n$ can be written as a polynomial in $\lambda_k$ and vice versa.
  \end{proof}

\subsubsection{Regularity of soliton removal: Finding the scattering data} 

Next we consider the question of recovering $\bbeta$. For this we use 
the left and right Jost functions, and recall that 
$\bbeta(z_j)$ is the proportionality constant  between them at the poles,
and should be accurate to the order of the pole,
\[
\psi_{l}(x,z) + e^{2i \bbeta(z)}\psi_r(x, z) = O(z-z_j)^{m_j} \qquad \text{ for $z$ near $z_j$ and $x$ in a compact set}.
\]

  To define $\bbeta$ we fix some $x \in \R$ and compare $\psi_l(z)$ and  $\psi_r(z)$ at $x$. 
Each of these two values depends analytically on $z$ and also on $v$ for $v$
near $v_0$.
 
Hence, near each $z_j^0$ we find a ball $B_j$ where either we have
$|\Psi_l^1((z,x)| \ge \frac12 |\Psi_l ((z,x)|$ or $|\Psi_l^2((z,x)|
\ge \frac12 |\Psi_l ((z,x)|$. To fix the notations assume the former.

Here the size $r$ of each $B_j$ depends on the Lipschitz constant for 
$|\psi_l|^{-1} \psi_l$ in $z$ at the point $x$. These balls can overlap, and we identify them if 
the centers are much closer, i.e. 
\[
B_j = B_k \qquad \text{if} \ \ |z_j - z_k| \ll r
\]
and 
\[
2B_j \cap 2B_k = \emptyset  \qquad \text{if} \ \ |z_j - z_k| \gg  r
\]
We choose $r$ so that these are the only alternatives. The same will hold not only for $v_0$,
but also for $v$ in a neighbourhood.  Then we locally define the function $\bbeta_0$ as
\[
e^{2i\bbeta_0(z)} = -\frac{\psi^1_{l}(z)}{\psi_{r}^1(z)}.
\]
(or using the second component, or a linear combination, whichever
works for $v$ near some given state $v_0$.) This is holomorphic near
$z_j$, with a smooth local dependence on $v$. By the Chinese remainder theorem (see
  Lemmas~\ref{l:find-r},~\ref{l:find-r-beta}) there exists a unique
  real polynomial $\bbeta(z)$ of degree at most $2N-1$, so that
\[
\bbeta = \bbeta_0 \qquad (\mod P_\bfz P_{\bar{\bfz}}).
\]
This will define the scattering parameters $\bbeta$ for $v$,
in a manner that depends smoothly on $v \in H^s$. 

We go one step further, and also define a corresponding 
holomorphic family of unbounded wave functions by setting
\[
\psi(z) = T_v(z)(e^{-i \bbeta(z)} \psi_{l}(x,z) +  e^{i \bbeta(z)} \psi_{r}(x,z)).
\]
This will also depend smoothly in $H^{s+1}_{loc}$ on $v \in H^s$. This 
suffices for the local regularity, but we also need to investigate more carefully 
what happens near $\pm \infty$. Consider or instance a neighborhood
of $\infty$. We can localize $v$ there to $\tilde v$, which is now small in $H^s$.
Then we can write the wave function $\psi$ as a wave function for $\tilde v$,
with scattering parameter $\tilde \beta$ which depends smoothly on $v \in H^s$.

\subsubsection{Smooth soliton removal: Finding the
 background} 
  
Once the spectral and scattering parameters are smoothly recovered, we
can recover also $v$ in terms of $u$ following the removal
transformation. To see that we can start by density with the case of
distinct eigenvalues. In this case the iterated soliton removal maps
are smooth with respect to $u$ and the result is independent of the
order.

The case of multiple eigenvalues is obtained as a limit, using Lemma~\ref{l:zhol},
since the unbounded wave functions $\psi$ obtained 
above have a smooth dependence on $u$ and are holomorphic in $z$.
This yields pointwise bounds. The $H^s$ bounds can be dealt with exactly 
as in the case of the soliton addition map. This splits into two parts: (i) locally, which is exactly the same as before, and 
(ii) near infinity, where, as discussed above, this is 
identical to the corresponding argument for the soliton addition map.

\subsubsection{Soliton removal: Uniformity for $\bbeta$ 
in a compact set}

What changes here is that the left and right Jost functions have a single bump
which depends nicely on $v$. The location of the bump depends only on $\bbeta_1$.
Then we choose $x_0$ near this peak, which insures that $\bbeta$ depends 
uniformly smoothly on $v$, and also that $\psi$ has a similar dependence on $v$ 
away from the bump. The rest is similar to the soliton addition map.

\section{The structure of solitons}
\label{s:structure}

Single solitons can be seen as bump functions, with uniform
exponential decay away from the center of the bump. Here we
investigate the similar question for multisolitons. Precisely, we will
show that each $N$-soliton can be viewed as a collection of at most
$N$ unit sized bumps, with exponential decay in between and at
infinity.  Forthermore, each of these bumps has to be exponentially
close  to a lower dimensional soliton.

 Our main result concerning the  structure of $N$ multisolitons is as follows:

\begin{theorem}\label{t:struct}
a) The $N$ multisoliton solutions are functions with exactly $N$ bumps (possibly overlapping),
and exponential decay away from these bumps.

b) If bumps separate into $k$ groups at distance at least $R$, then the multisoliton can be 
approximately viewed as the sum of $k$ multisoliton solutions, with an accuracy 
of $O(e^{-cR})$.
\end{theorem}

\begin{proof}

a) Let $Q= Q_{\bfz,\bbeta}$ be an $N$-soliton and $P_\bfz$ its characteristic polynomial.
Denote by $R = P_{\bfz} P_{\bar \bfz}$, which has real coefficients.
Consider the action of the $n$-th flow on $Q$, or more precisely on $\bbeta$.
This gives
\[
\dot \bbeta = i 2^{n-1} z^{n} \qquad (\mod R).
\]
It follows that the first $2N$ flows are linearly dependent when acting on $Q$.

Precisely, to any real polynomial 
\[
R(z) = \sum_{j=0}^{2N} r_j (2z)^j
\]
we can associate the Hamiltonian
\[
H_R = \sum_{j=0}^{2N} r_j H_j.
\]
Then for $R = P_{\bfz} P_{\bar \bfz}$ defined above, we know that $Q$ is a steady state for the $H_R$
flow. This is equivalent to 
\[
DH_R(Q) = 0
\]
which shows that $Q$ solves a semilinear ODE of order $2N$. 
The linear part of this ODE is given by the operator $R(D_x) = R(\frac{1}{i} \partial_x)$. 
So we can rewrite this ODE in the form
\begin{equation}\label{ode-2N}
R(D_x) Q = N(Q^{(\leq 2N-1)})
\end{equation}
Equivalently, we can rewrite this as a first order system for 
the variables 
\[
y  = \{ D^j Q; j = 0,2N-1 \},
\]
namely 
\begin{equation}\label{ode-1}
\partial_x y = i A y +N(y),
\end{equation}
where the matrix $A$ has $R$ as a characteristic polynomial and $N$ 
is polynomial and contains quadratic and higher order terms.

The state $y = 0$ is a fixed point for the system \eqref{ode-1}, and the 
eigenvalues for the linearization around $y = 0$ are $\pm iz_k$, neither 
of which is on the imaginary axis. Hence $0$ is a hyperbolic fixed point
for this dynamical system. Hence, by the Hartman-Grobman theorem,
the dynamics around $y=0$ are well described by the corresponding linearized
 flow, up to a local H\"older continous homeomophism with H\"older continuous inverse.  

Now we are able to complete the qualitative description of the solitons.
We consider the localized mass of $Q$ in unit intervals $I_j = [j,j+1]$,
\[
M_j = \int_{I_j} |Q|^2\, dx,
\]
In intervals where $M_j$ is small, all of the Cauchy data of $Q$ must
be small so the Hartman-Grobman theorem applies. But the total mass is
finite, so there can be only finitely many intervals where $M_j$ is
large.  Outside this finite number of intervals, the soliton $Q$ must
follow the linearized dynamics and decay exponentially.  This
completes the proof of part (a) of the theorem, modulo the counting of
the bumps; we still need to show that, if $\epsilon$ is small enough,
then there are at most $N$ regions where $|Q| > \epsilon$. This will
follow as a corollary of the proof in (b).

b)  Denote by $R \gg 1$ the smallest gap between two bumps.
Then in between each two bumps, we will find a smallest value for $y$,
\[
|y(x_k) | \lesssim e^{-cR}, \qquad k = 1,K.
\]
Away from $x_k$, $y$ will grow exponentially. We use the $x_k$ as sharp cut points
for $Q$, splitting it on the intervals $I_k = (x_k, x_{k+1})$ where 
$x_0 = -\infty$ and $x_{K+1} = +\infty$,
\[
Q = Q_1 + \cdots + Q_{K+1},  \qquad Q_k = 1_{I_k} Q.
\]
On one hand, we have the obvious energy relation
\[
\| Q\|_{L^2}^2 = \sum \|Q_k\|_{L^2}^2.
\]

On the other hand, we investigate the relation between the transmission coefficients 
of $Q$ and those of $Q_j$. We work with $z$ away from the spectral parameters $\bfz$ 
of $Q$, say on a contour around $\bfz$. For such $z$, the renormalized Jost function
$\tpsi_l$ associated to $Q$ satisfies 
\[
|\tpsi_l| \gtrsim 1, \qquad \lim_{x \to \infty} \tpsi_l = T_Q^{-1}(z).
\]
Furthermore, around the points $x_k$ the coupling between the two
components of the $\tpsi_l$ equation \eqref{eq:tpsil}
is exponentially
small, therefore we also obtain the exponential smallness
\[
|\tpsi_l^2(x_k)| \lesssim e^{-cR}, \qquad |\tpsi_r^1(x_k)| \gtrsim 1.
\]

Next we consider the corresponding Jost function $\psi_{j,l}$ for $Q_j$. There the effective evolution is in $I_j$, with initial data
\[
\tpsi_{j,l}(x_j) = e_1,
\]
and terminal data
\[
\tpsi_{j,l}(x_{j+1}) = T_{Q_j}^{-1}(z) e_1 + c e_2.
\]
Now on the interval $I_j$ we compare $\tpsi_l$ and $\tpsi_{j,l}$,
which solve the same equation and have nearly collinear data.
It immediately follows that we must have the relation
\[
T_{Q_j}^{-1}(z) = \frac{\tpsi_l^2(x_{j+1})}{\tpsi_l^2(x_{j})} + O(e^{-cR}).
\]
Multiplying these relations, it follows that for $z$ on our curve $\gamma$ we have
\[
T_Q^{-1}(z) = \prod_{j = 0}^K  T_{Q_{j}}^{-1}(z)+ O(e^{-cR}).
\]
This implies that the product on the right must have the same number of zeroes
as the left hand side, call them $\tilde \bfz$, and further that the zeros of the left hand side $\bfz$ and 
$\tilde \bfz$ must be close,
\[
d(\bfz,\tilde \bfz) \lesssim e^{-cR},
\]
for some new uniform constant $c$. 

Applying the soliton removal map to $Q_j$ within the contour $\gamma$ we get 
\[
\bf B^{N_j}_- Q_j = (u_j,   \tilde \bfz_j, \bbeta_j),
\]
where $\tilde \bfz_j$ are the poles of $T_{Q_{j}}$, which represent a subset of $\tilde \bfz$.
By the trace formula for $Q_j$ we get
\[
\|Q_j\|_{L^2}^2 = \|u_j\|_{L^2}^2 + \im \tilde \bfz_j,
\]
while by the trace formula for $Q$,
\[
\|Q_j\|_{L^2}^2 = \im \bfz_j.
\]
Summing up in the first relation and comparing with the second,
we obtain 
\[
\sum_{j=0}^K \|u_j\|_{L^2}^2 \lesssim e^{-2cR}.
\]
A corollary of this is that each $Q_j$ must have at least an eigenvalue within $\gamma$,
or else it would have to have a very small $L^2$ norm. This implies that there can be at 
most $N$ such $Q_j$, which completes the proof in part (a).

Finally, we define the multi-solitons 
\begin{equation}\label{def-tQj}
 \tilde Q_j = \bf B^{N_j}_+(0,   \tilde \bfz_j, \bbeta_j)   
\end{equation}
By the uniform regularity of the soliton addition map, we have
\[
\| Q_j - \tilde Q_j\|_{L^2} \lesssim e^{-cR},
\]
so that 
\[
Q = \sum \tilde Q_j +  O_{L^2} (e^{-cR}),
\]
as desired. We note that the $L^2$ bound in the error can easily be
upgraded to any higher Sobolev norm by interpolation.
This concludes the proof of the theorem.
\end{proof}

An interesting question which emerges from the proof of the above theorem 
is whether one can lift the above correspondence to the level of the soliton 
manifolds. Above we have defined a map 
\begin{equation}\label{def-Gamma}
\M^N \ni Q \to  \Gamma(Q) :=  \{ \tilde Q_j \} \in \prod \M^{N_j}
\end{equation}
with the property that 
\[
\| Q - \sum \tilde Q_j \|_{H^s} \lesssim e^{-cR}.
\]

One could also argue in reverse fashion, namely start with the solitons 
$\tilde Q_j$ and sum them,
\[
v = \sum \tilde Q_j,
\]
Then the same argument as in the proof of the theorem shows that 
$v$ is a near soliton, in the sense that its residual energy is small.
Precisely, if 
\[
\bf B_-^N v = (u, \bfz,\bbeta).
\]
then we have 
\[
\| u\|_{H^s} \lesssim e^{-cR}, \qquad s > -\frac12.
\]
Hence by the mapping properties of the soliton addition,
it follows that the map
\begin{equation}\label{def-tGamma}
\times \M^{N_j} \ni \{\tilde Q_j\} \to \tilde \Gamma (\{\tilde Q_j\}) = Q := \bf B_+^N(0,\bfz,\bbeta) \in \M^N,
\end{equation}
is a near addition in the uniform norm,
\[
\| Q - \sum \tilde Q_j\|_{H^s} \lesssim e^{-cR}.
\]

If follows that the two manifolds $\M^N$ and $\sum \M^{N_j}$ are locally $O(e^{-cR})$ close.
Here the product can be interpreted as a smooth manifold via the addition map, since 
the manifolds $\M^{N_j}$ are locally uniformly transversal; this is because the elements in their tangent space are exponentially localized near the $O(R)$ separated points $x_j$.

But the two manifolds are also locally uniformly smooth; it follows that they must also be close
in any smooth topology:

\begin{theorem}
Let $Q$ be an $N$ soliton with $R$ separated bumps, and let $\tilde Q_j$ be as \eqref{def-tQj}. Then locally, near $Q$, respectively 
$\sum \tilde Q_j$, the manifolds $\M^N$ and $\sum \M^{N_j}$  are $O(e^{-cR})$ close as smooth manifolds.
\end{theorem}

We note that this does not inply that either of the maps $\Gamma$, respectively $\tilde \Gamma$,
are uniformly smooth near identity maps between the two manifolds. This would require a uniform regularity statement for the soliton removal map, which we wo not have. Nevertheless, we conjecture that such a result should be true.

\section{The stability result}\label{s:stable}

Here we prove the stability result using the regularity of the soliton addition map.
We first restate the result in a more accurate form:

\begin{theorem}\label{t:stability}  
Let $s > -\dfrac12$, and $U$ a compact subset of the upper half-plane. There exist $\varepsilon_0>0$ and $C>0$ so that the following is true. 
Let $v$ be a pure $N$-soliton solution for either NLS or mKdV with initial 
data $v_0 \in \M_U^N$. If 
\begin{equation}
\| v_0 -w_0\|_{H^s} = \varepsilon \leq \varepsilon_0,
\end{equation}
then there exists another pure $N$-soliton solution $\tilde v$ so that 
\begin{equation}
\sup_{t \in \R} \| w(t) - \tilde v(t)\|_{H^s} \leq C \epsilon.
\end{equation}
b) Furthermore, this result is uniform with respect to all $N$-soliton solutions with spectral 
parameters in a compact subset of the open upper half-plane.    
\end{theorem}

\begin{proof}

We denote by $\bfz_0,\bbeta_0$ the spectral, respectively scattering parameters for $v_0$.
The transmission coefficient for $v_0$ is then given by 
\[
T_{v_0}(z) =  \prod_{k=1}^N \frac { z- \bar z_{k0}}{z - z_{k0}},
\]
and has poles at $\bfz_0$.

Away from the poles, the transmission coefficient depends smoothly on the input function.
Hence, if $\varepsilon_0$ is small enough (depending only on $U$), it follows that the transmission 
coefficent of $w_0$ has exactly $N$ poles $\bfz$ in a small neighbourhood of $U$,
and that $\bfz$ is close to $\bfz_0$,
\[
d(\bfz,\bfz_0) \lesssim \varepsilon   
\]
where the distance is measured using the symmetric polynomials.

We now apply the soliton removal map to $v_0$, denoting
\[
\Undress (w_0) = (u_0,\bfz,\bbeta),
\]
and define the initial data 
\[
\tilde v_0  = \Dress(0,\bfz,\bbeta). 
\] 

By the trace theorem, the $H^s$ energy of $w_0$ splits into
\[
E_s(w_0) = E_s(u_0) + \sum_{k=1}^N \Xi_s(z_k) =:F_s(w_0) + \sum_{k=1}^N \Xi_s(z_k)
\]
where $F_s(w_0)$ denotes the "no soliton energy" of $w_0$. This 
is uniformly smooth in the $H^s$ topology, see \cite{MR3874652} and also Theorem~\ref{energies}, and vanishes of second order on the $N$-soliton manifold $\M_U^N$. It follows 
that 
\[
F_s(w_0) \lesssim \epsilon^2 
\]
which reinterpreted in terms of $u_0$ shows that
\[
E_s(u_0) \lesssim \epsilon^2.
\]
Since $E_s$ is positive definite for small data, it follows that 
\[
\| u_0\|_{H^s} \lesssim \epsilon,
\] 
and by the uniform regularity of the soliton addition map,
\[
\| w_0 - \tilde v_0\|_{H^s} \lesssim \epsilon.
\]

The  $N$-soliton solution $\tilde v$ with initial data $\tilde v_0$ to either NLS or mKdV is given by
\[
\tilde v(t) =  \Dress(0,\bfz,\bbeta(t)),
\]
where the parameter $\beta(t)$ depends on whether we consider the NLS or mKdV flow.

Denote by $u$ the solution to NLS or mKdV with initial data $u_0$. Since $E_s$ is conserved, this remains small,
\[
\| u(t)\|_{H^s} \lesssim \varepsilon.
\]

On the other hand the soliton addition map commutes with the flows, so we must have
\[
w(t) = \Dress(u(t),\bfz,\beta(t)).
\]
Using our result on the uniform regularity of the soliton addition map
in Theorem~\ref{t:dress-reg}, it follows that
\[
\| w(t) - \tilde v(t)\|_{H^s} \lesssim \| u(t)\|_{H^s} \lesssim \varepsilon,
\]
which concludes the proof of our theorem.
\end{proof}

\section{Double eigenvalues}

\label{doubleeigenvalue} 
In this section w undertake a case study of double eigenvalues to gain some additional intuition and to provide some examples of multisoliton  dynamics. 

\subsection{The asymptotic shift due to interaction} 
We begin with the case of two different eigenvalues $z_1 \ne z_2$ and the $z_j$ waves for the Lax operator with trivial potential, 
\[ \psi_1 = \left( \begin{matrix} e^{\gamma_1 - i z_1x} \\ e^{-\gamma_1 + i z_1 x} \end{matrix} \right)  \] 
with $|\real \gamma_1| \lesssim 1$. We assume that $ \real \gamma_2$ is large and choose 
\[ \psi_2 = \left( \begin{matrix} 1 \\  e^{-2\gamma_2+ iz_2 x} \end{matrix} \right).\] 
In this regime it is convenient to apply the iterated \Backlund transform. The second intertwining operator (the one with respect to the index $2$) is 
\[  D_2= \left( \begin{matrix}  i\partial  - \bar z_2  &  0\\ 0 & - i\partial - \bar z_2  \end{matrix} \right)-2i \frac{\im z_2}{1+ e^{-2\real (\gamma_2- iz_2 x)}  } \left( \begin{matrix} 1 & e^{-\overline{ (\gamma_2-iz_2 x)}   } \\ e^{-(\gamma_2-iz_2x)} &  e^{-2 (\real \gamma_2 + \im z_2 x)} \end{matrix} \right) \] 

We apply the second  intertwining operator to $\psi_1$, 
\begin{equation} \label{expansion}   D_2 \psi_1 = \left( \begin{matrix} 
\displaystyle \Big[(z_1 -  z_2)+2i \im z_2 \frac{e^{-2 \real \gamma_2 -2 \im z_2 x}}{1+e^{-2\real \gamma_2 -2 \im z_2 x} }\Big]    e^{\gamma_1 -i z_1 x} -2i \im z_2   \frac{e^{- \overline{\gamma_2 -iz_2 x} - \gamma_1+ i z_1x} }{1+ e^{-2 \real \gamma_2-2\im z_2 x}} 
\\[2mm]  \displaystyle \Big[ (z_1-\bar z_2) -2i \im z_2 \frac{e^{-2(\real \gamma_2 + \im z_2 x)}}{1+ e^{-2(\real \gamma_2+ \im z_2 x)}}    \Big]
e^{-\gamma_1 + i z_1x} -2i \im z_2 \frac{e^{-(\gamma_2-iz_2 x)} e^{\gamma_1 -i z_1 x}}{1+ e^{-2(\real \gamma_2 + \im z_2 x)} }\end{matrix} \right)
\end{equation} 

 Without interaction the positions of the solitons would be
 the point $x_j$ where both components of $\psi_j$ have the same size, 
 \[  x_j = - \frac{\real \gamma_j}{\im z_j}. \] 
We assume without loss of generality $x_2 \le x_1$. We are interested in the case that the two  soliton function  will consist of two separated bumps. We define their position as the point where the amplitude has a local maximum, or, equivalently, where both components of $D_2 \psi_1$ have the same size. The  second soliton is far to the left of the first soliton if
    \[   \im z_2 e^{ -  \real \gamma_2 - \im z_2 x_1}  << |z_1-z_2|. \] 
Then 
\[ D_2 \Psi(x_1) =  \Big(1 + O\Big(\frac{\im z_2}{|z_1-z_2|} e^{-\real \gamma_2 - \im z_2 x_1}\Big)\Big)  \left( \begin{matrix} z_1-z_2 \\ z_1-\bar z_2 \end{matrix} \right) \] 
and due to the exponential factor $ e^{ \pm ( \gamma_1-iz_1 x)}$ 
\begin{equation} \label{asym}  y_1 = \frac{1}{2\im z_1}  \ln \frac{ |z_1-\bar z_2|}{|z_1- z_2|}+O\Big( \frac{\im z_2}{|z_1-z_2|} e^{-\real \gamma_2 - \im z_2 x_1}  \Big).  \end{equation}  
This gives the asymptotic shift due to the interaction when the solitons are well separated. it is not hard to work out the shape of the solitons in this case

\subsection{ An algebraic computation}\label{ss:w}
 In the sequel we seek for a more detailed understanding when the solitons are well separated  with a separation independent of the distance between the eigenvalues, a much more involved task.     
We consider again two  states $\psi_1$ and $\psi_2$  associated to eigenvalues
$z_1, z_2$ and consider the corresponding matrix $M$ 
\[
M = i \left( \begin{matrix} \dfrac{\psi_1^* \psi_1}{\bar z_1 -  z_1} 
& \dfrac{\psi_2^* \psi_1}{\bar z_2 -  z_1} \\ 
 \dfrac{\psi_1^* \psi_2}{\bar z_1 -  z_2} &  \dfrac{\psi_2^* \psi_2}{\bar z_2 -  z_2}  \end{matrix} \right). 
\]
Let $m$ be the inverse of $M$. We evaluate the expression
\begin{equation}\label{w-def}
   w = 2\bar \psi_j^2 m_{jk} \psi_k^1,
\end{equation}
which arises in the definition of the \Backlund transform
in \eqref{eigen-m}.   Our first
task is to compute the determinant of $M$,
\[
\begin{split}
\det M = & \  \prod_{i,j=1}^2  \frac{-1}{\bar z_i - z_j} 
(| z_1 - \bar z_2|^2 |\psi_1|^2 |\psi_2|^2 - 4 \im z_1 \im z_2 |\psi_1^* \psi_2|^2)
\\
= & \ \prod_{i,j=1}^2  \frac{-1}{\bar z_i - z_j} \Big(| z_1 - z_2|^2 |\psi_1|^2 |\psi_2|^2 + 
4 \im z_1 \im z_2 ( |\psi_1|^2 |\psi_2|^2 - |\psi_1^* \psi_2|^2\Big),
\end{split}
\]
where we have used
\begin{equation} \label{usefulid} 
|z_1-\bar z_2|^2 = |z_1-z_2|^2 +  4 \im z_1 \im z_2.      
\end{equation} 
We recall that we can choose
\[ 
\psi_j  
= \left( \begin{matrix} e^{\gamma_j} \\ e^{-\gamma_j}\end{matrix} \right), \qquad \gamma_j = -i \sum_{k=0}^3 \beta_k z_j^k    -i z_j x.
\] 
Then we can rewrite the expression  
\[
D=-\prod\limits_{i,j=1}^2  (\bar z_i -z_j) \det M
\]
 as
\begin{equation}
\label{determinant} 
\begin{split} 
D & \, =  4 |z_1-z_2|^2 \cosh( 2 \real \gamma_1) \cosh(2\real  \gamma_2) 
+ 16\im z_1 \im z_2 \left| \sinh(\gamma_1-\gamma_2)\right|^2 
\\ & =    2|z_1-z_2|^2 [|\cosh( \gamma_1+ \gamma_2)|^2 + |\sinh( \gamma_1+ \gamma_2)|^2
+ | \cosh (\gamma_1-\gamma_2)|^2]
\\ & \ \ \ 
+ 2 [|z_1+z_2|^2 - 2(z_1 z_2 + \bar z_1 \bar z_2)] \left| \sinh(\gamma_1-\gamma_2)\right|^2.
\end{split} 
\end{equation}  
We can read off important parts of the structure. At the right hand side of the first equality we see a sum of two terms, the first containing a factor $|z_1-z_2|^2$, and the second a factor $|\gamma_1-\gamma_2|^2$. This vanishes quadratically exactly when $z_1 = z_2$ and $ \gamma_1=\gamma_2$ modulo $i \pi$, or, equivalently, if  $\psi_1$ and $\psi_2$ are collinear.

On the right hand side of the second equality we consider 
$z_j$, $\gamma_j$ and $\bar z_j$ resp. $\bar \gamma_j$
as separated variables and see that exchanging 
$(z_1, \gamma_1)$ and $(z_2,\gamma_2)$ while keeping the complex conjugates changes the sign.

In the complement of the set where $\det M$ vanishes  we can write
\[
w = -\frac{2 A}{ D }
\]
where $A$ is given by  $4$ times
\[
 ( e^{-\bar \gamma_1}, e^{-\bar \gamma_2}) \! \left(\begin{matrix} \frac12 |z_2\!-\!\bar z_1|^2 \im z_1 (e^{\gamma_2+\bar \gamma_2}\!+\! e^{-\gamma_2-\bar \gamma_2}) &\hspace{-2mm}  - i  \im z_1 \im z_2  (\bar z_1\!-\!z_2)(e^{\gamma_1+\bar \gamma_2}\! +\! e^{-\gamma_1-\bar \gamma_2})  \\  - i \im z_1 \im z_2 (\bar z_2\!-\!z_1)(e^{ \gamma_2+\bar \gamma_1}\!+\! e^{-\gamma_2-\bar \gamma_1})  &  \frac12 |z_2\!-\!\bar z_1|^2 \im z_2  (e^{\gamma_1+\bar \gamma_1}\!+ \!e^{-\gamma_1-\bar \gamma_1})   \end{matrix} \right)\! \!\left( \begin{matrix} e^{\gamma_1} \\ e^{\gamma_2} \end{matrix} \right)\!\!. 
 \] 

We rewrite $A$ as follows: 
\begin{equation} \label{numerator}  \begin{split} 
 A= &  \  4|z_2 -  \bar z_1|^2 \left\{ \im z_1  ( \exp(2i \im \gamma_1) \cosh(\gamma_2+\bar \gamma_2)    
+  \im z_2  (\exp(2i \im z_2) \cosh( \gamma_1+\bar\gamma_1)  \right\} 
\\ &  
- 4i \im z_1 \im  z_2   \Big\{  (\bar z_1-z_2) ( \exp(\gamma_1 -\bar \gamma_1 +\gamma_2 +\bar \gamma_2)+ \exp(-\gamma_1-\bar \gamma_1 +\gamma_2-\bar \gamma_2)) \\ &\hspace{3cm} + (\bar z_2 - z_1)( \exp(\gamma_1+\bar  \gamma_1 + \gamma_2-\bar \gamma_2) + \exp(\gamma_1-\bar \gamma_1-\gamma_2-\bar \gamma_2) 
 \Big\}
 \\  = & \  4|z_1-z_2|^2 ( \im z_1  e^{ 2i\im \gamma_1} \cosh (2 \real \gamma_2) +  \im z_2 e^{2i \im \gamma_2} \cosh (2 \real \gamma_1) )
 \\ &  + 8i \im z_1 \im z_2 \Big\{ (z_1-z_2) e^{\gamma_1+\gamma_2} 
 \sinh(\bar \gamma_1 -\bar \gamma_2)
 + (\bar z_1-\bar z_2) e^{-\bar \gamma_1-\bar \gamma_2} 
 \sinh(\gamma_1 -\gamma_2)\Big\}
 \\ = & \  |z_1-z_2|^2 \Big( 
2\im(z_1+z_2)(e^{ 2i\im \gamma_1} \cosh (2 \real \gamma_2)+
e^{2i \im \gamma_2} \cosh (2 \real \gamma_1))
\\  & \qquad 
+ ((z_1-z_2)+(\bar z_1- \bar z_2) )(e^{\gamma_1-\bar \gamma_1} \cosh (\gamma_2+\bar \gamma_2)-
e^{\gamma_2-\bar \gamma_2} \cosh (\gamma_1+\bar \gamma_1))\Big)
 \\ &  + 8i \im z_1 \im z_2 \Big\{ (z_1-z_2) e^{\gamma_1+\gamma_2} 
 \sinh(\bar \gamma_1 -\bar \gamma_2)
 + (\bar z_1-\bar z_2) e^{-\bar \gamma_1-\bar \gamma_2} 
 \sinh(\gamma_1 -\gamma_2)\Big\}
 \\= & \ |z_1-z_2|^2 \Big\{ 2\im (z_1+z_2)
 \left( e^{\gamma_1+\gamma_2} \cosh(\bar \gamma_1 - \bar \gamma_2)
 +e^{-\bar \gamma_1 -\bar \gamma_2} \cosh(\gamma_1 - \gamma_2)\right)
 \\
 & \ \ \ \ \ \ 
  - i (z_1-z_2) e^{-\bar \gamma_1 -\bar \gamma_2}\sinh(\gamma_1-\gamma_2) - i (\bar z_1 - \bar z_2)
 e^{\gamma_1+\gamma_2} \sinh(\bar \gamma_1 - \bar \gamma_2)\Big\} 
 \\ &  + i  [|z_1+z_2|^2 - 2(z_1 z_2 + \bar z_1 \bar z_2)]
 \\ & \qquad \times \Big\{ (z_1-z_2) e^{\gamma_1+\gamma_2} 
 \sinh(\bar \gamma_1 -\bar \gamma_2)
 + (\bar z_1-\bar z_2) e^{-\bar \gamma_1-\bar \gamma_2} 
 \sinh(\gamma_1 -\gamma_2)\Big\}
   \end{split}
\end{equation}

Both $A$ and $D$ are smooth. It is an easy consequence that $w = -2AD^{-1}$ is smooth in the set $ \{ z_1\ne z_2\} \cap \{\gamma_1-\gamma_2 \notin i\pi \Z\}$ and that it vanishes if $\gamma_1-\gamma_2 \notin i\pi \Z$ but $z_1=z_2$. 

In order to resolve the apparent singularity at the zeroes 
of the denominator, we view $z_j$ and $\bar z_j$ as separate variables,
and similarly for $\gamma_j$ and $\bar \gamma_j$.
We first observe that both $A$ and $D$ are 
odd with respect to the separate symmetries 
\[
(z_1,\gamma_1) \leftrightarrow (z_2,\gamma_2),
\]
respectively
\[
(\bar z_1,\bar \gamma_1) \leftrightarrow (\bar z_2,\bar \gamma_2).
\]
Then their ratio is invariant under both separate exchanges.
To capture the cancellation allowed by this symmetry we introduce 
the auxiliary variables $\gamma,\alpha$
\begin{equation}\label{w-notation}
2\gamma = \gamma_1+ \gamma_2 \qquad \alpha = \frac{\sinh(\gamma_1-\gamma_2)}{z_1-z_2}, 
\end{equation}
and cancel a $|z_1-z_2|^2$ factor. We obtain 

\begin{lemma} 
With the notations in \eqref{w-notation}, the expression $w$ in \eqref{w-def}
can be represented in the nondegenerate form
\[
w  
= -\frac{2A_0}{D_0} 
\] 
   where 
\begin{equation}\label{A0} 
\begin{split}
 A_0 = & \  2\im (z_1+z_2)
 \left( e^{2\gamma} \cosh(\bar \gamma_1 - \bar \gamma_2)
 +e^{-2\bar \gamma} \cosh(\gamma_1 - \gamma_2)\right)
 \\
& 
  - i \alpha (z_1-z_2)^2 e^{-2\bar \gamma} - i \bar \alpha (\bar z_1 - \bar z_2)^2
 e^{2\gamma}  
 + i [|z_1+z_2|^2 - 2(z_1 z_2 + \bar z_1 \bar z_2)] \Big\{ \bar\alpha e^{2\gamma} 
 + \alpha e^{-2\bar \gamma} \Big\}
\end{split}
\end{equation} 
and 
\begin{equation} \label{D0} 
D_0 =  2\Big[|\cosh(2\gamma)|^2 + |\sinh( 2\gamma)|^2 + | \cosh (\gamma_1-\gamma_2)|^2\Big] 
+ 2\Big[|z_1+z_2|^2 - 2(z_1 z_2 + \bar z_1 \bar z_2)\Big] |\alpha|^2  
\end{equation} 
\end{lemma} 
Assuming that $z_1,z_2$ are confined to a (small) compact subset of the upper half-plane, $s_1=z_1+z_2$, $s_2=z_1^2+z_2^2$
we interpret this expression as a zero homogeneous form 
\[
w = w(\mu,s_1,s_2)
\]
in the complex variables
\[
\mu = (\mu_1,\mu_2,\mu_3,\mu_4) = (\cosh(2\gamma), \sinh( 2\gamma), \cosh (\gamma_1-\gamma_2), \alpha),
\]
with smooth coefficients which are symmetric functions separately in $(z_1,z_2)$ and 
$(\bar z_1,\bar z_2)$, resp. smooth coefficients in $s_1$ and $s_2$.  

For the function $w$ we note the pointwise bound:
\begin{equation}\label{w-size}
|w| \lesssim \frac{(|\mu_1| + |\mu_2|)(|\mu_3| +|\mu_4|)}{|\mu|^2}  := w_0, 
\end{equation}
which in particular shows that for unbalanced $\mu$'s $w$ must be small:
\[
|w| \approx 1 \quad \implies  |\mu_1| + |\mu_2| \approx |\mu_3| +|\mu_4|.
\]
We also have similar bounds for the derivatives of $w$ with respect to $\mu$,
\begin{equation}\label{dw-size}
| \mu|^{|\alpha|} \Big| \partial_\mu^{\alpha}   w\Big| \lesssim  w_0 . 
\end{equation}

One might be tempted to parametrize $w$ as a function of 
$z_1,z_2,\gamma$ and $\alpha$, but $\cosh (\gamma_1-\gamma_2)$ can only be viewed
locally as a smooth function of $\alpha$ for $\gamma_1-\gamma_2$ away from $(\frac12 + \Z)\pi i$.
Thus it is better to think of these variables, together with $z_1+z_2$ and $z_1 z_2$ 
as functions on a smooth complex manifold $M$ of complex dimension $4$, which is 
the is cartesian product of the smooth Riemann surface 
\[ \{ (\mu_1, \mu_2): \mu_1^2 -\mu_2^2 =1\}  \] 
and the three dimensional complex manifold (recall that 
$(z_1-z_2)^2=  2s_2 -s_1^2 $)
\[ 
\{(\mu_3,\mu_4, s_1,s_2)\in \C^4 :   \ \mu_3^2 - (2s_2-s_1^2) \mu_4^2 = 1 \},
\]
which is smooth since $ \mu_3^2 - (2s_2-s_1^2) \mu_4^2 -1 $ 
is nondegenerate in a neighborhood of the manifold.

We remark that on $M$ we have the relations
\[
|\mu| \geq 1, \quad ||\mu_1| - |\mu_2|| \le  1 , \qquad \Big||\mu_3|-|z_1-z_2| |\mu_4| \Big| \le 1 
\]
Then we can bound
\[
| w_0 | \le \frac{(1+2|\mu_1|)(1+2|\mu_4|)}{|\mu|^2}
\]
if $|z_1-z_2| \le 1$,
which can only be large if $|\mu_1|+1 \approx |\mu_4|+1$.

The function $w$ above will describe the pointwise size of a soliton. 
Because of that, the next question we want to address is where is $w$ large. 
Heuristically we expect to have two regions of interest 

\begin{enumerate}[label=(\roman*)]
\item {\bf The one bump case} $|\mu_1|, |\alpha| \lesssim 1$
where the amplitude of $w$ could get as high as $2 \im(z_1+z_2)= 4\im z $,with $2 z = z_1 + z_2$.
This  value is attained when $\gamma_1 = \gamma_2 = \alpha = 0$.

\item  {\bf Separated bumps} $|\mu_1| \approx |\alpha| \gg 1$, where we have amplitudes closer to $2 \im z$ if $z_1-z_2 $ is small. 
\end{enumerate}

We are particularly interested in understanding this in the (near) degenerate case, when $z_1$ and $z_2$ are close but $\alpha$ is large 
and approximatively balances the $\cosh(4 \real \gamma)$ in the denominator, so that $w$ has size $O(1)$. Toward that goal, we denote
\begin{equation} \label{def-sigma} 
\sigma = (z_1 - z_2) \coth (\gamma_1-\gamma_2)
\end{equation} 
which is bounded when $\alpha$ is large, and has limit $\pm(z_1-z_2)$
as $\real(\gamma_1 - \gamma_2)$ goes to $\pm \infty$.
Then we can rewrite $D_0$ as 
\[
D_0 = 2 \cosh(4 \real \gamma) + 2[|z_1+z_2|^2 - 2 (z_1z_2+\bar z_1 \bar z_2) + \sigma^2] |\alpha|^2.
\]
Since 
\[
|\sigma|^2 = |z_1 - z_2|^2 + O(|\alpha|^{-2}),
\]
it follows that 
\[
D_0 = 2 \cosh(4 \real \gamma) + 2|z_1-\bar z_2|^2 |\alpha|^2 + O(1).
\]
On the other hand we can rewrite the expression $A_0$ as
\begin{equation}\label{A0+} 
A_0 = c_+ \bar\alpha e^{2\gamma} + c_- \alpha e^{-2\gamma},
\end{equation} 
where $c_+$ and $c_-$ are bounded, 
\[
\begin{split}
 c_+ = & \  i [|z_1+z_2|^2 - 2(z_1 z_2 + \bar z_1 \bar z_2)-(\bar z_1-\bar z_2)^2] + 2 \im(z_1+z_2) \bar \sigma
  \\
 c_- = & \  i [|z_1+z_2|^2 - 2(z_1 z_2 + \bar z_1 \bar z_2)-( z_1- z_2)^2] + 2 \im(z_1+z_2)  \sigma
 \end{split}
\]
Thus we get
\begin{equation}\label{w-app1}
w = -2 e^{2 i \im \gamma} 
\frac{ c_+ \bar\alpha e^{2\real \gamma} + c_- \alpha e^{-2 \real \gamma}}
{2 \cosh(4 \real \gamma) + 2|z_1-\bar z_2|^2 |\alpha|^2} 
+O(\frac{1}
{\cosh(4 \real \gamma) +  |\alpha|^2} ).
\end{equation}
With a slightly larger error we can further simplify this as
\begin{equation}\label{w-app2}
w = - \frac{8   i (\im z)^2 (\bar\alpha e^{2\gamma} + \alpha e^{-2\bar \gamma})} 
 {\cosh{4 \real \gamma}
+ 8 (\im z)^2 |\alpha|^2 } +{O}(|z_1-z_2|) + O( \alpha^{-1}).
\end{equation}
This has near maximum amplitude $2 \im z$ when
\begin{equation}  \label{nearmax} 
\cosh(2 \real \gamma) = 2 \im z |\alpha|,
\end{equation}
and phase
\[
\pi/2 \mp (\arg \alpha - 2 \im \gamma), 
\]
where the sign depends on the sign of $\real \gamma$.

 It is also interesting to check the asymptotic behavior of \eqref{w-app1} as $\real(\gamma_1-\gamma_2) \to \infty$.
 There we can approximate
\[
\cosh(\gamma_1 - \gamma_2) \approx \sinh(\gamma_1 - \gamma_2) \sgn \real(\gamma_1-\gamma_2).
\]
This yields
\[
w = \frac{-\bar c_-      \bar\alpha e^{2\gamma} +
c_+  \alpha e^{-2\bar \gamma}} 
 {\cosh{4 \real \gamma}
+ 8 (\im z)^2 |\alpha|^2 } +{O}(|z_1-z_2|^2) + O( \alpha^{-2}),
\]
where 
\[
c_{\pm} = (8 i (\im z)^2 \pm 4 \im z(z_1-z_2)   \sgn \real(\gamma_1-\gamma_2)). 
\]
This  gives factors of $\im z \im z_1$ respectively 
$\im z \im z_2$ at the numerator, which will select the different bump 
amplitudes $2\im z_1$, respectively $2 \im z_2$.
Compared with the prior computation we see the transition from the amplitude $2\im z$ for one bump solitons to the amplitude $2\im z_j$ as the distance tends to infinity.

\subsection{Two soliton states} 

Separated two solitons   are close to the sum  of two  $1$-solitons. We study the general pure two soliton solution and estimate the difference to the algebraic sum of two solitons, whenever the two centers are far apart. This analysis is new, nontrivial and interesting in the case of two close eigenvalues. Asymptotically the eigenvalue parameters of the two solitons are the poles of the transmission coefficients. But as soon as their distance is closer than $\ln (2+ \frac{\im(z_1+z_2)}{|z_1-z_2| })$, the interaction is visible
and we can see a transition regime via effective soliton parameters, which we describe. As a consequence we obtain a uniform parametrization of the two soliton manifold across multiplicities.

Following the pattern in the previous sections, we begin by 
considering $\gamma_1$, $\gamma_2$ of the form
\begin{equation} \label{gammaj} 
\gamma_j = i(\beta_0 + \beta_1 z_j + \beta_2 z_j^2 +\beta_3 z_j^3)
\end{equation} 
with real coefficients $\beta_k$. Then in terms of the elementary sysmmetric polynomials $s_1$
and $s_2$
\[ 2\gamma =   i(\beta_0 + \beta_1( z_1+z_2) + \beta_2 (z_1^2+z_2^2) + \beta_3 ( z_1^3 + z_2^3)) 
= i(\beta_0 + \beta_1 s_1 + \beta_2 s_2 + \beta_3 s_1( \frac32 s_2- \frac12s_1^2) 
\] 
\[ \gamma_1-\gamma_2 = i (\beta_1( z_1-z_2) +  \beta_2 (z_1^2 - z_2^2) + \beta_3( z_1^3 - z_2^3) ) 
= i (z_1-z_2)(\beta_1 + \beta_2 s_1 + \beta_3 (\frac12 s_2+\frac12 s_1^2 )) 
\] 
\[ 
\mu_1 = \cosh(2\gamma), \quad \mu_2 = \sinh(2\gamma), \quad  \mu_3 = \cosh(\gamma_1-\gamma_2) , \quad \alpha= \frac{\sinh(\gamma_1-\gamma_2)}{z_1-z_2}.  
\] 

Then we can view the above $w$ as 
\[
w = w(\bfs, \bbeta), \qquad \bfs = (z_1+z_2,z_1^2+ z^2_2), \quad \bbeta = (\beta_0,\beta_1,\beta_2,\beta_3),
\]
where for $\bfz$ we use the topology defined by the symmetric polynomials.
For this function we have

\begin{lemma}
For $z_1,z_2$ in a compact subset of the upper half-space, the 
function $w$ is a uniformly smooth function of $(\bfs,\bbeta)$. Furthermore,
we have the uniform bound
\begin{equation}
| \partial_{\bfs}^a \partial_{\bbeta}^b w | \lesssim |w_0|,    
\end{equation}
where
\[
w_0 = \frac{(1+|\alpha|)\cosh(2\real \gamma) }{\cosh( 4 \real \gamma) + |\alpha|^2 }.
\]
\end{lemma}
\begin{proof}
The proof is straightforward. On one hand we know that  $w_0$ is of the same size as the one defined in \eqref{w-size}.  The bounds \eqref{w-size} and  \eqref{dw-size} and
\[ |\partial^a_{\beta_j, s_1,s_2}  \mu_j|  \lesssim |\mu| \] 
imply the uniform bounds. These bounds for $\mu_1$ and $\mu_2$ are obvious. Both $\cosh(\gamma_1-\gamma_2)$
and $ \alpha= \frac{\sin(\gamma_1-\gamma_2)}{z_1-z_2}$ are even and analytic as functions of $z_2$ and $z_1$, and hence they are holomorphic functions of $s_1$ and $s_2$. The bounds on derivatives then follow by Cauchy's integral formula on balls around $s_1$ and $s_2$.
\end{proof}

We now describe two soliton states. Relative to the $(\bfz,\bbeta)$ parametrization with $ \beta_j \in \R$, this corresponds to choosing 
\[
Q_{\bfz,\bbeta}(x) = w(\bfz,\tilde \bbeta),
\]
where 
\[
\tilde \bbeta = (\beta_0,\beta_1+x, \beta_2,\beta_3),
\]
and the corresponding  $\gamma_j$'s are 
\[
\gamma_j = i ( \beta_0 + (\beta_1+x)z_j + \beta_2 z_j^2 + \beta_3 z_j^3).
\]
To compare with our general set-up, the  associated scattering parameters $\kappa_j$ are 
\[
\kappa_j = i ( \beta_0 + \beta_1 z_j + \beta_2 z_j^2 + \beta_3 z_j^3).
\]

In particular we have
\begin{equation}\label{eq:gamma} 
2\gamma = i(2\beta_0 + (\beta_1+x) (z_1+z_2) + \beta_2 (z_1^2+z_2^2) +\beta_3(z_1^3+z_2^3))
\end{equation} 
and 
\begin{equation} \label{gamma0} 
\gamma_0:= \frac{\gamma_1-\gamma_2}{z_1-z_2} = i(\beta_1 + x + \beta_2 (z_1+z_2) +\beta_3(z_1^2+z_1 z_2 +z_2^2)).
\end{equation} 
Next we consider the location of the two bumps for the $2$-solitons. We begin with the location of the 
single bumps for the corresponding $1$-solitons with the same spectral  and scattering parameters, whose centers are given by
$x_1$, $x_2$ determined by
\begin{equation}\label{xj}  
\real \gamma_j = 0 \Longleftrightarrow x_j=  - \frac{\im \kappa_j}{\im z_j} = -\frac{\im (\beta_1 z_j +  \beta_2 z_j^2 + \beta_3 z_j^3)}{\im z_j} .   
\end{equation} 
For later considerations we denote their phase at the center of soliton by 
\begin{equation}\label{thetaj}   \theta_j = \beta_0 +  (\beta_1 + x_j ) \real z_j + \beta_2 \real z_j^2 + \beta_3 \real z_j^3. 
\end{equation}   
Then
\[
\gamma_j = i ( \theta_j + z_j(x-x_j))
\]
and 
\[ 
\gamma_0 = i  \frac{\theta_1-\theta_2 + z_1(x-x_1) - z_2(x-x_2) }{z_1-z_2}. 
\]
Recall that
\[
\alpha = \frac{\sinh((z_1-z_2) \gamma_0)}{z_1-z_2}, \qquad 
\sigma = (z_1-z_2) \coth ((z_1-z_2) \gamma_0).
\]
The two bumps are centered (recall \eqref{nearmax}) for $|z_1-z_2|\ll \im z_1$) where 
\begin{equation}\label{bumps}
\cosh(2\real \gamma) \approx |z_1-\bar z_2| |\alpha|.
\end{equation}
Here the expression $\real \gamma$ is linear in $x$, decreasing at a uniform rate, and vanishing at a point $x_0$, which is related to 
$x_1$ and $x_2$ by the relation
\begin{equation} \label{centerx} 
x_0 =- \beta_1 - \frac{ \beta_2 \im (z_1^2 + z_2^2) + \beta_3 \im (z_1^3+ z_2^3)}{\im z_1 +\im z_2}    = \frac{\im z_1}{\im (z_1+z_2)} x_1 + \frac{\im z_2}{\im (z_1 +z_2)} x_2 
\end{equation} 
which can be seen as the center of mass of the $2$-soliton state.

We also define an averaged phase at the center by 
\begin{equation} \label{centerphase} 
2 \theta = \theta_1+\theta_2 + \frac{\im z_1 \real z_2 - \im z_2 \real z_1}{\im(z_1+z_2)}(x_1-x_2),
\end{equation} 
in order to have
\begin{equation}\label{def-center}
2\gamma = i (2 \theta + (z_1+z_2)(x-x_0)).
\end{equation}

Moreover,  $\cosh(2\real \gamma)$ grows at uniform exponential rates away from $x_0$. On the other hand, $\alpha$ has a smaller logarithmic
derivative,
\[
\frac{\partial \alpha}{\partial x} = i \alpha \sigma = O( |\alpha||z_1-z_2| + 1)
\]
As a consequence, if $|\alpha(x_0)|$ is large then there are exactly two unit size regions where  \eqref{bumps} (where we consider both sides as functions of $x$) is satisfied. Furthermore, in this region
the coefficients $c_+$ and $c_-$ are slowly varying, as 
\[
\frac{\partial \sigma}{\partial x} = i \alpha^{-2}
\]
This implies that, with $O(|\alpha^{-2}|)$ accuracy, the maximum points
of $w$ are described by the relation
\begin{equation}\label{bumps+}
\cosh(2\real \gamma) = |z_1-\bar z_2| |\alpha|.
\end{equation}

To accurately calculate the roots of \eqref{bumps+} it is useful
to consider the value of $\alpha$ at the center $x_0$.
This is determined by 
\begin{equation} \label{gamma00}  
(z_1-z_2) \gamma_{00} := (z_1-z_2) \gamma_0(x_0) = i \left(  \theta_1-\theta_2 + \frac{z_1 \im z_2 + z_2 \im z_1}{2\im z} (x_2 -x_1) \right),  
\end{equation}  
and in particular 
\begin{equation} \label{realpha0}
\real ((z_1-z_2)\gamma_{00}) = \frac{\im z_1 \im z_2}{\im z}(x_2-x_1).
\end{equation} 
Thus  we define 
\begin{equation} \label{alpha0}  
\alpha_0 = \frac{\sinh((z_1-z_2)\gamma_{00})}{z_1-z_2}, 
\qquad \sigma_0 = (z_1-z_2) \coth((z_1-z_2)\gamma_{00}),
\end{equation} 
where $\gamma_{00}$ is linear in $\beta_2$ and $\beta_3$,
\[
\gamma_{00} = \gamma_0(x_0) =: a_2 \beta_2 + a_3 \beta_3
\]
where the coefficients $a_2$, $a_3$ are symmetric functions 
in $z_1,z_2$, given by (denoting $z_1+z_2=2z$)
\begin{equation} \label{def-a2}
\begin{split}
a_2 = i (z_1+z_2) - i \frac{\im (z_1^2+z_2^2)}{\im (z_1+z_2) }
 = - 2\im z -i \frac{\real(z_2-z_1)\im(z_2-z_1)}{2\im z}
\end{split}
\end{equation}
\begin{equation} \label{def-a3}
\! \begin{split}
a_3 &\,= i   (z_1^2 + z_1 z_2 +z_2^2) - i \frac{\im (z_1^3+z_2^3)}{\im (z_1+z_2) } 
 \\
&\, = - \frac32 \real(z_1\!+ \! z_2) \im (z_1\!+\!z_2) - \frac{i}{2} [\im (z_1\!+\!z_2)]^2
+ \frac{i}{4}(z_1\!-\!z_2)^2 - \frac{3i}4 \frac{\im [(z_1+z_2)(z_1-z_2)^2]}{\im (z_1+z_2)}
\end{split} \!\!\!
\end{equation}

These can be checked to be linearly independent over $\R$, for instance 
by verifying that the following expression is nonzero ( we write $ z_j = x_j+i y_j$)
\[
\begin{split} 
J = &\,  \im(z_1+z_2)\im (a_2 \bar a_3) 
\\ = &\,  \im(z_1+ z_2)\Big[\im (z_1+z_2)\real (z_1^2+ z_1 z_2 + z_2^2) - \im (z_1^3 + z_2^3)\Big] 
\\ & - \Big[\real (z_1 + z_2)\im (z_1+z_2) - \im (z_1^2+ z_2^2)\Big]\im (z_1^2+z_1z_2 + z_2^2) 
\\ = &\,  (y_1+y_2)^2(-\frac12(x_1-x_2)^2 -2y_1y_2) - \frac32 (y_1+y_2)(x_1-x_2) (y_1-y_2)(x_1+x_2)
\\  &
+ (x_1-x_2)(y_1-y_2)\Big[\frac32 (y_1+y_2)(x_1+x_2)  +\frac12 (y_1-y_2)(x_1-x_2)\Big]  
\\  = & \, -2 y_1 y_2 ((x_1-x_2)^2 + (y_1+y_2)^2 )
\\ = & \,  - 2 \im z_1 \im z_2 |z_1-\bar z_2|^2.
\end{split}
\]

This implies that
\[
|\gamma_{00}| \approx |\beta_2| + |\beta_3|.
\]

In particular we will be interested in $\real ((z_1-z_2) \gamma_{00})$ (see \eqref{gamma00}) , which can be alternatively expressed in the form
\[
\real ((z_1-z_2) \gamma_{00}) = \frac{\im z_1 \im z_2}{\im z}(x_1-x_2).
\]
On the other hand for the imaginary part we get
 \[
 \im ((z_1-z_2) \gamma_{00}) = \theta_1 - \theta_2 + 
 \frac{\im(z_1 z_2)}{2 \im z}(x_2-x_1).
 \]

As above we distinguish two scenarios, still assuming $|z_1-z_2| \ll  1$:
\begin{enumerate}[label=(\roman*)]
\item {\bf Single bump case.} This becomes
\[ \text{ dist } (\gamma_{00} , i \pi (z_1-z_2)^{-1} \Z) \lesssim 1  \]  
and  corresponds to $2$-solitons $Q_{\bfz,\bbeta}$ which have two overlapping solitons. 
 The amplitude of $w$ could get as high as $2 \im(z_1+z_2) $, which value is attained at $x = x_0$ when $\beta_2=\beta_3 = 0$. In this case we have exponential decay away from $x_0$ and the two soliton is close to a 2-soliton with $z_1=z_2$ since the 2-soliton depends smoothly on $s_1$ and $s_2$.  
\item {\bf Two bumps case.}  
\[ \text{ dist } (\gamma_{00} ,  i \pi (z_1-z_2)^{-1} \Z) \gg 1 , \] 
which corresponds
to $2$-solitons $Q_{\bfz,\bbeta}$ which have
two simple bumps, with amplitudes closer to the range between $2\im z_1$ and $2\im z_2$. 

\end{enumerate}

In the second case above, we seek a more accurate description 
of the location of the two bumps, which are given by the relation
\eqref{bumps+}, which  translates to 
\[
\cosh(\im (z_1+ z_2)(x-x_0)) = |z_1 - \bar z_2| \left| \frac{\sinh((z_1-z_2)(\gamma_{00}+i(x-x_0))}{z_1-z_2} \right|.
\]
Based on the discussion above, this equation will have two roots,
one above and one below $x_0$.  

In a first approximation we evaluate the size of $|x-x_0|$ for 
the two roots by
\[
|x-x_0| \approx \log \left|\frac{\sinh((z_1-z_2)\gamma_{00})}{z_1-z_2} \right|.
\]
In this region we take a Taylor expansion of $\ln \alpha$,
\[
\ln \alpha = \ln \alpha_0 + i (x-x_0)\sigma_0
+ O(|x-x_0|^2 |\alpha_0|^{-2}).
\]
At the roots $x$ this leads to 
\[
2\im z |x-x_0| = \ln |z_1-\bar z_2| +
\ln (2|\alpha_0|) - (x-x_0) \im \sigma_0
+ O(\frac{\ln^2|\alpha_0|}{|\alpha_0|^2}), 
\]
and finally to 
\[
x-x_0 = \pm \frac{\ln |z_1-\bar z_2| +
\ln (2|\alpha_0|)}{ 2\im z \pm \im \sigma_0}+ O(\frac{\ln^2|\alpha_0|}{|\alpha_0|^2}).
\]
This gives the approximate locations of the centers for the two bumps as
\begin{equation}\label{app-bumps}
 x^{\pm} = x_0   \pm \frac{\ln |z_1-\bar z_2| +
\ln (2|\alpha_0|)}{ 2\im z \pm \im \sigma_0}, 
\end{equation}
with accuracy $O(\epsilon)$ where 
\begin{equation}
\epsilon =  \frac{\ln^2|\alpha_0|}{|\alpha_0|^2} . 
\end{equation}

\bigskip

We remark that when $\real ((z_1-z_2) \gamma_{00})  \gg 1$ 
(which corresponds to $x_1-x_2 \gg 1$)
we can approximate
\[
\sigma_0 \approx z_1-z_2, \qquad \ln (2\alpha_0) \approx 
(z_1-z_2) \gamma_{00} - \ln (z_1-z_2) 
\]
at the expense of allowing larger errors of size
\[
e^{-2\real (z_1-z_2)\gamma_{00}} \approx (|z_1-z_2| |\alpha_0|)^{-2}.
\]

This yields with $z_+=z_1\approx \frac12(z +  \sigma_0)$, $z_-=z_2\approx \frac12(z-  \sigma_0) $ 
\[
x-x_0 = \pm \frac{\ln |z_1- \bar z_2| +
\ln 2(|\alpha_0|)}{ 2\im z_\pm }+ O\Big(\frac{\ln^2|\alpha_0|+ |z_1-z_2|^{-2}}{|\alpha_0|^2}\Big).
\]
Then the larger solution can be associated to $z_1$,
\[
\hat x_{{1}} \approx x_0 + \frac{\ln |z_1-\bar z_2| - \ln |z_1-z_2|
+ \real  [(z_1-z_2) \gamma_{00}]}{ 2 \im z_2} = x_2 + \frac{\ln |z_1-\bar z_2| - \ln |z_1-z_2|} { 2 \im z_2},
\]
and the smaller one can be associated to $z_1$, using \eqref{centerx} and \eqref{realpha0}, 
\[
\hat x_{{ 2}}  \approx x_0 - \frac{\ln |z_1-\bar z_2| - \ln |z_1-z_2|
+ \real [(z_1-z_2) \gamma_{00}]}{ 2 \im z_1} = x_1 - \frac{\ln |z_1-\bar z_2|- \ln |z_1-z_2|}{ 2 \im z_1}.
\]
These formulas agree with \eqref{asym} and the ones in the introduction.

\bigskip

Our next objective is to determine the associated effective 
spectral parameters $z^{\pm}$ and phase parameter $\theta^{\pm}$.
These are determined also with $\epsilon$ accuracy as follows:

\begin{enumerate}
    \item[(i)] The imaginary parts $\im z^{\pm}$ correspond to the 
amplitudes of the two bumps, 
\[
2 \im z^{\pm} \approx |Q(x^\pm)|.
\]
\item[(ii)] The real parts correspond to the frequencies near the two bumps,
\[
2 \real z^{\pm} \approx |Q^{-1}(x^{\pm})| \im \partial_x Q(x^{\pm}).  \]

\item[(iii)] The phases correspond to the arguments of $Q$ at $x^{\pm}$,
\[
2\theta^{\pm} = \arg{Q(x^{\pm})}.
\]
    
\end{enumerate}

We now proceed to compute the three quantities. We recall 
that, with
\[
\begin{split}
 c_+ = & \  i [|z_1+z_2|^2 - 2(z_1 z_2 + \bar z_1 \bar z_2)-(\bar z_1-\bar z_2)^2] + 2 \im(z_1+z_2) \bar \sigma,
  \\
 c_- = & \  i [|z_1+z_2|^2 - 2(z_1 z_2 + \bar z_1 \bar z_2)-( z_1- z_2)^2] + 2 \im(z_1+z_2)  \sigma,
 \end{split}
\]
we have 
\begin{equation}\label{w-app1-re}
Q(x) = e^{2 i \im \gamma} 
\frac{ c_+ \bar\alpha e^{2\real \gamma} + c_- \alpha e^{-2 \real \gamma}}
{2 \cosh(4 \real \gamma) + 2|z_1-\bar z_2|^2 |\alpha|^2} 
+O(\frac{1}
{\cosh(4 \real \gamma) +  |\alpha|^2} ).
\end{equation}
For the amplitude near $x^{+}$ we have $e^{-2\real \gamma}$ as the leading factor at the numerator, so we further simplify this as
\[
Q =  e^{2i \im \gamma}  c_- \frac{\alpha}{|\alpha|}
   \frac{  \cosh(2 \real \gamma) |\alpha|}
{\cosh(2 \real \gamma)^2 + |z_1-\bar z_2|^2 |\alpha|^2 }
+O(\frac{1}
{\cosh(4 \real \gamma) +  |\alpha|^2} ).
\]
Similarly, near $x^{-}$ we have 
\[
Q   =   e^{2i \im \gamma} c_+ \frac{\bar \alpha}{|\alpha|}
   \frac{  \cosh(2 \real \gamma) |\alpha|}
{\cosh(2 \real \gamma)^2 + |z_1-\bar z_2|^2  |\alpha|^2 } 
+O(\frac{1}
{\cosh(4 \real \gamma) +  |\alpha|^2} ).
\]
At the bump center $x^{+}$ (the approximate center is accurate enough) 
we can  also replace $\sigma$ by $\sigma_0$
given by
\[
\sigma_0 := \sigma(x_0) = (z_1-z_2) \coth((z_1-z_2)\gamma_{00}),
\]
to obtain, with slightly larger $\epsilon$ errors, 
\[
|Q(x^{+})| \approx \frac{|c_{-}|}{2|z_1-\bar z_2|}
\approx  \frac{|[|z_1+z_2|^2 - 2(z_1 z_2 + \bar z_1 \bar z_2)-(z_1-z_2)^2] - 2i \im(z_1+z_2) \sigma_0|}{2|z_1-\bar z_2|}.
\]

Here we can rewrite the expression at the numerator as (using again $2z=z_1+z_2$)
\[
4 \im z( 2\im z+ \im \sigma_0)
+ 2i [(\real (z_1 - z_2)\im (z_1 - z_2) + 2\real \sigma_0  \im z],
\]
and, with $\epsilon$ errors, we replace $(\real z_1 -\real z_2)(\im z_1 -\im z_2)$ by $\real \sigma_0 \im \sigma_0$ to get
\[
Q(x_+) \approx  \frac{|2( 2\im z+ \im \sigma_0) (2\im z - i \real \sigma_0)|}{|z_1-\bar z_2|}.
\]
Taking the square norm we replace back $ |\real \sigma_0|^2$
by $|\real (z_1 - z_2)|^2$. Then we get 
\begin{equation} 
|Q(x^{+})| =   2 \im z + \im \sigma_0  + O(\epsilon).
\end{equation} 
A similar computation yields
\[
|Q(x^{-})| =   2 \im z - \im \sigma_0 + O(\epsilon). 
\]
Based on this, we define the imaginary part of the effective 
spectral parameter for the bumps as  
\begin{equation}
\label{app-amplitude}
2 \im z^\pm =  2 \im z \pm \im \sigma_0.   
\end{equation}

\bigskip

Next we consider the effective frequency parameter. Near $x^+$ we have 
\[
\im Q^{-1} Q_x  =   2 \real z
+ \im (\alpha^{-1} \partial_x \alpha) + \im (c_+^{-1} \partial_x c_+) + 
O(\epsilon).
\]
The last term has size $O(\epsilon)$ and can be placed into the error.
For the middle term we compute
\[
\alpha^{-1} \partial_x \alpha =
\frac{i(z_1-z_2)\cosh((z_1-z_2)\gamma_0)}{ \sinh((z_1-z_2)\gamma_0)} = i \sigma_0.
\]
We can again freeze $\sigma$ to $\sigma_0$. This 
yields the approximate effective frequencies
\begin{equation}\label{app-frequency}
  z^\pm :=    z \pm \frac{\sigma_0}2.  
\end{equation}
\bigskip

Finally, we consider the phase, which at the maximal amplitude near $x^{+}$ respectively $x^-$ is given by 
\[
2 \theta^\pm \approx   2 \im \gamma \pm \arg(\alpha)  
+ \arg(c_{\mp}).
\]
We evaluate the three components. For $\gamma$ we have
\[
2 \im \gamma(x^{\pm}) = 2 (\theta + 
(x^{\pm}-x_0) \real z),
\]
For $\alpha$, using logarithmic derivatives,
\[
 \arg(\alpha) \approx \arg \alpha_0 + 
 (x^\pm - x_0) \real \sigma \approx \arg \alpha_0 + 
 (x^\pm - x_0) \real \sigma_0.
\]
Finally, for $c_{\mp}$, reusing some of the computations we 
did for the effective frequency,
we have 
\[
\arg (c_{\mp}) \approx  \arg(2 i \im z \pm \real \sigma_0)
= \arg(z_\pm -\bar z_\mp).
\]

We conclude that the phase is
\[
2 \theta^\pm \approx   2 \theta
\pm \arg(\alpha_0) + (x^{\pm}-x_0)
[\real(z_1+z_2) \pm \real \sigma_0] +
 \arg(z_\pm -\bar z_\mp)
\]
which we rewrite  as 
\begin{equation} \label{app-phase}
2 \theta^\pm =    2 \theta
\pm \arg(\alpha_0) + 2(x^{\pm}-x_0)
\real z^{\pm} \pm  \arg(z^{\pm} - \bar z^{\mp}).
\end{equation}

One can now match this with the known asymptotics for separated 
$x_1, x_2$ in \cite{MR905674}, see \eqref{app-frequency}. There we have 
\[
\real \sigma_0 = \pm \real(z_1-z_2).
\]
\begin{itemize}
\item The expression coming from the last term
\[
\im \ln(z_1-\bar z_2) \text{ or } \im \ln(z_2-\bar z_1)
\]
is one part of what we expect.

\item The term $\arg \alpha_0$ has two components,
$\im \ln (z_1-z_2)$ which we expect, and 
\[
\pm \im [(z_1-z_2) \gamma_{00}] \approx 
\pm [\theta_1-\theta_2+ \frac{\im(z_1 z_2)}{\im(z_1+z_2)}(x_1-x_2) ].  
\]
The first term combines with the first term of $2\theta$ to give 
$\theta_1$ or $\theta_2$. The second term combines with the 
second term of $\theta$ and with the expression $(x^{\pm}-x_0)
[\real(z_1+z_2) \pm \real \sigma_0]$ with $x^{\pm}$ replaced by 
$x_1$ or $x_2$ and $\sigma_0$ as above, and they all cancel.

\item We are left with the extra error coming from the substitution
$x^{\pm}$  by $x_1$ (or $x_2$) which is 
\[
\real z_1 (x_1 - \hat x_1) 
\]
which does not appear in the earlier asymptotics. But this is 
simply a matter of notations, i.e. in our 
computations the new phase is evaluated  at $\hat x_1$,
whereas in \cite{MR905674} it is evaluated at $x_1$.
 One  could also choosing the center of mass as a reference point,
in which case the phase adjustment would be
\[
(x^{\pm} -x_0) z^{\pm} + i (\theta^{\pm} - \theta)
\approx -\frac{\pi}{2} \pm \ln \alpha_0  \pm \ln 2(z^{\pm}-\bar z^{\mp})
\]
But this is a less stable computation.
\end{itemize}

Summarizing the outcome of the analysis in this section, we have proved the following:

\begin{theorem}\label{approximate} Let $ \sigma_0$ and $\alpha_0$ be defined as in \eqref{alpha0}, 
\[
\alpha_0 = \frac{\sinh( (z_1-z_2) \gamma_{00})}{z_1-z_2}, \quad 
\sigma_0 = (z_1-z_2) \frac{\cosh((z_1-z_2)\gamma_{00})}{\sinh((z_1-z_2)\gamma_{00})},
\] 
where 
\[ 
\begin{split} 
\gamma_{00}  =  &\, \Big( -2 \im z - i\frac{ \im(z_1-z_2)\real(z_1-z_2)}{2\im z}\Big) \beta_2 
\\ & +  \Big( -3 \im z^2 -2i (\im z)^2 + \frac{i}4 (z_1-z_2)^2   -\frac{3i} 4 \frac{\im \Big(z(z_1-z_2)^2\Big)}{\im z}\Big) \beta_3.
\end{split} 
\]
Let $z_{\pm}$, $x_{\pm}$ and $\theta_{\pm}$
be defined by \eqref{app-frequency}, \eqref{app-bumps} respectively 
\eqref{app-phase}, i.e.
\[
z_{\pm}= z \pm  \frac{\sigma_0}2,
\] 
\[
x_{\pm} = x_0 \pm \frac{\ln |z_1-\bar z_2| + \ln 2|\alpha_0|} {2 \im z_\pm},  
\] 
\[
\theta^{\pm}= \theta + (x_\pm -x_0) \real z^\pm \pm \frac{\arg \alpha_0+ \arg (z_\pm -\bar z_\mp )}2. 
\]  
Then $Q$ is a sum of two solitons with a small error. 
\begin{equation}
   Q(x)   =  \frac{ 2 e^{i \theta_++ 2i \real z_{+} ( x-  x_+)}}{\im z_+}   \sech\Big( \frac{2(x-x_+)}{\im z_+}  \Big) + \frac{2 e^{i \theta_-+ 2i \real z_{-} ( x-  x_-)}}{\im z_-}  \sech\Big( \frac{2(x-x_-)}{\im z_-}  \Big)
   + O \Big(\frac{ \ln^2 |\alpha_0|}{ |\alpha_0|^2 }  \Big). \! 
\end{equation}
\end{theorem}

The solitons are given as a function of $x$ by the quotient of $A_0$ in \eqref{A0} and $D_0$ in \eqref{D0} with $\gamma_j$ and $\gamma$ defined in  \eqref{gammaj} and \eqref{eq:gamma}, $\gamma_0$ in \eqref{gamma0}, $\gamma_{00}$ in \eqref{gamma00}, \eqref{def-a2}, \eqref{def-a3}, $\alpha_0$ and $\sigma_0$ in \eqref{alpha0}   , $x_j$ and $\theta_j$ in \eqref{xj} and \eqref{thetaj}, $x_0$ in \eqref{centerx} and $\theta$ in  \eqref{centerphase}.

In other words, the above approximation has errors which are
not only exponentially small in the distance between the bumps,
but also uniformly small as $z_1-z_2 \to 0$, and accurate enough to capture the leading order interaction between the two bumps.


\bigskip


\subsection{ A uniform parametrization of the 2-soliton manifold}
We have seen that the set of  pure $2$-solitons is a uniformly smooth manifold in $L^2$, or more generally in $H^s$ for $s >-\frac12$. In this section we will use Proposition  \ref{approximate}
to provide concrete uniform parametrizations. We will also discuss nonuniform parametrizations.
We begin by discussing several ways we can smoothly parametrize the  $2$-soliton manifold.
\medskip

\begin{enumerate}[label=\alph*)]

\item Using the variables 
\[
(\bfz,\bbeta),
\]
employed earlier in the paper in the general case of $N$-solitons.
Here $\bbeta$ describes 
the correspondence between $Q_{\bfz,0}$ and $Q_{\bfz,\bbeta}$
using the first four flows. This is the simplest description, 
but it is only uniform in the region $|\bbeta| \lesssim 1$.  Here we need to take a double quotient space for $\bbeta$, namely modulo $\kappa_j \in \pi i \Z$.

\medskip

\item  Centering the $2$-soliton around the center of mass and phase\footnote{We use this terminology for convenience here,
but the notion of \emph{center of phase} does not seem to be  well-defined outside of the $2$-soliton manifold.}
$(x_0,\theta)$ (which can be viewed as associated to the global translation and phase shift symmetries), we can instead use 
the following set of parameters:
\[
(\bfz, x_0,\theta,\gamma_{00}).
\]
Here $\gamma_{00}$ is linearly equivalent to $\beta_2$ and $\beta_3$. In this case the quotient structure decouples partially. Precisely,
we have 
\[
\theta \in \R \quad (\mod \pi), \qquad \gamma_{00} \in \C \quad \quad (\mod (z_1-z_2)^{-1} \pi i),
\]
but with the nontrivial gluing
\[
(\theta, \gamma_{00}+\frac{\pi i}{z_1-z_2}) \leftrightarrow
(\theta+\frac{\pi}2, \gamma_{00}).
\]

\medskip

\item By the set of parameters $P$ of pairs: 
\[
(z, x_0,\theta , (z_1-z_2)^2, \alpha_0,\mu_{0})
\]
obtained by replacing the parameter $\gamma_{00}$ by its hyperbolic functions,
\[
\alpha_0 = \frac{\sinh((z_1-z_2)\gamma_{00})}{z_1-z_2}, \qquad \mu_0 = 
\cosh((z_1-z_2)\gamma_{00})
\]
which lie on the smooth manifold 
\[
\mu_0^2 - (z_1-z_2)^2 \alpha_0^2 = 1.
\]

Here we have one remaining symmetry
\[
(\theta,\alpha_0,\mu_0) \to (\theta+\frac{\pi}2,-\alpha_0,-\mu_0). 
\]
Alternatively, away from $\alpha_0= 0$ one can replace $\mu_0$
by $\sigma_0$ given by 
\[
\sigma_0 = (z_1-z_2) \coth((z_1-z_2)\gamma_{00}).
\]

\medskip 
\item   We can also parametrize the $2$-solitons   by  the set of (approximate) effective parameters
\begin{equation}\label{uniform}   
\Big(z_-, z_+ , x_-, x_+, \theta_-,\theta_+\Big) 
\end{equation}  
 where $\theta_{\pm}  \in  \R / (\pi \Z)$, provided the solitons are well separated.  Here we  describe the two soliton set for separated solitons by their approximate position and their phases.  

 The set of pure two solitons is a uniformly smooth manifold by 
 Theorem \ref{t:pure}. The sum of two solitons with the effective parameters is clearly a uniformly smooth manifold with the uniform parametrization by these parameters. Since the Hausdorff distance between the set of pure 2 solitons and the sum of the two solitons is close in $L^2$ as well as in any other Sobolev space $H^s$ we see that we obtain a uniformly smooth parametrization of the pure 2 solitons in the well separated regime.

This will turn out to be uniform, but it is defined only for separated solitons. 
We can also use half of the above parameters along with the center
parameters
\[ 
(z, z_+ , x_0, x_+, \theta,\theta_+). 
\] 
which can be defined via the relations
\[
z = \frac{z_+ + z_-}{2}
\]
\[
x_0 = \frac{x_+ \im z_+ + x_- \im z_-}{2 \im z}
\]
\[
\theta = \frac{\theta_+ +\theta_-}2
-  (x_+ -x_0) \real z^+ -  (x_- -x_0) \real z^- - \arg(z_+-z_-) - \frac{\pi}2
\]


\end{enumerate}

In order to describe a uniform parametrization we distinguish two cases.

\bigskip 

I) \textbf{The double bump region.} This corresponds to 
$|\beta_2|+|\beta_3| \lesssim 1$ in (a) or equivalently to 
$|\gamma_{00}| \lesssim 1$ in (b), to $|\alpha_0| \lesssim 1$
in (c) but is not covered by (d). Here matters are simple 
because the uniform topology is used in the three cases  (a), (b) and (c).

\bigskip

II) \textbf{Separated bumps.} This is the region covered in (d),
where we the metric is simply equivalent to the euclidean metric,
\[
g = d z_{\pm}^2 + d x_{\pm}^2 + d \theta_{\pm}^2
\]
We next recast this metric in terms of the parametrization in 
(c). We have
\[
2z_{\pm} = 2z \pm \sigma_0 
\]
\[
x_{\pm} = x_0 \pm \frac{\frac12 \ln ( 4 \im z^2 +  \real(\frac{1}{\alpha_0^2}-\sigma_0^2)) + \ln (2 |\alpha_0|)}{2 \im z_{\pm}} 
\]
\[
\theta_{\pm} = \theta \pm \arg \alpha_0 + 2 x^\pm \real z^{\pm} \pm 
\arg( z^\pm -z^\mp )
\]
We take the uniform coordinates  \eqref{uniform}  and the corresponding standard metric in these coordinates -
recall that search for uniform estimates for $ z_1, z_2 $ in compact subset of the open upper half plane. 
We write the metric on an $8$ dimensional set in a schematic fashion as 
\[
d z_{\pm}^2 + d x_\pm^2 +  d \theta_{\pm}^2 
\]
and seek to express it in equivalent form in terms of the variables $z,\alpha_0$ and  $\mu_0$.  To be more precise we set up some notation for this section. For a real function $f$ $df$ denotes the differential and $df^2$ the quadratic form 
\[ (y_1,y_2) \to  df^2(x)(y_1,y_2) :=  (df(x)y_1) df(x)y_2\] 
and similarly for vector valued functions $F$ 
\[ d F^2(y_1,y_2) = \langle dF(x)y_1, dF(x)y_2\rangle. \]
We identify maps to $\C$ with the corresponding map to $\R^2$.

We shall see in the end that the metric tensor is at least as large as the standard metric. This will allow to neglect some terms. we write 
\[ dF \sim dG    \Longleftrightarrow   dF^2 \sim dG^2 \] 
if the Gram matrices have small distance, i.e. 
$ \Vert dF^T dF - dG^T dG \Vert \lesssim 1 $. 

For $z_{\pm}$ we have the obvious relation
\[
dz_-^2 + dz^2_{+} \approx dz^2 +  d \sigma_0^2.
\]
Next we consider $x_{\pm}$, for which we harmlessly discard the middle component, 
\[
dx_{\pm} \sim d x_0 \pm  d ( \frac{\ln |\alpha_0|}{\im z_{\pm}} )
= dx_0 \pm \left(\frac{1}{\im z_{\pm}} d \ln |\alpha_0| - \frac{\ln |\alpha_0|}{\im^2 z_{\pm}} d \im z_{\pm} \right). 
\]
We multiply the '+' equation by $\im z_+$, the '-' equation  by $\im z_-$ and add
\[
\! 2\im z dx_0 - \ln |\alpha_0| \left( \frac1{\im z_{+}} d \im z_{+} - 
\frac{1}{\im z_{-}} d \im z_{-}\right) 
= 2\im z dx_0 - \frac{\ln |\alpha_0|}{\im z_+ \im z_{-}}
\left( 2\im z d \sigma_0 - \sigma_0 d \im z \right)\!\!
\]
Discarding the $\im z_{\pm}$ denominators we are left with
\begin{equation}\label{unif1}
 \im z_- dx_- +  \im z_+ dx_+ \sim  \im z ( 4\im^2 z -\im^2\sigma) dx_0 -2 \ln |\alpha_0| \left( 2\im z d \im \sigma_0 - \im \sigma_0 d \im z     \right)
\end{equation}
Next we take the difference  to obtain
\[
dx_+-dx_-\sim \frac{1}{\im^2 z_{+} \im^2 z_-}\left( 2 \im z( \im^2 z - \frac14 \im^2 \sigma)
 d \ln |\alpha_0| - \ln |\alpha_0| (\im^2 z_- d \im z_- + \im^2 z_+ d \im z_-)\right)\hspace{-22pt} 
\]
and discarding the fraction 
\begin{equation}\label{unif2}
 \sim 2 \im z( \im^2 z - \frac14 \im^2 \sigma)
 d \ln |\alpha_0| - 2\ln |\alpha_0| \left((\im^2 z + \frac14 \im^2 \sigma_0) d \im z -  \im z \im \sigma_0  d \im \sigma_0\right)
\end{equation}

We repeat the same computation for $\theta_\pm$. We can 
harmlessly discard the last term, as well as the $dx_{\pm}$ component, leaving us with the equivalent form
\[
d\theta_{\pm} \sim d \theta \pm (d \arg \alpha_0 +  \frac{\log |\alpha_0|}{\im z_{\pm}} d\real z_{\pm})
\]
Adding the $\pm$ forms yields
\[
2 d\theta+ \ln |\alpha_0| \left( \frac1{\im z_{+}} d \real z_{+} - 
\frac{1}{\im z_{-}} d \real z_{-}\right) 
= 2d\theta+
\frac{\ln |\alpha_0|}{\im z_+ \im z_{-}}
\left( 2\im z d \real \sigma_0 - \im \sigma_0 d \real z     \right)
\]
Discarding the $\im z_{\pm}$ denominators we are left with
\begin{equation}\label{unif3}
d\theta_++d\theta_- \sim 2( \im^2 z -\frac14 \im^2 \sigma_0) d\theta +
\ln |\alpha_0| \left( 2\im z d \real \sigma_0 - \im \sigma_0 d \real z     \right)
\end{equation}
On the other hand taking the difference we obtain
\[
d\theta_+-d\theta_- \sim  2 d \arg(\alpha_0) + \frac{2\ln |\alpha_0|}{\im z_+ \im z_-} 
 ( \im z d \real z - \frac14 \im \sigma_0 d \real \sigma_0)
\]
and eliminating the denominators
\begin{equation}\label{unif4}d\theta_+-d\theta_-  \sim 
 2(\im^2 z - \frac14 \im^2 \sigma_0) d \arg(\alpha_0) + 2\ln |\alpha_0| 
 ( \im z d \real z - \frac14 \im \sigma_0 d \real \sigma_0)
\end{equation}

Combining \eqref{unif1} with \eqref{unif3} and \eqref{unif2}
with \eqref{unif4} we obtain the  set of forms
\begin{equation}\label{unif5}
    e_1 = 2(\im^2 z - \frac14 \im^2 \sigma_0)(\im z dx_0 + i d\theta)+
\ln |\alpha_0| \left( 2\im z d \sigma_0 - \im \sigma_0 d z     \right)
\end{equation}
respectively
\begin{equation}\label{unif6}
 e_2=2(\im^2 z - \frac14 \im^2 \sigma_0) d \ln \alpha_0 + 2i \ln |\alpha_0| (\im z dz - \frac14 \im \sigma_0 d \sigma_0)
\end{equation}
which is equivalent to $(dx_{\pm}, d\theta_{\pm})$,
\[ 
dx_+^2+ dx_-^2 + d\theta_+^2+ d\theta_-^2 \sim e_1^2+ e_2^2. \]

Then the correct metric is
\begin{equation}\label{good-metric}
\begin{split}
 g =   dz^2 + d \sigma_0^2 +   e_1^2 + e_2^2.
 \end{split}
\end{equation}
Thus we can uniformly characterize the two soliton manifold:

\begin{theorem}
The two soliton manifold is smoothly and uniformly parametrized by the parameters  $(z, x_0,\theta_0, \alpha_0,\sigma_{0})$ endowed with the metric \eqref{good-metric} in the range when $|\alpha_0|\ge 1$,
and by the parameters  $(z, x_0,\theta_0, \alpha_0,\mu_{0})$ endowed with the euclidean metric in the range when $|\alpha_0| \lesssim 1$.
\end{theorem}

\subsection{ Double solitons}
These are the limiting solitons where we have a double
eigenvalue.  They are parametrized by the eigenvalue $z$ and the flow parameters $\beta_0$, $\beta_1$, $\beta_2$
and $\beta_3$. To better describe the bump locations, we  translate these into  the alternative set consisting of the spectral parameter $z$, the center
of mass/momentum $x_0,\theta$ and $\gamma_{00}$.

For the center $(x_0,\theta)$ it is easiest to use the formula \eqref{def-center}, which yields
\[
\theta- z x_0 = \beta_0 + \beta_1 z + \beta_2 z^2+ \beta_3 z^3.
\]
Matching imaginary parts we get
\[
x_0 = - \beta_1 - 2 \real z  \beta_2 - (3 \real^2 z-\im^2 z) \beta_3,
\]
and matching real parts,
\[
\theta = \beta_0 - |z|^2 \beta_2 - 2 \real z |z|^2 \beta_3.
\]

On the other hand by \eqref{def-a2} and \eqref{def-a3} we have 
\[
\alpha_0 = 1/\sigma_0 = \gamma_{00} = -2 \im z \beta_2 -(6 \real z \im z + 2i \im^2 z)  \beta_3,
\]
and 
\[ \alpha = -2 \im z \beta_2 - (3 \im  z^2 + 2i \im^2 z)  \beta_3  + i(x-x_0).
\] 
For simplicity we set $x_0 = \theta=0$, which amount to a shift in $x$ plus adjusting the phase. We  plug these values into the expression of Proposition \ref{approximate} to obtain the corresponding approximate 
spectral and scattering parameters
\[ 
z_{\pm}= z \pm  \frac1{2\gamma_{00}},      
\] 
\[
x_{\pm} = x_0 \pm \frac{2\ln 2 +  \ln \im z + \ln |\gamma_{00}|}{2\im z_{\pm} }, 
\] 
and 
\[ 
\theta_{\pm} =  \theta \pm \frac{ \real z_\pm}{2\im z_\pm} \Big( 2\ln 2 +  \ln \im z + \ln |\gamma_{00}|  \Big) 
\pm  \frac12 (\arg \gamma_{00}+ \arg(z_\pm - \bar z_{\mp})).  
\] 
These formulas are valid when $|\gamma_{00}| \gg 1$, which corresponds to $|\beta_2|+|\beta_3| \gg 1$.

It may also be interesting to write down the exact formula for the $2$-soliton, 
namely 
\begin{equation} \label{double-Q} 
Q = -4\im z \frac{e^{2\gamma}   + e^{-2\bar \gamma} 
  +  2i \im z ( \bar \alpha e^{2\gamma}+ \alpha e^{-2\bar \gamma})}{ |\cosh(2\gamma)|^2+ |\sinh(2\gamma)|^2 +  1  +  8  |\im z|^2 |\alpha|^2},
\end{equation}
with 
\[ 
2\alpha = 2\gamma_0 = 2\gamma_{00} + 2i(x-x_0) = 
2 (a_2 \beta_2 + a_3 \beta_3)  + 2i(x-x_0)
\]

  We scale and apply a Galilean transform to normalize to $z=i$. 
 Then $ \gamma_{00}= -2 \beta_2 -2i \beta_3$. After a translation and a phase change
 we have $x_0=\theta=0$
 which leads to 
 $\beta_1 = \beta_3$, $\beta_0 = \beta_2$ 
 and with \eqref{eq:gamma}   $ \gamma  = \red{-} 2x$.     
  Then $\alpha= -2 \beta_2 +i(x-2 \beta_3)  $, and the normalized 
  double soliton has the form

 \begin{equation}\label{double-Q1}
 Q= 4 \frac{(1-4 i\beta_2) \cosh(2x) - 2(x-2\beta_3)\sinh(2x)}{ \cosh^2(2x) + 4(4\beta_2^2 + (x-2\beta_3)^2)}.     
 \end{equation} 
which we plot for selected parameters. First we show real 2-soliton  functions  corresponding to $\beta_2=0$: 


 \begin{figure}[ht]
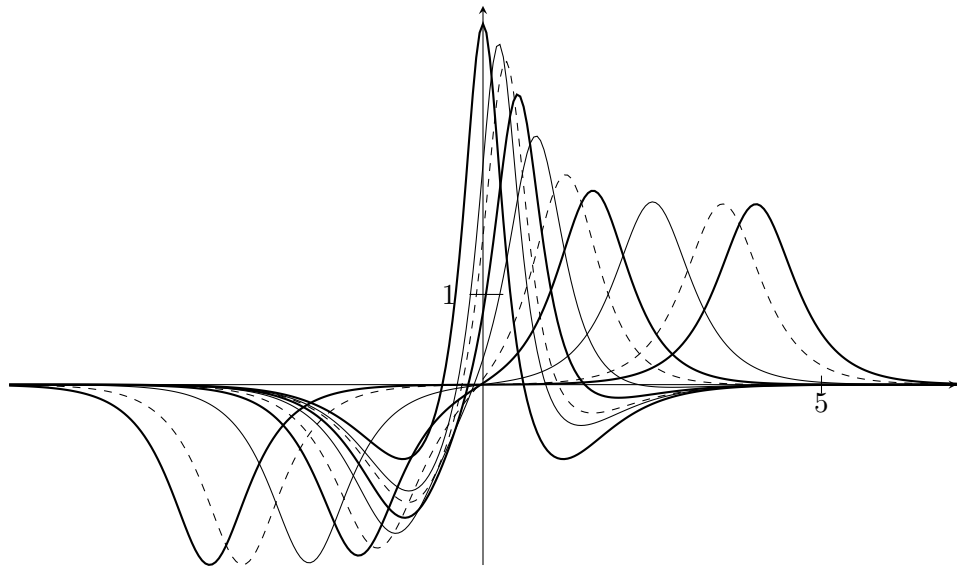


\caption{$\beta_2=0, \beta_3= 0,0.2,0.3,0.5,0.9, 2.4,6,20,150, 400$.} 
   \label{mkdv-figure} 
\end{figure} 
\bigskip

In the general case we plot real and imaginary values for a few values of $\beta_2$ and $\beta_3$, see Figure~\ref{complexsol}.

\begin{figure}[h]
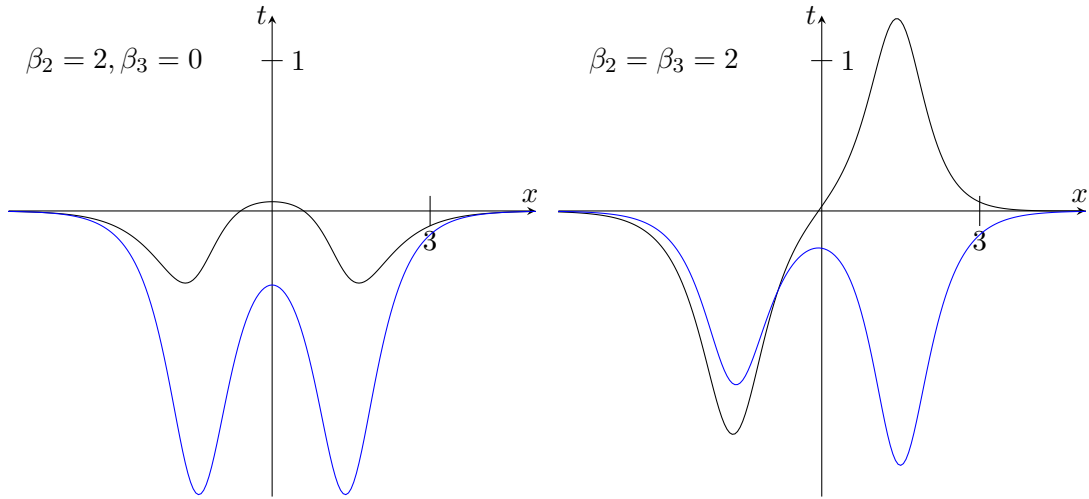

%
\caption{The real part is shown with a black line, the imaginary part by a blue line.} 
\label{complexsol} 
\end{figure}

\subsection{A description of the two soliton dynamics for NLS}

In this section we describe the possible patterns for the interaction
of two solitons with nearby spectral parameters along the NLS flow.

Along the flow the two spectral parameters $z_1$ and $z_2$ stay fixed,
while the scattering parameters $\kappa_1$ and $\kappa_2$ evolve along
the NLS flow according to
\[
\dot \kappa_1 = i z_1^2, \qquad \dot \kappa_2 = iz_2^2
\]
which expressed in terms of the $\beta$'s becomes
\[
\dot \beta_0 = 0, \dot \beta_1 = 0, \dot \beta_2 = 1, \dot \beta_3 = 0,
\]
i.e. $\beta_3$ is the NLS time, which we redenote by $t$, and the others stay fixed.
We set the trivial parameters $\beta_0$ and $\beta_1$ to zero,
and we work out the formulas for the approximate effective 
position of Proposition \ref{approximate} in this case. We begin with the center of mass, 
which moves with velocity
\[
\dot x_0 = - \frac{\im(z_1^2 + z_2^2)}{\im(z_1 + z_2)}.
\]
The remaining interesting parameter is $\gamma_{00}$, which also moves linearly,
with velocity 
\[
\dot \gamma_{00} = a_2,
\]
where we recall that the coefficient $a_2$ is given by
\[
a_2  = i \frac{(z_1+z_2) \im (z_1+z_2)- \im (z_1^2+z_2^2)}{\im (z_1+z_2) }
= - \im(z_1+z_2) -i  \frac{(\real z_1 - \real z_2)(\im z_1-\im z_2)}{\im z_1+\im z_2}.
\]

Assuming that $z_1$ and $z_2$ are close, this has a small real part, and an imaginary part which is away from zero. Then we write
$\gamma_{00}$ in the form 
\[
\gamma_{00}(t) = a +  a_2 t,
\]
where the complex parameter $a$ is our remaining degree of freedom. We can use time translations 
to further normalize $a$, e.g. by choosing it purely real, and then periodicity to insure
that $|a| \lesssim |z_1-z_2|^{-1}$.

With these notations we have
\[
\alpha_0(t)= \frac{\sinh\left((z_1-z_2) (a+ta_2)\right) }{z_1-z_2},
\]
\[
\sigma_0(t) = (z_1-z_2)\coth\left( (z_1-z_2) (a+ta_2)\right),
\]
and the approximate effective bump position is 
\[
x_{\pm}(t) = x_0(t) \pm \frac{ \ln |z_1-\bar z_2| + \ln 2 |\alpha_0(t)|}{\im z_1 +\im z_2 \pm \sigma_0(t) }.
\]

To understand the behavior of the two bumps in time, we need to look at the 
location of the line 
\[
L: \qquad t \to (z_1-z_2) \gamma_{00} = (z_1-z_2) (a+ta_2)
\]
relative to the imaginary axis, and, more importantly, relative to $i \pi \Z $.
Based on this relative position we distinguish two main scenarios, with several 
interesting subcases each. These are described in terms of the difference
$\delta z = z_2 - z_1$ of the two spectral parameters.

\begin{enumerate}[label=(\alph*)]

\item The double soliton case, $\delta z= 0$, which will be viewed both separately and as a limit 
of the scenarios below.

\item Split velocities, where $L$ is fully transversal to the imaginary axis. 
This corresponds to $|\im \delta z| \gtrsim |\real \delta z|$.  Depending on how how close $L$ gets 
to $i \pi \Z$ we have two subcases:

\begin{enumerate}[label=(\roman*)]
    \item Nonresonant, where $d(L,i \pi \Z) \approx 1$, where the two bumps stay as far as 
    possible from each other, i.e. $|\log |\delta z||$ at the closest approach.
    \item Resonant,   where $d(L,i \pi \Z) \ll 1$, and the bumps approach closer than the 
    above threshold. The double soliton case can be seen as a limit of this scenario 
    where $a$ is fixed, and the closest approach is $|\log |a||$.
\end{enumerate}

\item Split scales,  where $L$ is close to parallel to the imaginary axis. 
This corresponds to $|\im \delta z| \ll |\real \delta z|$.  Depending on  how close $L$ gets 
to $i \pi \Z$ we also have two main subcases, and an interesting limiting case:

\begin{enumerate}[label=(\roman*)]
    \item Nonresonant, where $d(L,i \pi \Z) \approx  d_0 = |\im \delta z| / |\real \delta z|$, 
    where the two bumps stay as far as possible from each other, i.e. $|\log |\im \delta z||$ at the closest approach.
    
    \item Resonant,   where $d(L,i \pi \Z)  = d_1 \ll  d_0$, and the bumps approach closer than the 
    above threshold. The double soliton case can also be seen as a limit of this scenario 
    where $a$ is fixed and the closest approach is $|\log |a||$.
    
    \item Quasiperiodic, where $\im \delta z = 0$ and $L$ is parallel to the imaginary axis, at 
    distance $d$. There the soliton distance oscillates between $\log |a|$ and $\log |z_1-z_2|$.
\end{enumerate}

\end{enumerate}
We successively discuss each of these scenarios in turn.

\def\dz{\delta z}

\subsubsection*{(a) The double solitons  $z_1 = z_2$} There we have 
\[
\dot x_0 = - 2 \real z, \qquad \dot \gamma_{00} =   2i \im z. 
\]
Hence after suitable space and time translations we can set 
\[
x_0 = - 2t \real z
\]
\[
\gamma_{00} = 2 it \im z + a, \qquad a \in \R
\]
Hence if $|a| \gg 1$ we have the approximate bump locations 
\[ 
z_{\pm}= z \pm  \frac1{4it\im z+2a}      
\] 
\[
x_{\pm} = x_0 \pm \frac{2\ln 2 +  \ln \im z + \ln |2 it \im z + a |}{2\im z_{\pm} } 
\] 

The separation between the two bumps is $O(\ln |\gamma_{00}|)$, with a $\log |a|$ minimum.
The trajectories of the bumps are like in the following picture:

\begin{figure}[h]
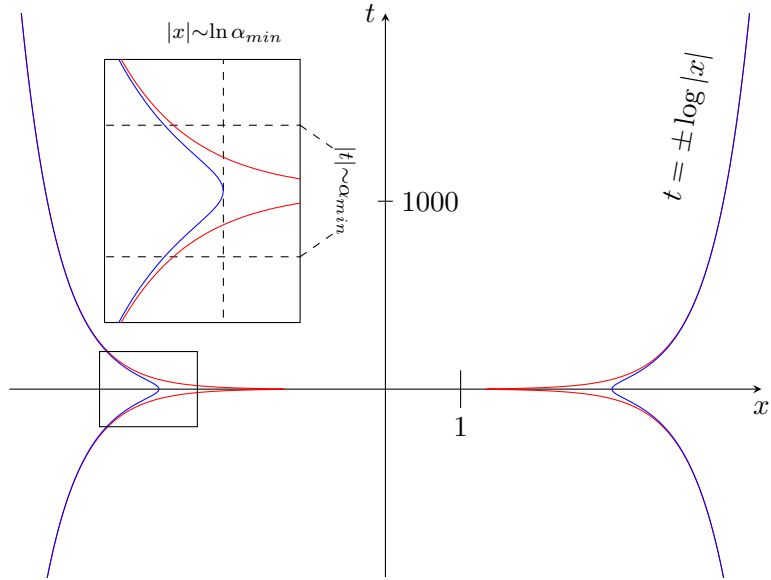
\label{doublez}  

\caption{
The path $x_{\pm}(t)$ of the  solitons  for an eigenvalue $z=i$  of multiplicity $2$. The red curve has $a=2$ and the blue curve $a=100$. Turning is smooth, as seen in the enlarged window.}
\end{figure}


\subsubsection*{ More general $2$-solitons}
Now we consider the case of two different but close spectral parameters.  
We use a galilean transformation and a translation to normalize so that the
center of mass is time independent at $x_0 = 0$, and set
\[
z_1^2 + z_2^2 = -2
\]
so that both $z_1$ and $z_2$ are close to $i$. To describe the dynamics we will
use the small parameter $\dz = z_1-z_2$. This parameter will play a major 
role for in the region where $|\real \alpha_0| \lesssim 1$, which happens for  
a time range
\[
T \approx \frac{1}{\real \dz},
\]
after which the interaction of the two bumps trivializes, in the sense  that the two bumps will evolve linearly but with a spatial shift as predicted by Proposition \ref{approximate}. 

\bigskip

\subsubsection*{(b)(i) Split velocities $|\im \dz| \lesssim |\real \dz|$ and nonresonant 
$ \dist( \dz \gamma_{00}, i \pi \Z) \gtrsim 1 $.}
Then solitons come together with their respective speed, until they reach distance $-\log |\dz|$.
Then they exchange spectral parameters and move away. The effective 
scattering parameters are shifted between the asymptotes at $\pm \infty$ by $\ln |z|$.

\subsubsection*{(b)(ii) Split velocities $|\im \dz| \lesssim |\real \dz|$ and resonant $ \dist(\dz \gamma_{00}, i \pi \Z) = r \ll 1 $.}
Then solitons come together with their respective speed, until they reach distance $-\log |\dz|$ but then they 
continue to approach logarithmically for another $-\log r$ before turnaround.
In the limiting case $z_1 = z_2$ then this reduces to only the logarithmic pattern which comes up to 
minimal distance $- \log |a|$.


\begin{figure}  

\bigskip 

\begin{tikzpicture}[yscale=2]
\draw[->] (0.0000, -2.5)--(0.0000,2.5);
\draw[->] (-5.0000, 0.0000)--(5.0000,0.0000);
\draw (-.1,1.25)--(.1,1.25); \draw(.6,1.25) node{$500$};
\draw (1,.1)--(1,-.1); \draw (1,-.2) node{$1$};
\draw (1.942,-.1)--(1.942,.1); 
\draw (1.942,.7) node[rotate=90]{$x\sim \ln \alpha_{min}$};
\draw(4,1.7) node[rotate=85]{dashed: $t \sim \log x$};
\draw(4.8,2) node[rotate=78]{solid: affine};
\draw (-.3,2.5) node{$t$};
\draw (5,-.2) node{$x$};
\clip (-5,-2.5) rectangle (5,2.5); 
\draw[dashed] (4.5146,-2.6075)--(4.4896,-2.4800)--(4.4645,-2.3587)--(4.4394,-2.2433)--(4.4143,-2.1335)--(4.3893,-2.0291)--(4.3642,-1.9298)--(4.3391,-1.8353)--(4.3140,-1.7455)--(4.2889,-1.6600)--(4.2637,-1.5787)--(4.2386,-1.5013)--(4.2135,-1.4277)--(4.1884,-1.3577)--(4.1632,-1.2912)--(4.1381,-1.2278)--(4.1129,-1.1676)--(4.0878,-1.1103)--(4.0626,-1.0558)--(4.0374,-1.0039)--(4.0122,-0.9546)--(3.9870,-0.9077)--(3.9618,-0.8630)--(3.9366,-0.8206)--(3.9114,-0.7802)--(3.8861,-0.7418)--(3.8609,-0.7053)--(3.8356,-0.6705)--(3.8103,-0.6374)--(3.7850,-0.6060)--(3.7597,-0.5761)--(3.7344,-0.5476)--(3.7090,-0.5205)--(3.6837,-0.4948)--(3.6583,-0.4703)--(3.6329,-0.4470)--(3.6074,-0.4248)--(3.5820,-0.4037)--(3.5565,-0.3837)--(3.5310,-0.3646)--(3.5055,-0.3465)--(3.4800,-0.3292)--(3.4544,-0.3128)--(3.4288,-0.2971)--(3.4031,-0.2823)--(3.3775,-0.2682)--(3.3518,-0.2547)--(3.3260,-0.2419)--(3.3002,-0.2298)--(3.2744,-0.2182)--(3.2485,-0.2072)--(3.2226,-0.1967)--(3.1967,-0.1868)--(3.1707,-0.1773)--(3.1446,-0.1683)--(3.1185,-0.1597)--(3.0923,-0.1515)--(3.0660,-0.1438)--(3.0397,-0.1364)--(3.0134,-0.1294)--(2.9869,-0.1227)--(2.9604,-0.1164)--(2.9338,-0.1103)--(2.9071,-0.1046)--(2.8803,-0.0991)--(2.8534,-0.0939)--(2.8265,-0.0890)--(2.7994,-0.0843)--(2.7722,-0.0798)--(2.7449,-0.0755)--(2.7174,-0.0715)--(2.6898,-0.0676)--(2.6621,-0.0640)--(2.6343,-0.0605)--(2.6062,-0.0572)--(2.5781,-0.0540)--(2.5497,-0.0510)--(2.5212,-0.0482)--(2.4924,-0.0454)--(2.4635,-0.0429)--(2.4343,-0.0404)--(2.4049,-0.0381)--(2.3753,-0.0358)--(2.3454,-0.0337)--(2.3152,-0.0317)--(2.2847,-0.0298)--(2.2539,-0.0280)--(2.2228,-0.0263)--(2.1913,-0.0246)--(2.1595,-0.0230)--(2.1273,-0.0216)--(2.0946,-0.0201)--(2.0616,-0.0188)--(2.0281,-0.0175)--(1.9941,-0.0163)--(1.9596,-0.0151)--(1.9246,-0.0140)--(1.8891,-0.0130)--(1.8530,-0.0120)--(1.8165,-0.0110)--(1.7793,-0.0101)--(1.7417,-0.0093)--(1.7037,-0.0084)--(1.6654,-0.0077)--(1.6268,-0.0069)--(1.5884,-0.0062)--(1.5503,-0.0056)--(1.5132,-0.0049)--(1.4777,-0.0043)--(1.4448,-0.0037)--(1.4159,-0.0032)--(1.3925,-0.0027)--(1.3764,-0.0022)--(1.3687,-0.0017)--(1.3694,-0.0013)--(1.3757,-0.0008)--(1.3829,-0.0004)--(1.3863,-0.0000)--(1.3839,0.0004)--(1.3771,0.0008)--(1.3702,0.0012)--(1.3683,0.0016)--(1.3745,0.0021)--(1.3892,0.0026)--(1.4115,0.0031)--(1.4396,0.0036)--(1.4719,0.0042)--(1.5070,0.0048)--(1.5439,0.0054)--(1.5818,0.0061)--(1.6202,0.0068)--(1.6588,0.0075)--(1.6972,0.0083)--(1.7352,0.0091)--(1.7729,0.0100)--(1.8101,0.0109)--(1.8468,0.0118)--(1.8829,0.0128)--(1.9186,0.0138)--(1.9536,0.0149)--(1.9882,0.0161)--(2.0223,0.0173)--(2.0559,0.0186)--(2.0890,0.0199)--(2.1217,0.0213)--(2.1540,0.0228)--(2.1859,0.0243)--(2.2174,0.0260)--(2.2486,0.0277)--(2.2794,0.0295)--(2.3100,0.0314)--(2.3402,0.0334)--(2.3701,0.0355)--(2.3998,0.0377)--(2.4293,0.0400)--(2.4585,0.0424)--(2.4875,0.0450)--(2.5162,0.0477)--(2.5448,0.0505)--(2.5732,0.0535)--(2.6014,0.0566)--(2.6295,0.0599)--(2.6574,0.0634)--(2.6851,0.0670)--(2.7127,0.0708)--(2.7402,0.0748)--(2.7675,0.0790)--(2.7947,0.0835)--(2.8218,0.0881)--(2.8488,0.0930)--(2.8757,0.0982)--(2.9025,0.1036)--(2.9292,0.1093)--(2.9558,0.1153)--(2.9824,0.1216)--(3.0088,0.1282)--(3.0352,0.1352)--(3.0615,0.1425)--(3.0878,0.1502)--(3.1140,0.1583)--(3.1401,0.1668)--(3.1662,0.1757)--(3.1922,0.1851)--(3.2182,0.1950)--(3.2441,0.2054)--(3.2700,0.2163)--(3.2958,0.2277)--(3.3216,0.2398)--(3.3473,0.2525)--(3.3731,0.2658)--(3.3987,0.2798)--(3.4244,0.2945)--(3.4500,0.3100)--(3.4756,0.3263)--(3.5011,0.3434)--(3.5266,0.3614)--(3.5521,0.3803)--(3.5776,0.4002)--(3.6031,0.4211)--(3.6285,0.4431)--(3.6539,0.4662)--(3.6793,0.4905)--(3.7047,0.5160)--(3.7300,0.5429)--(3.7554,0.5711)--(3.7807,0.6007)--(3.8060,0.6319)--(3.8313,0.6647)--(3.8565,0.6992)--(3.8818,0.7354)--(3.9070,0.7735)--(3.9323,0.8135)--(3.9575,0.8556)--(3.9827,0.8999)--(4.0079,0.9464)--(4.0331,0.9953)--(4.0583,1.0467)--(4.0834,1.1007)--(4.1086,1.1576)--(4.1338,1.2173)--(4.1589,1.2801)--(4.1841,1.3461)--(4.2092,1.4155)--(4.2343,1.4884)--(4.2594,1.5651)--(4.2846,1.6457)--(4.3097,1.7305)--(4.3348,1.8196)--(4.3599,1.9133)--(4.3850,2.0117)--(4.4100,2.1152)--(4.4351,2.2241)--(4.4602,2.3385)--(4.4853,2.4587)--(4.5103,2.5852);

\draw[dashed] (-4.5146,-2.6075)--(-4.4896,-2.4800)--(-4.4645,-2.3587)--(-4.4394,-2.2433)--(-4.4143,-2.1335)--(-4.3893,-2.0291)--(-4.3642,-1.9298)--(-4.3391,-1.8353)--(-4.3140,-1.7455)--(-4.2889,-1.6600)--(-4.2637,-1.5787)--(-4.2386,-1.5013)--(-4.2135,-1.4277)--(-4.1884,-1.3577)--(-4.1632,-1.2912)--(-4.1381,-1.2278)--(-4.1129,-1.1676)--(-4.0878,-1.1103)--(-4.0626,-1.0558)--(-4.0374,-1.0039)--(-4.0122,-0.9546)--(-3.9870,-0.9077)--(-3.9618,-0.8630)--(-3.9366,-0.8206)--(-3.9114,-0.7802)--(-3.8861,-0.7418)--(-3.8609,-0.7053)--(-3.8356,-0.6705)--(-3.8103,-0.6374)--(-3.7850,-0.6060)--(-3.7597,-0.5761)--(-3.7344,-0.5476)--(-3.7090,-0.5205)--(-3.6837,-0.4948)--(-3.6583,-0.4703)--(-3.6329,-0.4470)--(-3.6074,-0.4248)--(-3.5820,-0.4037)--(-3.5565,-0.3837)--(-3.5310,-0.3646)--(-3.5055,-0.3465)--(-3.4800,-0.3292)--(-3.4544,-0.3128)--(-3.4288,-0.2971)--(-3.4031,-0.2823)--(-3.3775,-0.2682)--(-3.3518,-0.2547)--(-3.3260,-0.2419)--(-3.3002,-0.2298)--(-3.2744,-0.2182)--(-3.2485,-0.2072)--(-3.2226,-0.1967)--(-3.1967,-0.1868)--(-3.1707,-0.1773)--(-3.1446,-0.1683)--(-3.1185,-0.1597)--(-3.0923,-0.1515)--(-3.0660,-0.1438)--(-3.0397,-0.1364)--(-3.0134,-0.1294)--(-2.9869,-0.1227)--(-2.9604,-0.1164)--(-2.9338,-0.1103)--(-2.9071,-0.1046)--(-2.8803,-0.0991)--(-2.8534,-0.0939)--(-2.8265,-0.0890)--(-2.7994,-0.0843)--(-2.7722,-0.0798)--(-2.7449,-0.0755)--(-2.7174,-0.0715)--(-2.6898,-0.0676)--(-2.6621,-0.0640)--(-2.6343,-0.0605)--(-2.6062,-0.0572)--(-2.5781,-0.0540)--(-2.5497,-0.0510)--(-2.5212,-0.0482)--(-2.4924,-0.0454)--(-2.4635,-0.0429)--(-2.4343,-0.0404)--(-2.4049,-0.0381)--(-2.3753,-0.0358)--(-2.3454,-0.0337)--(-2.3152,-0.0317)--(-2.2847,-0.0298)--(-2.2539,-0.0280)--(-2.2228,-0.0263)--(-2.1913,-0.0246)--(-2.1595,-0.0230)--(-2.1273,-0.0216)--(-2.0946,-0.0201)--(-2.0616,-0.0188)--(-2.0281,-0.0175)--(-1.9941,-0.0163)--(-1.9596,-0.0151)--(-1.9246,-0.0140)--(-1.8891,-0.0130)--(-1.8530,-0.0120)--(-1.8165,-0.0110)--(-1.7793,-0.0101)--(-1.7417,-0.0093)--(-1.7037,-0.0084)--(-1.6654,-0.0077)--(-1.6268,-0.0069)--(-1.5884,-0.0062)--(-1.5503,-0.0056)--(-1.5132,-0.0049)--(-1.4777,-0.0043)--(-1.4448,-0.0037)--(-1.4159,-0.0032)--(-1.3925,-0.0027)--(-1.3764,-0.0022)--(-1.3687,-0.0017)--(-1.3694,-0.0013)--(-1.3757,-0.0008)--(-1.3829,-0.0004)--(-1.3863,-0.0000)--(-1.3839,0.0004)--(-1.3771,0.0008)--(-1.3702,0.0012)--(-1.3683,0.0016)--(-1.3745,0.0021)--(-1.3892,0.0026)--(-1.4115,0.0031)--(-1.4396,0.0036)--(-1.4719,0.0042)--(-1.5070,0.0048)--(-1.5439,0.0054)--(-1.5818,0.0061)--(-1.6202,0.0068)--(-1.6588,0.0075)--(-1.6972,0.0083)--(-1.7352,0.0091)--(-1.7729,0.0100)--(-1.8101,0.0109)--(-1.8468,0.0118)--(-1.8829,0.0128)--(-1.9186,0.0138)--(-1.9536,0.0149)--(-1.9882,0.0161)--(-2.0223,0.0173)--(-2.0559,0.0186)--(-2.0890,0.0199)--(-2.1217,0.0213)--(-2.1540,0.0228)--(-2.1859,0.0243)--(-2.2174,0.0260)--(-2.2486,0.0277)--(-2.2794,0.0295)--(-2.3100,0.0314)--(-2.3402,0.0334)--(-2.3701,0.0355)--(-2.3998,0.0377)--(-2.4293,0.0400)--(-2.4585,0.0424)--(-2.4875,0.0450)--(-2.5162,0.0477)--(-2.5448,0.0505)--(-2.5732,0.0535)--(-2.6014,0.0566)--(-2.6295,0.0599)--(-2.6574,0.0634)--(-2.6851,0.0670)--(-2.7127,0.0708)--(-2.7402,0.0748)--(-2.7675,0.0790)--(-2.7947,0.0835)--(-2.8218,0.0881)--(-2.8488,0.0930)--(-2.8757,0.0982)--(-2.9025,0.1036)--(-2.9292,0.1093)--(-2.9558,0.1153)--(-2.9824,0.1216)--(-3.0088,0.1282)--(-3.0352,0.1352)--(-3.0615,0.1425)--(-3.0878,0.1502)--(-3.1140,0.1583)--(-3.1401,0.1668)--(-3.1662,0.1757)--(-3.1922,0.1851)--(-3.2182,0.1950)--(-3.2441,0.2054)--(-3.2700,0.2163)--(-3.2958,0.2277)--(-3.3216,0.2398)--(-3.3473,0.2525)--(-3.3731,0.2658)--(-3.3987,0.2798)--(-3.4244,0.2945)--(-3.4500,0.3100)--(-3.4756,0.3263)--(-3.5011,0.3434)--(-3.5266,0.3614)--(-3.5521,0.3803)--(-3.5776,0.4002)--(-3.6031,0.4211)--(-3.6285,0.4431)--(-3.6539,0.4662)--(-3.6793,0.4905)--(-3.7047,0.5160)--(-3.7300,0.5429)--(-3.7554,0.5711)--(-3.7807,0.6007)--(-3.8060,0.6319)--(-3.8313,0.6647)--(-3.8565,0.6992)--(-3.8818,0.7354)--(-3.9070,0.7735)--(-3.9323,0.8135)--(-3.9575,0.8556)--(-3.9827,0.8999)--(-4.0079,0.9464)--(-4.0331,0.9953)--(-4.0583,1.0467)--(-4.0834,1.1007)--(-4.1086,1.1576)--(-4.1338,1.2173)--(-4.1589,1.2801)--(-4.1841,1.3461)--(-4.2092,1.4155)--(-4.2343,1.4884)--(-4.2594,1.5651)--(-4.2846,1.6457)--(-4.3097,1.7305)--(-4.3348,1.8196)--(-4.3599,1.9133)--(-4.3850,2.0117)--(-4.4100,2.1152)--(-4.4351,2.2241)--(-4.4602,2.3385)--(-4.4853,2.4587)--(-4.5103,2.5852);

\draw (4.8217,-2.5737)--(4.7860,-2.4874)--(4.7508,-2.4029)--(4.7162,-2.3203)--(4.6822,-2.2396)--(4.6488,-2.1608)--(4.6159,-2.0839)--(4.5835,-2.0088)--(4.5516,-1.9357)--(4.5203,-1.8645)--(4.4894,-1.7952)--(4.4590,-1.7278)--(4.4291,-1.6623)--(4.3996,-1.5986)--(4.3705,-1.5368)--(4.3418,-1.4769)--(4.3134,-1.4187)--(4.2854,-1.3624)--(4.2578,-1.3078)--(4.2304,-1.2550)--(4.2033,-1.2039)--(4.1765,-1.1545)--(4.1499,-1.1068)--(4.1236,-1.0606)--(4.0974,-1.0161)--(4.0714,-0.9731)--(4.0456,-0.9317)--(4.0199,-0.8917)--(3.9944,-0.8532)--(3.9690,-0.8161)--(3.9437,-0.7804)--(3.9184,-0.7460)--(3.8932,-0.7130)--(3.8681,-0.6812)--(3.8430,-0.6506)--(3.8180,-0.6212)--(3.7929,-0.5930)--(3.7679,-0.5658)--(3.7429,-0.5398)--(3.7178,-0.5148)--(3.6927,-0.4909)--(3.6676,-0.4679)--(3.6424,-0.4458)--(3.6172,-0.4247)--(3.5919,-0.4045)--(3.5666,-0.3851)--(3.5412,-0.3665)--(3.5157,-0.3487)--(3.4901,-0.3317)--(3.4644,-0.3153)--(3.4385,-0.2997)--(3.4126,-0.2848)--(3.3865,-0.2705)--(3.3604,-0.2569)--(3.3340,-0.2438)--(3.3075,-0.2313)--(3.2809,-0.2194)--(3.2541,-0.2080)--(3.2271,-0.1971)--(3.2000,-0.1867)--(3.1726,-0.1767)--(3.1450,-0.1672)--(3.1173,-0.1582)--(3.0892,-0.1495)--(3.0610,-0.1412)--(3.0325,-0.1334)--(3.0037,-0.1258)--(2.9747,-0.1186)--(2.9454,-0.1118)--(2.9157,-0.1052)--(2.8858,-0.0990)--(2.8555,-0.0930)--(2.8248,-0.0873)--(2.7937,-0.0819)--(2.7623,-0.0768)--(2.7305,-0.0718)--(2.6982,-0.0671)--(2.6654,-0.0627)--(2.6322,-0.0584)--(2.5985,-0.0543)--(2.5643,-0.0504)--(2.5296,-0.0467)--(2.4944,-0.0432)--(2.4586,-0.0398)--(2.4224,-0.0366)--(2.3857,-0.0336)--(2.3486,-0.0307)--(2.3112,-0.0279)--(2.2735,-0.0253)--(2.2359,-0.0228)--(2.1985,-0.0204)--(2.1617,-0.0181)--(2.1258,-0.0159)--(2.0914,-0.0138)--(2.0592,-0.0119)--(2.0298,-0.0100)--(2.0039,-0.0082)--(1.9822,-0.0065)--(1.9651,-0.0049)--(1.9527,-0.0033)--(1.9450,-0.0019)--(1.9417,-0.0005)--(1.9424,0.0009)--(1.9471,0.0023)--(1.9563,0.0038)--(1.9702,0.0054)--(1.9889,0.0071)--(2.0121,0.0088)--(2.0392,0.0106)--(2.0696,0.0125)--(2.1026,0.0145)--(2.1375,0.0166)--(2.1738,0.0188)--(2.2109,0.0211)--(2.2484,0.0236)--(2.2860,0.0261)--(2.3236,0.0288)--(2.3609,0.0316)--(2.3979,0.0346)--(2.4344,0.0377)--(2.4705,0.0409)--(2.5061,0.0444)--(2.5411,0.0479)--(2.5757,0.0517)--(2.6097,0.0556)--(2.6433,0.0598)--(2.6763,0.0641)--(2.7089,0.0687)--(2.7411,0.0734)--(2.7728,0.0784)--(2.8041,0.0837)--(2.8350,0.0892)--(2.8655,0.0950)--(2.8957,0.1010)--(2.9256,0.1074)--(2.9551,0.1140)--(2.9844,0.1210)--(3.0133,0.1283)--(3.0420,0.1359)--(3.0704,0.1439)--(3.0986,0.1523)--(3.1265,0.1611)--(3.1542,0.1703)--(3.1817,0.1800)--(3.2090,0.1901)--(3.2361,0.2007)--(3.2630,0.2117)--(3.2898,0.2233)--(3.3163,0.2354)--(3.3428,0.2481)--(3.3691,0.2613)--(3.3952,0.2752)--(3.4212,0.2897)--(3.4471,0.3048)--(3.4729,0.3207)--(3.4986,0.3372)--(3.5241,0.3545)--(3.5496,0.3726)--(3.5750,0.3914)--(3.6003,0.4111)--(3.6256,0.4316)--(3.6508,0.4530)--(3.6759,0.4754)--(3.7010,0.4987)--(3.7261,0.5230)--(3.7512,0.5483)--(3.7762,0.5747)--(3.8012,0.6022)--(3.8263,0.6308)--(3.8513,0.6606)--(3.8764,0.6916)--(3.9016,0.7238)--(3.9268,0.7573)--(3.9520,0.7921)--(3.9774,0.8283)--(4.0029,0.8658)--(4.0284,0.9048)--(4.0542,0.9453)--(4.0800,0.9872)--(4.1061,1.0307)--(4.1323,1.0758)--(4.1587,1.1224)--(4.1854,1.1707)--(4.2123,1.2207)--(4.2394,1.2723)--(4.2669,1.3257)--(4.2947,1.3809)--(4.3228,1.4378)--(4.3513,1.4965)--(4.3801,1.5571)--(4.4093,1.6195)--(4.4390,1.6838)--(4.4691,1.7499)--(4.4996,1.8180)--(4.5306,1.8879)--(4.5621,1.9597)--(4.5942,2.0335)--(4.6267,2.1092)--(4.6598,2.1867)--(4.6934,2.2662)--(4.7276,2.3475)--(4.7624,2.4307)--(4.7978,2.5158);

\draw (-4.8217,-2.5737)--(-4.7860,-2.4874)--(-4.7508,-2.4029)--(-4.7162,-2.3203)--(-4.6822,-2.2396)--(-4.6488,-2.1608)--(-4.6159,-2.0839)--(-4.5835,-2.0088)--(-4.5516,-1.9357)--(-4.5203,-1.8645)--(-4.4894,-1.7952)--(-4.4590,-1.7278)--(-4.4291,-1.6623)--(-4.3996,-1.5986)--(-4.3705,-1.5368)--(-4.3418,-1.4769)--(-4.3134,-1.4187)--(-4.2854,-1.3624)--(-4.2578,-1.3078)--(-4.2304,-1.2550)--(-4.2033,-1.2039)--(-4.1765,-1.1545)--(-4.1499,-1.1068)--(-4.1236,-1.0606)--(-4.0974,-1.0161)--(-4.0714,-0.9731)--(-4.0456,-0.9317)--(-4.0199,-0.8917)--(-3.9944,-0.8532)--(-3.9690,-0.8161)--(-3.9437,-0.7804)--(-3.9184,-0.7460)--(-3.8932,-0.7130)--(-3.8681,-0.6812)--(-3.8430,-0.6506)--(-3.8180,-0.6212)--(-3.7929,-0.5930)--(-3.7679,-0.5658)--(-3.7429,-0.5398)--(-3.7178,-0.5148)--(-3.6927,-0.4909)--(-3.6676,-0.4679)--(-3.6424,-0.4458)--(-3.6172,-0.4247)--(-3.5919,-0.4045)--(-3.5666,-0.3851)--(-3.5412,-0.3665)--(-3.5157,-0.3487)--(-3.4901,-0.3317)--(-3.4644,-0.3153)--(-3.4385,-0.2997)--(-3.4126,-0.2848)--(-3.3865,-0.2705)--(-3.3604,-0.2569)--(-3.3340,-0.2438)--(-3.3075,-0.2313)--(-3.2809,-0.2194)--(-3.2541,-0.2080)--(-3.2271,-0.1971)--(-3.2000,-0.1867)--(-3.1726,-0.1767)--(-3.1450,-0.1672)--(-3.1173,-0.1582)--(-3.0892,-0.1495)--(-3.0610,-0.1412)--(-3.0325,-0.1334)--(-3.0037,-0.1258)--(-2.9747,-0.1186)--(-2.9454,-0.1118)--(-2.9157,-0.1052)--(-2.8858,-0.0990)--(-2.8555,-0.0930)--(-2.8248,-0.0873)--(-2.7937,-0.0819)--(-2.7623,-0.0768)--(-2.7305,-0.0718)--(-2.6982,-0.0671)--(-2.6654,-0.0627)--(-2.6322,-0.0584)--(-2.5985,-0.0543)--(-2.5643,-0.0504)--(-2.5296,-0.0467)--(-2.4944,-0.0432)--(-2.4586,-0.0398)--(-2.4224,-0.0366)--(-2.3857,-0.0336)--(-2.3486,-0.0307)--(-2.3112,-0.0279)--(-2.2735,-0.0253)--(-2.2359,-0.0228)--(-2.1985,-0.0204)--(-2.1617,-0.0181)--(-2.1258,-0.0159)--(-2.0914,-0.0138)--(-2.0592,-0.0119)--(-2.0298,-0.0100)--(-2.0039,-0.0082)--(-1.9822,-0.0065)--(-1.9651,-0.0049)--(-1.9527,-0.0033)--(-1.9450,-0.0019)--(-1.9417,-0.0005)--(-1.9424,0.0009)--(-1.9471,0.0023)--(-1.9563,0.0038)--(-1.9702,0.0054)--(-1.9889,0.0071)--(-2.0121,0.0088)--(-2.0392,0.0106)--(-2.0696,0.0125)--(-2.1026,0.0145)--(-2.1375,0.0166)--(-2.1738,0.0188)--(-2.2109,0.0211)--(-2.2484,0.0236)--(-2.2860,0.0261)--(-2.3236,0.0288)--(-2.3609,0.0316)--(-2.3979,0.0346)--(-2.4344,0.0377)--(-2.4705,0.0409)--(-2.5061,0.0444)--(-2.5411,0.0479)--(-2.5757,0.0517)--(-2.6097,0.0556)--(-2.6433,0.0598)--(-2.6763,0.0641)--(-2.7089,0.0687)--(-2.7411,0.0734)--(-2.7728,0.0784)--(-2.8041,0.0837)--(-2.8350,0.0892)--(-2.8655,0.0950)--(-2.8957,0.1010)--(-2.9256,0.1074)--(-2.9551,0.1140)--(-2.9844,0.1210)--(-3.0133,0.1283)--(-3.0420,0.1359)--(-3.0704,0.1439)--(-3.0986,0.1523)--(-3.1265,0.1611)--(-3.1542,0.1703)--(-3.1817,0.1800)--(-3.2090,0.1901)--(-3.2361,0.2007)--(-3.2630,0.2117)--(-3.2898,0.2233)--(-3.3163,0.2354)--(-3.3428,0.2481)--(-3.3691,0.2613)--(-3.3952,0.2752)--(-3.4212,0.2897)--(-3.4471,0.3048)--(-3.4729,0.3207)--(-3.4986,0.3372)--(-3.5241,0.3545)--(-3.5496,0.3726)--(-3.5750,0.3914)--(-3.6003,0.4111)--(-3.6256,0.4316)--(-3.6508,0.4530)--(-3.6759,0.4754)--(-3.7010,0.4987)--(-3.7261,0.5230)--(-3.7512,0.5483)--(-3.7762,0.5747)--(-3.8012,0.6022)--(-3.8263,0.6308)--(-3.8513,0.6606)--(-3.8764,0.6916)--(-3.9016,0.7238)--(-3.9268,0.7573)--(-3.9520,0.7921)--(-3.9774,0.8283)--(-4.0029,0.8658)--(-4.0284,0.9048)--(-4.0542,0.9453)--(-4.0800,0.9872)--(-4.1061,1.0307)--(-4.1323,1.0758)--(-4.1587,1.1224)--(-4.1854,1.1707)--(-4.2123,1.2207)--(-4.2394,1.2723)--(-4.2669,1.3257)--(-4.2947,1.3809)--(-4.3228,1.4378)--(-4.3513,1.4965)--(-4.3801,1.5571)--(-4.4093,1.6195)--(-4.4390,1.6838)--(-4.4691,1.7499)--(-4.4996,1.8180)--(-4.5306,1.8879)--(-4.5621,1.9597)--(-4.5942,2.0335)--(-4.6267,2.1092)--(-4.6598,2.1867)--(-4.6934,2.2662)--(-4.7276,2.3475)--(-4.7624,2.4307)--(-4.7978,2.5158);
\draw[color=red]  (4.8217,-2.5737)--(2.1207,  3.8063);
\draw[color=red]  (-4.8217,+2.5737)--(-2.1207, -3.8063);

\draw[dashed,color=red,->] (0 ,-2)--(4.5 ,-2);  
\draw[dashed,color=red,->] (0,-2)--(-2.8,-2);
\draw (2,-1.8) node[color=red]{Shift due to interaction};
\end{tikzpicture}
\caption{For reference, the dashed lines are the path for the double eigenvalue $i$. The solid line is the path of the two solitons with 
$z_1=0.0005+i$, $z_2=-0.0005+i$ and $a=10$, which switches back and forth from affine to logarithmic shape.
The red lines show the asymptotic shift.}

\label{figgeneric} 
\end{figure}
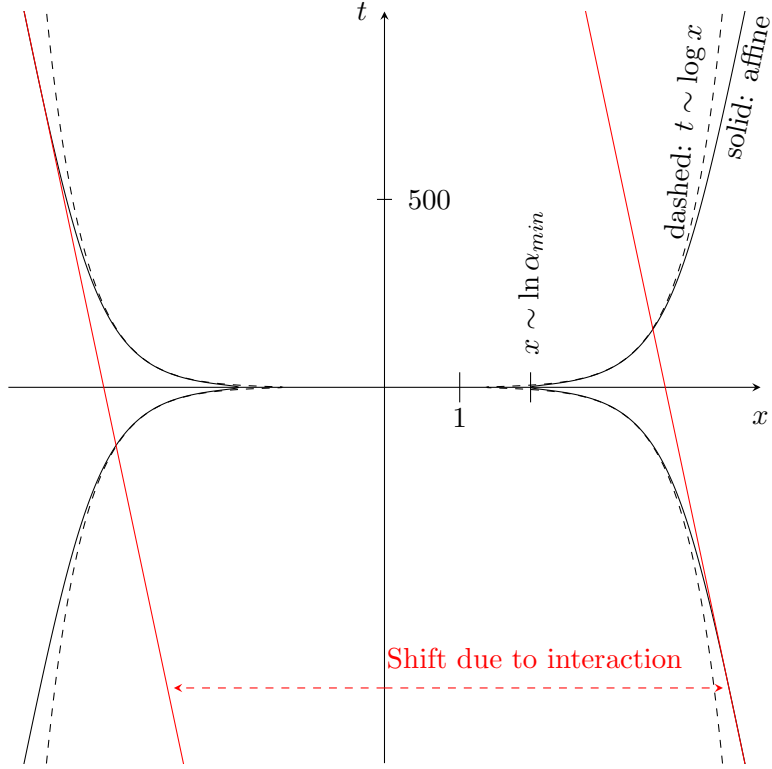



\bigskip

\subsubsection*{(c)(i) Split scales $|\im \dz| \gg |\real \dz|$ and nonresonant 
$ r_0 = \dist(\dz \gamma_{00}, i \pi \Z)  \approx \dfrac{\real \dz}{\im \dz} $.}
Then solitons come together with their respective speed, until they reach distance $-\log |\dz|$ but then they 
continue to approach logarithmically until distance $-\log r_0$ before turnaround; this pattern repeats until  distance grows again above $-\log |\dz|$, for a time $T \approx 1/\Re \dz$.

\subsubsection*{(c)(ii) Split scales $|\im z| \gg |\Re z|$ and resonant 
$ r = \dist(\dz \gamma_{00}, i \pi \Z)  \lesssim  \dfrac{\real \dz}{\im \dz} $.}
This is the same periodic pattern as above but it comes closer in exactly once,
to distance $-\log r$.

\subsubsection*{ (c)(iii) The quasiperiodic solutions $\Im \delta z = 0$} \
In this case we can set
\[
\gamma_{00} = 2i\Im z + a, \qquad a \in \R
\]
which leads to the time frequency
\[
\omega = 2 \real \delta z \Im z
\]
and the time period $2\pi/\omega$.

Assuming that 
\[ 
1 << |a|  \lesssim |z_1-z_2|^{-1},       
\]
the maximal distance from the center of mass axis is approximately $\log |z_1-z_2|$  and the minimal distance  is about $\log |a|$. Else we get a uniform $\log |a|$ distance.


\begin{figure}\label{periodic}  

\bigskip 

\begin{tikzpicture}[yscale=2]
\draw[->] (0.0000, -2.5)--(0.0000,2.5);
\draw[->] (-5.0000, 0.0000)--(5.0000,0.0000);
\draw (-.1,1.25)--(.1,1.25); \draw(.6,1.25) node{$500$};
\draw (1,.1)--(1,-.1); \draw (1,-.2) node{$1$};
\draw (1.942,-.1)--(1.942,.1); 
\draw (1.942,.7) node[rotate=90]{$x\sim \ln \alpha_{min}$};
\draw(4,1.7) node[rotate=85]{dashed: $t \sim \log x$};
\draw(4.8,2) node[rotate=78]{solid: affine};
\draw (-.3,2.5) node{$t$};
\draw (5,-.2) node{$x$};
\clip (-5,-2.5) rectangle (5,2.5); 
\draw[dashed] (4.5146,-2.6075)--(4.4896,-2.4800)--(4.4645,-2.3587)--(4.4394,-2.2433)--(4.4143,-2.1335)--(4.3893,-2.0291)--(4.3642,-1.9298)--(4.3391,-1.8353)--(4.3140,-1.7455)--(4.2889,-1.6600)--(4.2637,-1.5787)--(4.2386,-1.5013)--(4.2135,-1.4277)--(4.1884,-1.3577)--(4.1632,-1.2912)--(4.1381,-1.2278)--(4.1129,-1.1676)--(4.0878,-1.1103)--(4.0626,-1.0558)--(4.0374,-1.0039)--(4.0122,-0.9546)--(3.9870,-0.9077)--(3.9618,-0.8630)--(3.9366,-0.8206)--(3.9114,-0.7802)--(3.8861,-0.7418)--(3.8609,-0.7053)--(3.8356,-0.6705)--(3.8103,-0.6374)--(3.7850,-0.6060)--(3.7597,-0.5761)--(3.7344,-0.5476)--(3.7090,-0.5205)--(3.6837,-0.4948)--(3.6583,-0.4703)--(3.6329,-0.4470)--(3.6074,-0.4248)--(3.5820,-0.4037)--(3.5565,-0.3837)--(3.5310,-0.3646)--(3.5055,-0.3465)--(3.4800,-0.3292)--(3.4544,-0.3128)--(3.4288,-0.2971)--(3.4031,-0.2823)--(3.3775,-0.2682)--(3.3518,-0.2547)--(3.3260,-0.2419)--(3.3002,-0.2298)--(3.2744,-0.2182)--(3.2485,-0.2072)--(3.2226,-0.1967)--(3.1967,-0.1868)--(3.1707,-0.1773)--(3.1446,-0.1683)--(3.1185,-0.1597)--(3.0923,-0.1515)--(3.0660,-0.1438)--(3.0397,-0.1364)--(3.0134,-0.1294)--(2.9869,-0.1227)--(2.9604,-0.1164)--(2.9338,-0.1103)--(2.9071,-0.1046)--(2.8803,-0.0991)--(2.8534,-0.0939)--(2.8265,-0.0890)--(2.7994,-0.0843)--(2.7722,-0.0798)--(2.7449,-0.0755)--(2.7174,-0.0715)--(2.6898,-0.0676)--(2.6621,-0.0640)--(2.6343,-0.0605)--(2.6062,-0.0572)--(2.5781,-0.0540)--(2.5497,-0.0510)--(2.5212,-0.0482)--(2.4924,-0.0454)--(2.4635,-0.0429)--(2.4343,-0.0404)--(2.4049,-0.0381)--(2.3753,-0.0358)--(2.3454,-0.0337)--(2.3152,-0.0317)--(2.2847,-0.0298)--(2.2539,-0.0280)--(2.2228,-0.0263)--(2.1913,-0.0246)--(2.1595,-0.0230)--(2.1273,-0.0216)--(2.0946,-0.0201)--(2.0616,-0.0188)--(2.0281,-0.0175)--(1.9941,-0.0163)--(1.9596,-0.0151)--(1.9246,-0.0140)--(1.8891,-0.0130)--(1.8530,-0.0120)--(1.8165,-0.0110)--(1.7793,-0.0101)--(1.7417,-0.0093)--(1.7037,-0.0084)--(1.6654,-0.0077)--(1.6268,-0.0069)--(1.5884,-0.0062)--(1.5503,-0.0056)--(1.5132,-0.0049)--(1.4777,-0.0043)--(1.4448,-0.0037)--(1.4159,-0.0032)--(1.3925,-0.0027)--(1.3764,-0.0022)--(1.3687,-0.0017)--(1.3694,-0.0013)--(1.3757,-0.0008)--(1.3829,-0.0004)--(1.3863,-0.0000)--(1.3839,0.0004)--(1.3771,0.0008)--(1.3702,0.0012)--(1.3683,0.0016)--(1.3745,0.0021)--(1.3892,0.0026)--(1.4115,0.0031)--(1.4396,0.0036)--(1.4719,0.0042)--(1.5070,0.0048)--(1.5439,0.0054)--(1.5818,0.0061)--(1.6202,0.0068)--(1.6588,0.0075)--(1.6972,0.0083)--(1.7352,0.0091)--(1.7729,0.0100)--(1.8101,0.0109)--(1.8468,0.0118)--(1.8829,0.0128)--(1.9186,0.0138)--(1.9536,0.0149)--(1.9882,0.0161)--(2.0223,0.0173)--(2.0559,0.0186)--(2.0890,0.0199)--(2.1217,0.0213)--(2.1540,0.0228)--(2.1859,0.0243)--(2.2174,0.0260)--(2.2486,0.0277)--(2.2794,0.0295)--(2.3100,0.0314)--(2.3402,0.0334)--(2.3701,0.0355)--(2.3998,0.0377)--(2.4293,0.0400)--(2.4585,0.0424)--(2.4875,0.0450)--(2.5162,0.0477)--(2.5448,0.0505)--(2.5732,0.0535)--(2.6014,0.0566)--(2.6295,0.0599)--(2.6574,0.0634)--(2.6851,0.0670)--(2.7127,0.0708)--(2.7402,0.0748)--(2.7675,0.0790)--(2.7947,0.0835)--(2.8218,0.0881)--(2.8488,0.0930)--(2.8757,0.0982)--(2.9025,0.1036)--(2.9292,0.1093)--(2.9558,0.1153)--(2.9824,0.1216)--(3.0088,0.1282)--(3.0352,0.1352)--(3.0615,0.1425)--(3.0878,0.1502)--(3.1140,0.1583)--(3.1401,0.1668)--(3.1662,0.1757)--(3.1922,0.1851)--(3.2182,0.1950)--(3.2441,0.2054)--(3.2700,0.2163)--(3.2958,0.2277)--(3.3216,0.2398)--(3.3473,0.2525)--(3.3731,0.2658)--(3.3987,0.2798)--(3.4244,0.2945)--(3.4500,0.3100)--(3.4756,0.3263)--(3.5011,0.3434)--(3.5266,0.3614)--(3.5521,0.3803)--(3.5776,0.4002)--(3.6031,0.4211)--(3.6285,0.4431)--(3.6539,0.4662)--(3.6793,0.4905)--(3.7047,0.5160)--(3.7300,0.5429)--(3.7554,0.5711)--(3.7807,0.6007)--(3.8060,0.6319)--(3.8313,0.6647)--(3.8565,0.6992)--(3.8818,0.7354)--(3.9070,0.7735)--(3.9323,0.8135)--(3.9575,0.8556)--(3.9827,0.8999)--(4.0079,0.9464)--(4.0331,0.9953)--(4.0583,1.0467)--(4.0834,1.1007)--(4.1086,1.1576)--(4.1338,1.2173)--(4.1589,1.2801)--(4.1841,1.3461)--(4.2092,1.4155)--(4.2343,1.4884)--(4.2594,1.5651)--(4.2846,1.6457)--(4.3097,1.7305)--(4.3348,1.8196)--(4.3599,1.9133)--(4.3850,2.0117)--(4.4100,2.1152)--(4.4351,2.2241)--(4.4602,2.3385)--(4.4853,2.4587)--(4.5103,2.5852);

\draw[dashed] (-4.5146,-2.6075)--(-4.4896,-2.4800)--(-4.4645,-2.3587)--(-4.4394,-2.2433)--(-4.4143,-2.1335)--(-4.3893,-2.0291)--(-4.3642,-1.9298)--(-4.3391,-1.8353)--(-4.3140,-1.7455)--(-4.2889,-1.6600)--(-4.2637,-1.5787)--(-4.2386,-1.5013)--(-4.2135,-1.4277)--(-4.1884,-1.3577)--(-4.1632,-1.2912)--(-4.1381,-1.2278)--(-4.1129,-1.1676)--(-4.0878,-1.1103)--(-4.0626,-1.0558)--(-4.0374,-1.0039)--(-4.0122,-0.9546)--(-3.9870,-0.9077)--(-3.9618,-0.8630)--(-3.9366,-0.8206)--(-3.9114,-0.7802)--(-3.8861,-0.7418)--(-3.8609,-0.7053)--(-3.8356,-0.6705)--(-3.8103,-0.6374)--(-3.7850,-0.6060)--(-3.7597,-0.5761)--(-3.7344,-0.5476)--(-3.7090,-0.5205)--(-3.6837,-0.4948)--(-3.6583,-0.4703)--(-3.6329,-0.4470)--(-3.6074,-0.4248)--(-3.5820,-0.4037)--(-3.5565,-0.3837)--(-3.5310,-0.3646)--(-3.5055,-0.3465)--(-3.4800,-0.3292)--(-3.4544,-0.3128)--(-3.4288,-0.2971)--(-3.4031,-0.2823)--(-3.3775,-0.2682)--(-3.3518,-0.2547)--(-3.3260,-0.2419)--(-3.3002,-0.2298)--(-3.2744,-0.2182)--(-3.2485,-0.2072)--(-3.2226,-0.1967)--(-3.1967,-0.1868)--(-3.1707,-0.1773)--(-3.1446,-0.1683)--(-3.1185,-0.1597)--(-3.0923,-0.1515)--(-3.0660,-0.1438)--(-3.0397,-0.1364)--(-3.0134,-0.1294)--(-2.9869,-0.1227)--(-2.9604,-0.1164)--(-2.9338,-0.1103)--(-2.9071,-0.1046)--(-2.8803,-0.0991)--(-2.8534,-0.0939)--(-2.8265,-0.0890)--(-2.7994,-0.0843)--(-2.7722,-0.0798)--(-2.7449,-0.0755)--(-2.7174,-0.0715)--(-2.6898,-0.0676)--(-2.6621,-0.0640)--(-2.6343,-0.0605)--(-2.6062,-0.0572)--(-2.5781,-0.0540)--(-2.5497,-0.0510)--(-2.5212,-0.0482)--(-2.4924,-0.0454)--(-2.4635,-0.0429)--(-2.4343,-0.0404)--(-2.4049,-0.0381)--(-2.3753,-0.0358)--(-2.3454,-0.0337)--(-2.3152,-0.0317)--(-2.2847,-0.0298)--(-2.2539,-0.0280)--(-2.2228,-0.0263)--(-2.1913,-0.0246)--(-2.1595,-0.0230)--(-2.1273,-0.0216)--(-2.0946,-0.0201)--(-2.0616,-0.0188)--(-2.0281,-0.0175)--(-1.9941,-0.0163)--(-1.9596,-0.0151)--(-1.9246,-0.0140)--(-1.8891,-0.0130)--(-1.8530,-0.0120)--(-1.8165,-0.0110)--(-1.7793,-0.0101)--(-1.7417,-0.0093)--(-1.7037,-0.0084)--(-1.6654,-0.0077)--(-1.6268,-0.0069)--(-1.5884,-0.0062)--(-1.5503,-0.0056)--(-1.5132,-0.0049)--(-1.4777,-0.0043)--(-1.4448,-0.0037)--(-1.4159,-0.0032)--(-1.3925,-0.0027)--(-1.3764,-0.0022)--(-1.3687,-0.0017)--(-1.3694,-0.0013)--(-1.3757,-0.0008)--(-1.3829,-0.0004)--(-1.3863,-0.0000)--(-1.3839,0.0004)--(-1.3771,0.0008)--(-1.3702,0.0012)--(-1.3683,0.0016)--(-1.3745,0.0021)--(-1.3892,0.0026)--(-1.4115,0.0031)--(-1.4396,0.0036)--(-1.4719,0.0042)--(-1.5070,0.0048)--(-1.5439,0.0054)--(-1.5818,0.0061)--(-1.6202,0.0068)--(-1.6588,0.0075)--(-1.6972,0.0083)--(-1.7352,0.0091)--(-1.7729,0.0100)--(-1.8101,0.0109)--(-1.8468,0.0118)--(-1.8829,0.0128)--(-1.9186,0.0138)--(-1.9536,0.0149)--(-1.9882,0.0161)--(-2.0223,0.0173)--(-2.0559,0.0186)--(-2.0890,0.0199)--(-2.1217,0.0213)--(-2.1540,0.0228)--(-2.1859,0.0243)--(-2.2174,0.0260)--(-2.2486,0.0277)--(-2.2794,0.0295)--(-2.3100,0.0314)--(-2.3402,0.0334)--(-2.3701,0.0355)--(-2.3998,0.0377)--(-2.4293,0.0400)--(-2.4585,0.0424)--(-2.4875,0.0450)--(-2.5162,0.0477)--(-2.5448,0.0505)--(-2.5732,0.0535)--(-2.6014,0.0566)--(-2.6295,0.0599)--(-2.6574,0.0634)--(-2.6851,0.0670)--(-2.7127,0.0708)--(-2.7402,0.0748)--(-2.7675,0.0790)--(-2.7947,0.0835)--(-2.8218,0.0881)--(-2.8488,0.0930)--(-2.8757,0.0982)--(-2.9025,0.1036)--(-2.9292,0.1093)--(-2.9558,0.1153)--(-2.9824,0.1216)--(-3.0088,0.1282)--(-3.0352,0.1352)--(-3.0615,0.1425)--(-3.0878,0.1502)--(-3.1140,0.1583)--(-3.1401,0.1668)--(-3.1662,0.1757)--(-3.1922,0.1851)--(-3.2182,0.1950)--(-3.2441,0.2054)--(-3.2700,0.2163)--(-3.2958,0.2277)--(-3.3216,0.2398)--(-3.3473,0.2525)--(-3.3731,0.2658)--(-3.3987,0.2798)--(-3.4244,0.2945)--(-3.4500,0.3100)--(-3.4756,0.3263)--(-3.5011,0.3434)--(-3.5266,0.3614)--(-3.5521,0.3803)--(-3.5776,0.4002)--(-3.6031,0.4211)--(-3.6285,0.4431)--(-3.6539,0.4662)--(-3.6793,0.4905)--(-3.7047,0.5160)--(-3.7300,0.5429)--(-3.7554,0.5711)--(-3.7807,0.6007)--(-3.8060,0.6319)--(-3.8313,0.6647)--(-3.8565,0.6992)--(-3.8818,0.7354)--(-3.9070,0.7735)--(-3.9323,0.8135)--(-3.9575,0.8556)--(-3.9827,0.8999)--(-4.0079,0.9464)--(-4.0331,0.9953)--(-4.0583,1.0467)--(-4.0834,1.1007)--(-4.1086,1.1576)--(-4.1338,1.2173)--(-4.1589,1.2801)--(-4.1841,1.3461)--(-4.2092,1.4155)--(-4.2343,1.4884)--(-4.2594,1.5651)--(-4.2846,1.6457)--(-4.3097,1.7305)--(-4.3348,1.8196)--(-4.3599,1.9133)--(-4.3850,2.0117)--(-4.4100,2.1152)--(-4.4351,2.2241)--(-4.4602,2.3385)--(-4.4853,2.4587)--(-4.5103,2.5852);

\draw (4.8217,-2.5737)--(4.7860,-2.4874)--(4.7508,-2.4029)--(4.7162,-2.3203)--(4.6822,-2.2396)--(4.6488,-2.1608)--(4.6159,-2.0839)--(4.5835,-2.0088)--(4.5516,-1.9357)--(4.5203,-1.8645)--(4.4894,-1.7952)--(4.4590,-1.7278)--(4.4291,-1.6623)--(4.3996,-1.5986)--(4.3705,-1.5368)--(4.3418,-1.4769)--(4.3134,-1.4187)--(4.2854,-1.3624)--(4.2578,-1.3078)--(4.2304,-1.2550)--(4.2033,-1.2039)--(4.1765,-1.1545)--(4.1499,-1.1068)--(4.1236,-1.0606)--(4.0974,-1.0161)--(4.0714,-0.9731)--(4.0456,-0.9317)--(4.0199,-0.8917)--(3.9944,-0.8532)--(3.9690,-0.8161)--(3.9437,-0.7804)--(3.9184,-0.7460)--(3.8932,-0.7130)--(3.8681,-0.6812)--(3.8430,-0.6506)--(3.8180,-0.6212)--(3.7929,-0.5930)--(3.7679,-0.5658)--(3.7429,-0.5398)--(3.7178,-0.5148)--(3.6927,-0.4909)--(3.6676,-0.4679)--(3.6424,-0.4458)--(3.6172,-0.4247)--(3.5919,-0.4045)--(3.5666,-0.3851)--(3.5412,-0.3665)--(3.5157,-0.3487)--(3.4901,-0.3317)--(3.4644,-0.3153)--(3.4385,-0.2997)--(3.4126,-0.2848)--(3.3865,-0.2705)--(3.3604,-0.2569)--(3.3340,-0.2438)--(3.3075,-0.2313)--(3.2809,-0.2194)--(3.2541,-0.2080)--(3.2271,-0.1971)--(3.2000,-0.1867)--(3.1726,-0.1767)--(3.1450,-0.1672)--(3.1173,-0.1582)--(3.0892,-0.1495)--(3.0610,-0.1412)--(3.0325,-0.1334)--(3.0037,-0.1258)--(2.9747,-0.1186)--(2.9454,-0.1118)--(2.9157,-0.1052)--(2.8858,-0.0990)--(2.8555,-0.0930)--(2.8248,-0.0873)--(2.7937,-0.0819)--(2.7623,-0.0768)--(2.7305,-0.0718)--(2.6982,-0.0671)--(2.6654,-0.0627)--(2.6322,-0.0584)--(2.5985,-0.0543)--(2.5643,-0.0504)--(2.5296,-0.0467)--(2.4944,-0.0432)--(2.4586,-0.0398)--(2.4224,-0.0366)--(2.3857,-0.0336)--(2.3486,-0.0307)--(2.3112,-0.0279)--(2.2735,-0.0253)--(2.2359,-0.0228)--(2.1985,-0.0204)--(2.1617,-0.0181)--(2.1258,-0.0159)--(2.0914,-0.0138)--(2.0592,-0.0119)--(2.0298,-0.0100)--(2.0039,-0.0082)--(1.9822,-0.0065)--(1.9651,-0.0049)--(1.9527,-0.0033)--(1.9450,-0.0019)--(1.9417,-0.0005)--(1.9424,0.0009)--(1.9471,0.0023)--(1.9563,0.0038)--(1.9702,0.0054)--(1.9889,0.0071)--(2.0121,0.0088)--(2.0392,0.0106)--(2.0696,0.0125)--(2.1026,0.0145)--(2.1375,0.0166)--(2.1738,0.0188)--(2.2109,0.0211)--(2.2484,0.0236)--(2.2860,0.0261)--(2.3236,0.0288)--(2.3609,0.0316)--(2.3979,0.0346)--(2.4344,0.0377)--(2.4705,0.0409)--(2.5061,0.0444)--(2.5411,0.0479)--(2.5757,0.0517)--(2.6097,0.0556)--(2.6433,0.0598)--(2.6763,0.0641)--(2.7089,0.0687)--(2.7411,0.0734)--(2.7728,0.0784)--(2.8041,0.0837)--(2.8350,0.0892)--(2.8655,0.0950)--(2.8957,0.1010)--(2.9256,0.1074)--(2.9551,0.1140)--(2.9844,0.1210)--(3.0133,0.1283)--(3.0420,0.1359)--(3.0704,0.1439)--(3.0986,0.1523)--(3.1265,0.1611)--(3.1542,0.1703)--(3.1817,0.1800)--(3.2090,0.1901)--(3.2361,0.2007)--(3.2630,0.2117)--(3.2898,0.2233)--(3.3163,0.2354)--(3.3428,0.2481)--(3.3691,0.2613)--(3.3952,0.2752)--(3.4212,0.2897)--(3.4471,0.3048)--(3.4729,0.3207)--(3.4986,0.3372)--(3.5241,0.3545)--(3.5496,0.3726)--(3.5750,0.3914)--(3.6003,0.4111)--(3.6256,0.4316)--(3.6508,0.4530)--(3.6759,0.4754)--(3.7010,0.4987)--(3.7261,0.5230)--(3.7512,0.5483)--(3.7762,0.5747)--(3.8012,0.6022)--(3.8263,0.6308)--(3.8513,0.6606)--(3.8764,0.6916)--(3.9016,0.7238)--(3.9268,0.7573)--(3.9520,0.7921)--(3.9774,0.8283)--(4.0029,0.8658)--(4.0284,0.9048)--(4.0542,0.9453)--(4.0800,0.9872)--(4.1061,1.0307)--(4.1323,1.0758)--(4.1587,1.1224)--(4.1854,1.1707)--(4.2123,1.2207)--(4.2394,1.2723)--(4.2669,1.3257)--(4.2947,1.3809)--(4.3228,1.4378)--(4.3513,1.4965)--(4.3801,1.5571)--(4.4093,1.6195)--(4.4390,1.6838)--(4.4691,1.7499)--(4.4996,1.8180)--(4.5306,1.8879)--(4.5621,1.9597)--(4.5942,2.0335)--(4.6267,2.1092)--(4.6598,2.1867)--(4.6934,2.2662)--(4.7276,2.3475)--(4.7624,2.4307)--(4.7978,2.5158);

\draw (-4.8217,-2.5737)--(-4.7860,-2.4874)--(-4.7508,-2.4029)--(-4.7162,-2.3203)--(-4.6822,-2.2396)--(-4.6488,-2.1608)--(-4.6159,-2.0839)--(-4.5835,-2.0088)--(-4.5516,-1.9357)--(-4.5203,-1.8645)--(-4.4894,-1.7952)--(-4.4590,-1.7278)--(-4.4291,-1.6623)--(-4.3996,-1.5986)--(-4.3705,-1.5368)--(-4.3418,-1.4769)--(-4.3134,-1.4187)--(-4.2854,-1.3624)--(-4.2578,-1.3078)--(-4.2304,-1.2550)--(-4.2033,-1.2039)--(-4.1765,-1.1545)--(-4.1499,-1.1068)--(-4.1236,-1.0606)--(-4.0974,-1.0161)--(-4.0714,-0.9731)--(-4.0456,-0.9317)--(-4.0199,-0.8917)--(-3.9944,-0.8532)--(-3.9690,-0.8161)--(-3.9437,-0.7804)--(-3.9184,-0.7460)--(-3.8932,-0.7130)--(-3.8681,-0.6812)--(-3.8430,-0.6506)--(-3.8180,-0.6212)--(-3.7929,-0.5930)--(-3.7679,-0.5658)--(-3.7429,-0.5398)--(-3.7178,-0.5148)--(-3.6927,-0.4909)--(-3.6676,-0.4679)--(-3.6424,-0.4458)--(-3.6172,-0.4247)--(-3.5919,-0.4045)--(-3.5666,-0.3851)--(-3.5412,-0.3665)--(-3.5157,-0.3487)--(-3.4901,-0.3317)--(-3.4644,-0.3153)--(-3.4385,-0.2997)--(-3.4126,-0.2848)--(-3.3865,-0.2705)--(-3.3604,-0.2569)--(-3.3340,-0.2438)--(-3.3075,-0.2313)--(-3.2809,-0.2194)--(-3.2541,-0.2080)--(-3.2271,-0.1971)--(-3.2000,-0.1867)--(-3.1726,-0.1767)--(-3.1450,-0.1672)--(-3.1173,-0.1582)--(-3.0892,-0.1495)--(-3.0610,-0.1412)--(-3.0325,-0.1334)--(-3.0037,-0.1258)--(-2.9747,-0.1186)--(-2.9454,-0.1118)--(-2.9157,-0.1052)--(-2.8858,-0.0990)--(-2.8555,-0.0930)--(-2.8248,-0.0873)--(-2.7937,-0.0819)--(-2.7623,-0.0768)--(-2.7305,-0.0718)--(-2.6982,-0.0671)--(-2.6654,-0.0627)--(-2.6322,-0.0584)--(-2.5985,-0.0543)--(-2.5643,-0.0504)--(-2.5296,-0.0467)--(-2.4944,-0.0432)--(-2.4586,-0.0398)--(-2.4224,-0.0366)--(-2.3857,-0.0336)--(-2.3486,-0.0307)--(-2.3112,-0.0279)--(-2.2735,-0.0253)--(-2.2359,-0.0228)--(-2.1985,-0.0204)--(-2.1617,-0.0181)--(-2.1258,-0.0159)--(-2.0914,-0.0138)--(-2.0592,-0.0119)--(-2.0298,-0.0100)--(-2.0039,-0.0082)--(-1.9822,-0.0065)--(-1.9651,-0.0049)--(-1.9527,-0.0033)--(-1.9450,-0.0019)--(-1.9417,-0.0005)--(-1.9424,0.0009)--(-1.9471,0.0023)--(-1.9563,0.0038)--(-1.9702,0.0054)--(-1.9889,0.0071)--(-2.0121,0.0088)--(-2.0392,0.0106)--(-2.0696,0.0125)--(-2.1026,0.0145)--(-2.1375,0.0166)--(-2.1738,0.0188)--(-2.2109,0.0211)--(-2.2484,0.0236)--(-2.2860,0.0261)--(-2.3236,0.0288)--(-2.3609,0.0316)--(-2.3979,0.0346)--(-2.4344,0.0377)--(-2.4705,0.0409)--(-2.5061,0.0444)--(-2.5411,0.0479)--(-2.5757,0.0517)--(-2.6097,0.0556)--(-2.6433,0.0598)--(-2.6763,0.0641)--(-2.7089,0.0687)--(-2.7411,0.0734)--(-2.7728,0.0784)--(-2.8041,0.0837)--(-2.8350,0.0892)--(-2.8655,0.0950)--(-2.8957,0.1010)--(-2.9256,0.1074)--(-2.9551,0.1140)--(-2.9844,0.1210)--(-3.0133,0.1283)--(-3.0420,0.1359)--(-3.0704,0.1439)--(-3.0986,0.1523)--(-3.1265,0.1611)--(-3.1542,0.1703)--(-3.1817,0.1800)--(-3.2090,0.1901)--(-3.2361,0.2007)--(-3.2630,0.2117)--(-3.2898,0.2233)--(-3.3163,0.2354)--(-3.3428,0.2481)--(-3.3691,0.2613)--(-3.3952,0.2752)--(-3.4212,0.2897)--(-3.4471,0.3048)--(-3.4729,0.3207)--(-3.4986,0.3372)--(-3.5241,0.3545)--(-3.5496,0.3726)--(-3.5750,0.3914)--(-3.6003,0.4111)--(-3.6256,0.4316)--(-3.6508,0.4530)--(-3.6759,0.4754)--(-3.7010,0.4987)--(-3.7261,0.5230)--(-3.7512,0.5483)--(-3.7762,0.5747)--(-3.8012,0.6022)--(-3.8263,0.6308)--(-3.8513,0.6606)--(-3.8764,0.6916)--(-3.9016,0.7238)--(-3.9268,0.7573)--(-3.9520,0.7921)--(-3.9774,0.8283)--(-4.0029,0.8658)--(-4.0284,0.9048)--(-4.0542,0.9453)--(-4.0800,0.9872)--(-4.1061,1.0307)--(-4.1323,1.0758)--(-4.1587,1.1224)--(-4.1854,1.1707)--(-4.2123,1.2207)--(-4.2394,1.2723)--(-4.2669,1.3257)--(-4.2947,1.3809)--(-4.3228,1.4378)--(-4.3513,1.4965)--(-4.3801,1.5571)--(-4.4093,1.6195)--(-4.4390,1.6838)--(-4.4691,1.7499)--(-4.4996,1.8180)--(-4.5306,1.8879)--(-4.5621,1.9597)--(-4.5942,2.0335)--(-4.6267,2.1092)--(-4.6598,2.1867)--(-4.6934,2.2662)--(-4.7276,2.3475)--(-4.7624,2.4307)--(-4.7978,2.5158);
\draw[color=red]  (4.8217,-2.5737)--(2.1207,  3.8063);
\draw[color=red]  (-4.8217,+2.5737)--(-2.1207, -3.8063);

\draw[dashed,color=red,->] (0 ,-2)--(4.5 ,-2);  
\draw[dashed,color=red,->] (0,-2)--(-2.8,-2);
\draw (2,-1.8) node[color=red]{Shift due to interaction};
\end{tikzpicture}
\caption{
Here $a=1$. The dashed curve corresponds to the double eigenvalue $z=i$ for reference. The other curves correspond to 
$z_{0,1}  = (1\pm 0.01)i \mp h $ 
where $h = 0, 0.002,0.004 $ and $0.016$.}
\end{figure}
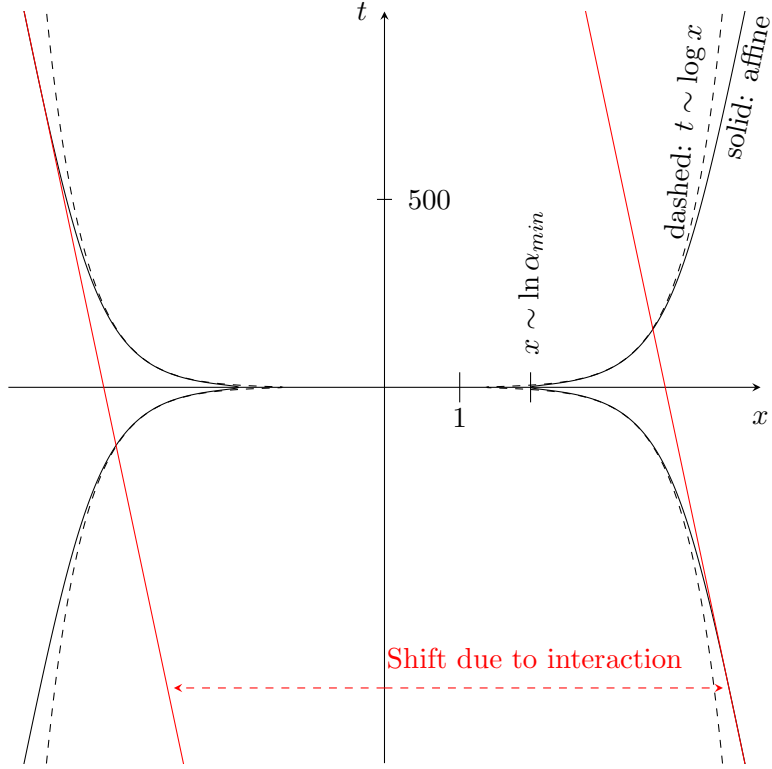

\bigskip

\subsection{A description of the 2-soliton  dynamics for mKdV} 
Here we are interested in real solutions to 
\[ u_t + u_{xxx} + 6 u^2 u_x = 0. \] 
In this case the spectrum is symmetric under the reflection on the imaginary axis $ z \to - \bar z $.
Hence $1$-soliton solutions correspond to pure imaginary eigenvalues $i w$ with $ w>0$. Then $\theta \in \pi/2 \Z$, the unbounded $iw$ wave is 
\[ 
\left( \begin{matrix} e^{- w(x-   x_0)-w^3t} \\ e^{w(x-x_0)- w^3 t}\end{matrix} \right)     
\] 
and the soliton  has speed $w^2$ and the explicit form
\[
w \cosh( w ( x-x_0 - w^2 t) ). 
\] 
For the two soliton case we first specialize the previous formulas.  We are interested in real solutions with close eigenvalues, which we write as 
\[
z_{1,2} = i (w \pm \rho) 
\] 
with $ \rho^2 \in \R$. Then we have
\[
\dot x_0 =  w^2+ 3 \rho^2,  \qquad 
\dot \gamma_{00} = 2i(\rho^2-w^2)  
\] 
\[ \gamma = - w[(x-x_0) - (w^2+ 3 \rho^2) t]  \] 
\[ 
\gamma_1-\gamma_2 = i \tilde \gamma = 2\rho(x-x_0 -t (3 w^2+\rho^2))+ i \mu. \] 
for some fixed constant $\mu$.
With these parameters, the two soliton $Q$ is given by
\[ 
Q= - 2 \frac{e^{2\gamma}\Big[ 2w \cosh(\overline{\gamma_1-\gamma_2}) 
- \dfrac{4w^2}{\bar \rho} \sinh(\overline{\gamma_1-\gamma_2}) \Big] 
 + e^{-2\gamma}
\Big[  2w \cosh(\gamma_1-\gamma_2) +
\dfrac{4w^2}\rho \sinh(\gamma_1-\gamma_2) \Big]}
{4 \cosh^2(2\gamma)  + 2 |\cosh(\gamma_1-\gamma_2) |^2-2 + 8(2(w^2/\rho^2)-1)   |\sinh(\gamma_1-\gamma_2)|^2.}  
\]  

Since we are interested in real solitons, the  parameter $\mu$ cannot 
be chosen arbitrarily. Precisely, there are three possibilities for real solutions:

\begin{enumerate}[label=\alph*)]
    \item  \label{separateeigen} Both eigenvalues are distinct and purely imaginary, $\rho \in (0.w)$. Here we have four connected components,
which up to symmetries are grouped in  two subcases: 
    
    (i) $\mu \in  2\pi  \Z$, which gives the formula 
\[ 
    Q=  - \frac{  4w \cosh(2\gamma) \cosh(\gamma_1-\gamma_2) 
- 4w^2 \sinh(2\gamma) \rho^{-1} \sinh(\gamma_1-\gamma_2) }
{ 2 \cosh^2(2\gamma)  + (8(w^2/ \rho)-3)   \sinh^2(\gamma_1-\gamma_2)} . 
\]     
which corresponds to two bumps of opposite signs, and $\mu \in  \pi + 2\pi  \Z$
which yields the soliton $-Q$ with two negative bumps.

(ii) If $ \mu \in   2\pi \Z + \pi/2 $  then we divide by $i$ and get    
\[ 
Q= - \frac{  4w \sinh(2\gamma) \sinh(\gamma_1-\gamma_2) 
- 4w^2 \cosh(2\gamma) \rho^{-1} \cosh(\gamma_1-\gamma_2) }
{ 2 \sinh^2(2\gamma)  + ( 8(w^2\rho^{-2})-3)  \cosh^2(\gamma_1-\gamma_2)} . 
\] 
which corresponds to two positive bumps, and $\mu \in  \frac32 \pi + 2\pi  \Z$ which yields $-Q$.

 
    \item \label{breather}  Distinct complex conjugate eigenvalues $ z_{1,2} = iw \pm \tilde \rho$ for some $ \tilde \rho >0$. This is the breather solution, which is time periodic in a moving frame. Then $\tilde \gamma$ is real valued, so after a time translation we can set $\mu = 0$, and 
 \[ Q=  - \frac{  4w \cosh(2\gamma) \cos(\tilde \gamma) 
- 4w^2 \sinh(2\gamma) \tilde \rho^{-1}\sin(\tilde \gamma) }
{ 2 \cosh^2(2\gamma)  + (8 w^2 /\tilde \rho^{2}+3) \sin^2(\tilde \gamma)} . 
\]  
    \item \label{doubleeigen} The double eigenvalue $2$-soliton, with eigenvalues $z_{1,2} = i w$, which is the transition between the regimes  \ref{separateeigen} i) and \ref{breather}   above. The formula for the  2 soliton is obtained as limit of the formulas above. 
    \[  Q= \frac{4  w \cosh(2 w (x-x_0- tw^2)) - 4w^2 (x-x_0- 3tw^2)\sinh(2(x-x_0-tw^2))}{2 \cosh^2(2w(x-x_0-tw^2)) + 8w^2(x-x_0-3tw^2)^2}.         \]
\end{enumerate}

Precisely, we can see the solitons in (a)(i), (b) and (c) above
as a single,  analytic function of $\rho^2$. If $\rho^2 <0$ we have breathers, and if $\rho^2 > 0$ we have  two soliton states with eigenvalues $i(w \pm |\rho|) $. By contrast, the solitons in (a)(ii) do not connect to the double eigenvalue and breather case. We briefly discuss the three cases
further below.

\subsubsection*{Case \ref{doubleeigen}  The double eigenvalue $z_{1,2} = iw$}

Here the center of mass evolves according to 
\[
\dot x_0 = w^2
\]
whereas $\alpha_0$ is purely imaginary, and is given by
\[
\dot \alpha_0  = -2i w^2
\]
After a time translation we can set 
\[
\alpha_0 = -2i t w^2,
\]
and then the two bumps are nearly symmetric around the center of mass,
at distance
\[
x_\pm - x_0 \approx \pm \frac{\ln \langle |t| w^2 \rangle}{2w}, \qquad 
|t| \gg 1.
\]
The  two soliton looks like a sum of two simple solitons unless the two are close. Several time sections of the graph are sketched in Figure \ref{mkdv-figure}.

\subsubsection*{Case \ref{separateeigen} The $2$-soliton  $z_{1,2} = iw_{1,2}$}
Here the center of mass evolves according to 
\[
\dot x_0 = w_1^2 + w_2^2 - w_1 w_2 = \Big(\frac{w_1+w_2}2\Big)^2 + \frac34 (w_1-w_2)^2 
\]
whereas in case (i) $\gamma_{00}$ is purely imaginary, and is given by
\[
\dot \gamma_{00}  = -2i w_1 w_2
\]
but in case (ii) it is shifted by $\frac{i\pi}{w_1-w_2}$.
Here we have two different pictures. 

In case (i),
which corresponds to bumps of opposite sign, the solitons cross 
each other and, relative to the center of mass, their centers 
move as in the similar NLS picture in Figure~\ref{figgeneric} 
but with a smaller $a$.

In case (ii), on the other hand, the two solitons approach only to distance 
$\log |w_1-w_2|$, where they exchange the effective spectral parameters, 
and then separate back.

\subsubsection*{Case \ref{breather}  The breather  $z_{1,2} = i w \pm  \rho$}

Here the center of mass evolves according to 
\[
\dot x_0 = w^2 - 3 \rho^2
\]
which is the same as the speed of each of the two corresponding solitons, taken separately.

On the other hand $\gamma_{00}$ is again purely imaginary, and is given by
\[
\dot \gamma_{00}  = -2i( w^2 + \rho^2)
\]
Here this implies that 
\[
(z_1-z_2) \gamma_{00} = -4i \rho(w^2+\rho^2)
\]
which yields the period 
\[
T = \frac{\pi}{2\rho(w^2+\rho^2)}
\]
which is presumably also the time scale on which the breather matches the double soliton.

The path of the soliton relative to the path of the center is periodic. If 
\[ \dist(  \gamma_{00} , i \pi \Z)   \ll |z_1-z_2| \] 
then the two-soliton is close a 2 soliton for the double eigenvalue. In the opposite regime the path is similar to figure \ref{periodic}, relative to the uniform movement of the center.

\bigskip

 \bibliography{nls} 
\bibliographystyle{plain} 
\end{document}